%% file: operations-sa.tex
\documentclass{amsart}
\usepackage[margin=1in]{geometry}
\usepackage{verbatim}
\usepackage[textsize=scriptsize]{todonotes}

\input{preamble-sa.tex}

\myexternaldocument[Colimits:]{colimits-sa}

\title{Normed motivic spectra and power operations}
\begin{document}
\maketitle

\begin{abstract}
We use motivic colimits to construct power operations on the homotopy groups of normed motivic spectra admitting a (normed) map from $H\F_2$.
We establish enough of their standard properties to prove that the motivic dual Steenrod algebra is generated by one element under ring and power operations, establishing a motivic analog of Steinberger's theorem.
\end{abstract}

\tableofcontents

\input{operations.tex}

\bibliographystyle{amsalpha}
\bibliography{powerops}

\end{document}

%% file: preamble-sa.tex
\input{preamble.tex}

\newenvironment{subappendices}{\appendix}{}

\usepackage{xstring}
\usepackage{xr}
\newcommand{\myexternaldocument}[2][] {{
\let\nl\newlabel
\def\nlxx##1##2##3##4##5##6{\nl{##1}{{#1##2}{}{}{}{}}}
\renewcommand\newlabel[2]{\IfBeginWith{##1}{tocindent}{}{\nlxx{##1}##2}}
\externaldocument{#2}
}}

%% file: preamble.tex
\usepackage{tikz}
\usetikzlibrary{matrix,arrows}
\usepackage{tikz-cd}
\usepackage{amsmath}
\usepackage{relsize}
\usepackage{amssymb}
\usepackage{amsthm}
\usepackage{amscd}
\usepackage{enumerate}

\usepackage{bbm}
\usepackage{etoolbox}
\usepackage{mathtools}
\usepackage{chngcntr}

\usepackage[utf8]{inputenc}

\newtoggle{final}
\toggletrue{final}

\numberwithin{equation}{section}

\newtheorem{proposition}[equation]{Proposition}
\newtheorem{corollary}[equation]{Corollary}
\newtheorem{lemma}[equation]{Lemma}
\newtheorem{theorem}[equation]{Theorem}
\newtheorem*{theorem*}{Theorem}
\newtheorem*{corollary*}{Corollary}
\newtheorem*{proposition*}{Proposition}
\newtheorem*{lemma*}{Lemma}
\theoremstyle{definition}
\newtheorem{definition}[equation]{Definition}
\newtheorem{construction}[equation]{Construction}
\newtheorem{assumption}[equation]{Assumption}
\newtheorem*{definition*}{Definition}
\newtheorem*{construction*}{Construction}

\theoremstyle{remark}
\newtheorem{remark}[equation]{Remark}
\newtheorem{example}[equation]{Example}

\newcommand{\triv}{\mathrm{triv}}
\newcommand{\ins}{\mathrm{ins}}
\newcommand{\id}{\operatorname{id}}
\newcommand{\Z}{\mathbb{Z}}
\def\C{\mathbb C}

\newcommand{\F}{\mathbb{F}}
\newcommand{\LEq}{\mathrm{LEq}}
\newcommand{\SSigma}{\mathlarger{\sum}}

\let\scr=\mathcal
\let\bb=\mathbb
\newcommand{\Gm}{{\mathbb{G}_m}}
\newcommand{\Gmp}[1]{{\mathbb{G}_m^{\wedge #1}}}
\def\A{\bb A}

\def\P{\bb P}
\def\DD{\bb D}
\def\R{\bb R}
\newcommand{\1}{\mathbbm{1}}
\newcommand{\hh}{\mathbbm{h}}

\newcommand{\eff}{{\text{eff}}}
\newcommand{\veff}{{\text{veff}}}
\newcommand{\DM}{\mathcal{DM}}
\newcommand{\SH}{\mathcal{SH}}

\newcommand{\Pic}{\scr Pic}
\def\ph{\mathord-}
\newcommand{\Th}{\mathrm{Th}}

\newcommand{\lra}[1]{\langle #1 \rangle}
\newcommand{\sw}{\mathrm{sw}}

\DeclareMathOperator*{\colim}{colim}

\let\lim=\relax
\DeclareMathOperator*{\lim}{lim}
\def\Map{\mathrm{Map}}

\def\imap{\ul{\mathrm{map}}}

\def\Zar{\mathrm{Zar}}

\def\Spec{\mathrm{Spec}}

\def\PTM{\mathrm{PTM}}

\def\CAlg{\mathrm{CAlg}}

\def\NAlg{\mathrm{NAlg}}

\def\HNAlg{\mathrm{HNAlg}}

\def\Spc{\mathcal{S}\mathrm{pc}{}}

\def\Fun{\mathrm{Fun}}

\def\triv{\mathrm{triv}}

\newcommand{\iso}{\cong}
\newcommand{\wequi}{\simeq}
\newcommand{\Mod}{\mathcal{M}\mathrm{od}}

\DeclareRobustCommand{\ul}{\underline}

\newcommand{\Hom}{\operatorname{Hom}}

\def\op{\mathrm{op}}

\let\cat=\mathrm
\def\Sm{{\cat{S}\mathrm{m}}}

\def\et{\mathrm{\acute et}}
\def\pt{*}
\def\MGL{\mathrm{MGL}}

\def\Sq{\mathrm{Sq}}
\def\mot{\mathrm{mot}}
\def\naive{\mathrm{naive}}

\def\Ops{\mathrm{Ops}}

\def\EE{\mathbb{E}}

\def\B{\mathrm{B}}

\def\H{\mathrm{H}}
\def\BB{\mathbb{B}}
\def\tfib{\mathrm{tfib}}
\def\fib{\mathrm{fib}}

\providecommand{\ceil}[1]{\left \lceil #1 \right \rceil }
\providecommand{\floor}[1]{\left \lfloor #1 \right \rfloor }

\newcommand{\limone}{\mathrm{lim}^1}

\newcommand{\fpsr}[1]{[\![ #1 ]\!]}

\newcommand{\TODO}[1]{\todo[color=red]{#1}}
\iftoggle{final} {
\newcommand{\NB}[1]{}
\newcommand{\tombubble}[1]{}
\newcommand{\tom}[1]{}
\newcommand{\elden}[1]{}
\newcommand{\eldenbubble}[1]{}
\newcommand{\jeremiah}[1]{}
\newcommand{\sjeremiah}[1]{}
\renewcommand{\todo}[1]{}
}{ 
\newcommand{\NB}[1]{\todo[color=gray!40]{#1}}
\newcommand{\tombubble}[1]{\todo[color=green!40]{#1}}
\newcommand{\tom}[1]{{\color{green!60!black}#1}}
\newcommand{\elden}[1]{{\color{blue!60!black}#1}}
\newcommand{\eldenbubble}[1]{\todo[color=blue!40]{#1}}
\newcommand{\jeremiah}[1]{{\color{red!60!black}#1}}
\newcommand{\sjeremiah}[1]{{\todo[color=red!40]{#1}}}
}

\author{Tom Bachmann}
\address{Department of Mathematics, University of Oslo, Oslo, Norway}
\email{tom.bachmann@zoho.com}

\author{Elden Elmanto}
\address{Department of Mathematics, Harvard University, Cambridge, MA, USA}
\email{eldenelmanto@gmail.com}

\author{Jeremiah Heller}
\address{Department of Mathematics, University of Illinois, Urbana-Champaign, IL, USA}
\email{jbheller@illinois.edu}

\date{\today}

%% file: operations.tex

\section{Introduction}
Power operations have a long and well established history within classical algebraic topology. They are an essential computational tool for making effective use of generalized homology theories. Essential to the definition of power operations is some amount of extra structure. In order to construct power operations on a spectrum $X$, one requires at minimum a ``coherently commutative binary multiplication'', also known as an $\mathcal{H}_2$ structure. Such structure is encoded as a map $(X^{\wedge 2})_{h\Sigma_2}\to X$ which satisfies some basic properties.   
Given such an $X$, there are operations on its homology which take the form
\[
Q^{i}:H_*(X)\to H_{*+i}(X),
\]
which are sometimes re-indexed as $Q_s$ defined by $Q_s(x):=Q^{s+|x|}(x)$.

Consider the following important classical theorem of Steinberger. 

\begin{theorem}[Steinberger, {\cite[III.2.2-III.2.4]{hinfty}}]\label{thm:steinberger} Let $\mathcal{A}_* = H\mathbb{F}_{2*}H\mathbb{F}_2$ be the dual Steenrod algebra which is a free polynomial algebra \cite{milnor-dual} of the form
	\[
	\mathcal{A}_* = \mathbb{F}_2[\xi_1, \xi_2, \cdots, \xi_j, \cdots ] \qquad |\xi_j| = 2^j-1.
	\]
Applying the power operation $Q_1$ repeatedly to the generator $\xi_1$ yields elements in the dual Steenrod algebra which also generate it freely:
	\begin{equation}\label{dl}
		\mathcal{A}_*  = \mathbb{F}_2[\xi_1, Q_1(\xi_1), (Q_1)^2(\xi_1), \cdots].
	\end{equation}
\end{theorem}

Theorem~\ref{thm:steinberger} is an ``economical" presentation of the dual Steenrod algebra --- it expresses other elements in this very large ring, corresponding to the duals of Steenrod's operations, in terms of $\xi_1$ (which corresponds to the Bockstein operation) and the $Q_1$-operation. The latter comes from a fundamental piece of structure of the spectrum $H\mathbb{F}_2$ --- that it is an $\mathcal{H}_2$-ring spectrum. In particular, it comes equipped with a map
\begin{equation}\label{eq:h2}
(H\F_2^{\wedge 2})_{h\Sigma_2} \rightarrow H\F_2.
\end{equation}
In particular, the $Q_1$-operations are functorial for maps of $\mathcal{H}_2$-ring spectra; this makes it possible to describe the homology of $\mathcal{H}_2$-ring maps out of $\mathbb{F}_2$.

The map~\eqref{eq:h2} is refined by a much richer structure on $H\mathbb{F}_2$, that is of an $\mathcal{E}_2$-algebra. This richer structure, combined with Theorem~\ref{thm:steinberger}, leads to one of the deepest and most far-reaching theorems in modern homotopy theory. This theorem is due to Mahowald \cite{mahowald}, and expresses the Eilenberg-Maclane spectrum $H\mathbb{F}_2$ as a certain Thom spectrum over the space $\Omega^2S^3$. The key input to this result is the observation that, up to the Thom isomorphism, the Dyer-Lashof algebra described by the right-hand-side of~\eqref{dl} is isomorphic to the Dyer-Lashof algebra of $\Omega^2S^3$. When $p$ is odd, the analogous theorem involving $H\mathbb{F}_p$ was later proved by Hopkins and we refer the reader \cite[Appendix A]{kn-lectures} or \cite[Theorem 5.1]{AB:Thom} for a modern version of the proof, which are based on the same ideas as the original. This result has far-reaching consequences, beyond homotopy theory itself. Indeed, one deduces from this B\"okstedt's famous periodicity theorem computing the topological Hochschild homology ($\mathrm{THH}$) of $\mathbb{F}_p$ as a free polynomial $\mathbb{F}_p$-algebra on one generator in degree $2$; see \cite[Theorem 1.1, Appendix A]{kn} for a proof and an extensive discussion on how these results are intimately related. B\"okstedt's periodicity has since been responsible for the heavy influence of homotopy theory on recent advances in $p$-adic Hodge theory \cite{bhatt2019topological, nikolaus-scholze} where $p$-adic cohomology theories have been constructed as associated graded pieces of motivic filtrations on various homotopy-theoretic invariants of $p$-adic schemes.


In the same vein, the Steenrod algebra had previously entered algebraic geometry in another crucial way via the Rost-Voevodsky proofs of Milnor and Bloch-Kato conjectures \cite{voevodsky2003reduced, voevodsky2003motivic,voevodsky2011motivic}. One of the crucial ingredients of the proof is the construction of the motivic analog of Margolis homology, which is a certain chain complex whose differentials are the (motivic analogs of the) Milnor operations. We refer the interested reader to \cite[Section 1.6]{bk-book} for a brief explanation on the role of power operations for motivic cohomology in the proof of the Milnor of Bloch-Kato conjectures, and \cite[Chapter 13]{bk-book} for a more extensive exposition.

Now we turn to our results. 
Let $S$ be a base scheme and $H\Z/2 \in \SH(S)$ be the motivic spectrum representing mod-$2$ motivic cohomology as constructed by Spitzweck. 
Write $\mathcal{A}_{**} = \pi_{**}(H\Z/2\wedge H\Z/2)$ for the dual motivic Steenrod algebra. This is a Hopf algebroid with underlying algebra
\[
\mathcal{A}_{**} \cong H\Z/2_{**}[\tau_0,\tau_1,\ldots, \xi_1,\xi_2,\ldots]/(\tau_i^2-\tau\xi_{i+1} -\rho\tau_{i+1}-\rho\tau_0\xi_{i+1})
\]
where $|\tau_i|= (2^{i+1}-1,2^{i}-1)$ and $|\xi_{i}|=(2(2^i-1), 2^i-1)$; see Appendix~\ref{app:steenrod} for details and the full Hopf algebroid structure.

The bulk of this paper concerns the construction and analysis of motivic power operations on the homotopy of ``quadratically normed" $H\Z/2$-modules. We discuss what this means shortly,  for the moment we simply note that our constructions apply to  $H\Z/2\wedge H\Z/2$. In particular, we obtain power operations acting on $\mathcal{A}_{**}$. Our main result is the following motivic version of Steinberger's theorem. 

\begin{theorem}[see Theorem \ref{thm:bockstein-generates}] \label{thm:mot-steinberger} 
	Let $S$ be a scheme where $\tfrac{1}{2} \in \scr O_S$.  Then:
	\begin{enumerate}
		\item there are \emph{motivic power operations}:
		\[
		Q^i: \mathcal{A}_{**} \rightarrow \mathcal{A}_{*+i, *+\ceil{i/2}} \qquad i \in \mathbb{Z}
		\]
		\item the smallest $H\Z/2_{**}$-subalgebra of $\mathcal{A}_{**}$ containing $\tau_0$ and stable under the operations $Q^i$, is $\mathcal{A}_{**}$ itself.
	\end{enumerate}
\end{theorem}

We note that one of the themes of this project is the indispensible nature of \emph{genuine} structures in motivic homotopy theory.
Indeed, one can easily construct \emph{naive} or \emph{categorical} power operations $Q_{\naive}^i: H\Z/2_{**}H\Z/2 \rightarrow H\Z/2_{*+i,2*}H\Z/2$ which will not interact correctly with the Tate twists, which are so central to the deeper properties of motivic cohomology. In particular, the element $\tau$ and applications of $Q_{\naive}^i$ on it will not be able to generate the dual Steenrod algebra. In classical homotopy theory, the map~\eqref{eq:h2} and its relationship with maps of the same flavor are ultimately responsible for the construction of these power operations. 
\subsection{Motivic power operations} 

The first author together with Marc Hoyois, introduced the theory of normed motivic spectra in \cite{norms}. Given $E\in\SH(S)$ and a finite \'etale morphism $p:T\to S$, a parameterized smash power $E^{\wedge T}$ is constructed in loc.~cit. (where it is denoted $p_\otimes p^*E$).
Now, informally speaking a normed motivic spectrum over $S$ is a motivic spectrum $E\in \SH(S)$ equipped with multiplications $(E|_{X})^{\wedge Y}\to E|_{X}$ parameterized by finite \'etale morphisms $Y\to X$ over $S$, which are required to satisfy coherency conditions. Examples normed motivic spectrum include familiar motivic spectra such as the sphere spectrum, $MGL$, $KGL$, and importantly $H\Z/2$. 

Even in the informal description just given, it is clear that the concept of a normed motivic spectrum encodes an incredible amount of structure. We prefer to rely on as little structure as possible for our construction of power operations. To this end we introduce the category of ``quadratically normed" spectra. In brief a  quadratically normed spectrum is one equipped with multiplications parameterized by finite \'etale morphisms of degree 2. This is formalized as a map $D^\mot_2(X)\to X$, where $D^\mot_2(-)$ is the motivic extended powers functor \cite{colimits}.  

Our construction relies on the interplay of two viewpoints: categorical and equivariant. On the one hand the motivic extended powers functor is a motivic colimit 
$D_2^\mot(E) \simeq \colim_{p:Y\to X}(E|_{X})^{\wedge Y}$. On the other hand we retain computational control via an equivariant model 
$\DD_2(E)\simeq \EE C_{2+}\wedge_{C_2}E^{\wedge \underline{2}}$, again see \cite{colimits}.

We construct a spectrum $\Ops_{H\Z/2}$ which parameterizes power operations on quadratically normed $H\Z/2$-modules. In Theorem~\ref{thm:hpty-of-ops}  we compute that if $S$ is essentially smooth over a Dedekind domain with residue characteristics $\neq 2$, then
\begin{equation*}
	\pi_{**}  \Ops_{H\Z/2} \iso \underset{n \in \Z}{\bigoplus} H\Z/2_{**}\{ e_{2n}\} \oplus \underset{n \in \Z}{\bigoplus} H\Z/2_{**}\{ e_{2n+1}\},
\end{equation*}
where $e_{i}$ are  classes with $|e_{2n}| = (2n, n)$ and $|e_{2n+1}| = (2n+1, n+1)$. 
These give rise to operations for a quadratically normed $H\Z/2$-module $E$ (over any base containing $1/2$),
\[
Q^{i}:\pi_{**}E\to \pi_{*+i,*+\ceil{i/2}}E.
\]

The following proposition summarizes the results of Section~\ref{sub:bp}, establishing the basic properties of these operations.

\begin{proposition} 
	Let $E \in \HNAlg_2(\SH(S)_{H\Z/2/}$ and $x,y\in \pi_{**}E$.
	\begin{description}
		\item[(1) naturality] If $\alpha:E\to F$ is a map of quadratically normed $H\Z/2$-modules, then $Q^i\alpha_* = \alpha_*Q^i$.
		
		\item[(2) base change] Let $f: S' \to S$ be any morphism of schemes. Then, $Q^if^* = f^*Q^i$. 
		
		\item[(3) additivity]  $Q^i(x+y) = Q^i(x) + Q^i(y)$.
		\item[(4) squaring] If $|x| = (i, \ceil{i/2})$ then  $Q^i(x) = x^2$.
		\item[(5) vanishing] If $|x| = (p,q)$ with $\floor{n/2}<p-q$ and $\ceil{n/2}\leq q$, we have $Q^{n}(x) = 0$.
		
		\item[(6) stability] the  square commutes
		\begin{equation*}
			\begin{tikzcd}
				\pi_{m+p+q, n+q}(E) \ar[r, "\cong"]\ar[d, "Q^i"'] & \pi_{m,n} \Omega^{p+q,q} E \ar[d, "Q^{i}"]\\
				\pi_{m+p+q+i, n+q+\ceil{i/2}}(E) \ar[r, "\cong"] & \pi_{m+i,n+\ceil{i/2}} \Omega^{p+q,q} E. 
			\end{tikzcd}
		\end{equation*}
		
		\item[(7) Cartan formula] 
		\[ 
		Q^{2i}(xy) = \sum_{j+k=i} Q^{2j}(x)Q^{2k}(y) + \tau\sum_{j+k=i+1} Q^{2j-1}(x)Q^{2k-1}(y) 
		\] 
		and 
		\[ 
		Q^{2i-1}(xy) = \sum_{j+k=i} (Q^{2j-1}(x)Q^{2k}(y) + Q^{2j}(x)Q^{2k-1}(y)) + \rho\sum_{j+k=i+1} Q^{2j-1}(x)Q^{2k-1}(y). 
		\]
		
	\end{description}
\end{proposition}

Furthermore we establish a variant of the ``co-Nishida relations'', see Theorem~\ref{thm:conishida}, and determine the action on the dual Steenrod algebra, see Theorem~\ref{thm:power-ops-on-dual-steenrod}.

\subsection{Outline of the paper and guide to proof of Theorem~\ref{thm:steinberger}}
We begin in \S\ref{sec:homotopy-normed-spectra} by formalizing our notion of (quadratically) homotopy normed spectrum.
Next in \S\ref{sec:diagonals} we introduce spectra with ``equivariant diagonals'', the most common example being suspension spectra.
One of the main points is that if $E$ has a 2-fold diagonal and $F$ is (quadratically) homotopy normed, then the internal mapping object $\imap(E,F)$ is still (quadratically) normed.
The next two sections contain our main results.
In \S\ref{sec:spectra-of-ops} we define the spectrum of (quadratic) operation and determine its homotopy groups and comodule structure.
These results are used in \S\ref{sec:power-ops} to define power operations and establish their standard properties.
With all of this in hand, the proof of our motivic version of Steinberger's theorem is a straightforward translation of the classical argument.

For convenience of the reader, we provide the following chronological guide to how one derives Theorem~\ref{thm:steinberger}; it also highlights the key role that the Nishida relations play:
\begin{enumerate}
\item the first thing to do is assemble to \emph{spectra of operations} $\Ops_{H\Z/2}$ which parametrizes these $Q^i$-operations and compute its homotopy groups as done in Theorem~\ref{thm:hpty-of-ops}.
\item Next, one establishes the coaction of the motivic Steenrod algebra $\scr A_{**}$ on the homotopy of the spectra of operations $\pi_{**} \Ops_{H\Z/2}$ as discussed on \S\ref{sec:coact}. 
\item We then establish two \emph{coherence results} concerning this coaction: 1) a comparison between the \emph{left} co-action with the \emph{right} co-action on a more explicit, geometric model for the spectra of operations as expressed in  Theorem~\ref{thm:homotopy-of-ops}. This results in the formulas for the left coaction is Corollary~\ref{cor:psil}. And 2) what we call, following Wilson \cite{wilson2017power}, the \emph{co-Nishida relations} in Corollary~\ref{cor:cohere} which explains how the $Q^i$ operation and the left coaction interacts.
\item Both coherence results crystalizes into the very explicit identity furnished in Theorem~\ref{thm:conishida} which is, at heart, the main result of the present paper. 
\item From this, one can read off formulas like Corollary~\ref{cor:power-ops-on-dual-steenrod} which proves Theorem~\ref{thm:steinberger}. While reading off these formulas can be involved, we emphasize that there are no further geometric or homotopical ideas involved; see the first proof of Theorem~\ref{thm:bockstein-generates} which only proves the claimed generation statement. It is also possible to give formulas for the action of the operations on $\tau_0$ as in Corollary~\ref{cor:power-ops-on-dual-steenrod} which gives a more explicit proof of the claimed theorem.
\end{enumerate}

We conclude the article with a sequence of somewhat technical appendices.
In \S\ref{app:dualization} we recall that over appropriate bases, certain infinite sums of Tate motives are reflexive, and use this to relate their homology and cohomology.
In \S\ref{sec:8fold} we recall standard formalism regarding the algebraic structures (Hopf algebroids, comodules, etc.) seen on the homotopy and homology groups of ring spectra and their modules.
Next in \S\ref{sec:oriented-spectra} we recall some facts about oriented motivic spectra, and in \S\ref{sec:mot-coh} we recall facts about Spitzweck's motivic cohomology spectrum.
Finally in \S\ref{app:excisive-functors} we show that given a $2$-excisive and reduced functor $F$, its values $F(X)$ and $F(\Sigma X)$ are related by a fiber sequence.

\subsection{Acknowledgements}
During the very long gestation period of this work, many people have provided helpful comments, among them David Gepner, Jeremy Hahn, John Rognes, and Paul Arne \O stv\ae r. Our special thanks goes to Dylan Wilson for answering our torrent of questions about classical and equivariant power operations. Heller was partially supported by the NSF under award DMS-1710966.

\subsection{Notation and conventions}

All schemes are assumed to be quasi-compact and quasi-separated (qcqs), $\Sm_S$ denotes the category of finitely presented smooth $S$-schemes. 
We denote antipodes in various algebraic structures by either $\chi$ or overlines, whichever is more convenient.
We freely use the notation from \cite{colimits}.
References thereto are displayed as in e.g. \S\ref{sec:motivic-colim-first-props}.

\section{Homotopy normed spectra}\label{sec:homotopy-normed-spectra}

We now introduce a category of minimally structured normed spectra, which is an analog of the category of $\mathcal{H}_2$-ring spectra in classical homotopy theory.

\begin{definition}
The category of \emph{(quadratically) homotopy normed spectra} $\HNAlg_2(\SH(S))$ is the 1-category with objects consisting of pairs $(E, m_E)$ with $E \in h(\SH(S))$ and $m_E: D_2^\mot(E) \to E$ a map in $h(\SH(S))$. The morphisms $(E, m_E) \to (F, m_F)$ are given by maps $\alpha: E \to F$ in $h(\SH(S))$ such that the following diagram commutes
\begin{equation*}
\begin{CD}
D_2^\mot(E) @>{D_2^\mot(\alpha)}>> D_2^\mot(F) \\
@Vm_EVV                            @Vm_FVV     \\
E           @>\alpha>>             F.
\end{CD}
\end{equation*}

Equivalently $\HNAlg_2(\SH(S))$ is the lax equalizer in (the $(2,1)$-category of) $1$-categories
\[
\LEq(D_2^\mot, \id): h(\SH(S)) \rightrightarrows h(\SH(S)).
\]
In other words, it fits into following pullback square
\begin{equation} \label{eq:lax-pull}
\begin{tikzcd}
\HNAlg_2(\SH(S)) \ar{r} \ar{d} & h(\SH(S))^{\Delta^1} \ar{d}{(\partial_0, \partial_1)} \\
h(\SH(S)) \ar{r}{(D_2^\mot, \id)} & h(\SH(S)) \times h(\SH(S)).
\end{tikzcd}
\end{equation}
\end{definition}

\begin{remark}\label{rmk:HNAlg2-modules}
The functor $D_2^\mot: \SH(S) \to \SH(S)$, and hence the functor on homotopy categories $h(D_2^\mot): h(\SH(S)) \to h(\SH(S))$, can be given the structure of an op-lax symmetric monoidal functor (see Proposition \ref{prop:Dmot-oplax-monoidal}). This can be used to promote $\HNAlg_2(\SH(S))$ to a symmetric monoidal category; see \cite[Construction IV.2.1]{nikolaus-scholze} for details.
\end{remark}

\begin{remark}There is a forgetful functor from $\HNAlg_2(\SH(S))$ to the category of non-unital commutative monoids in $h(\SH(S))$; given $(E,m_E) \in \HNAlg_2(\SH(S))$ the composite $m_E\circ \alpha_E$ defines a multiplication $E \wedge E \to E$  where $\alpha_E: E \wedge E \to D_2^\mot(E)$ is the canonical map (see Definition \ref{def:op-lax-witness}). In particular, given a map $E \to F \in \HNAlg_2(\SH(S))$, then $F$ canonically becomes an $E$-module (in the homotopy category).
\end{remark}

\begin{construction}
We have a forgetful functor \[ U_2: \NAlg(S) \to \HNAlg_2(\SH(S)), E \mapsto (UE, m_E: D_2(UE) \to UE) \] obtained as follows.
Note that the adjunction $U:\SH(S) \rightleftarrows \NAlg(S):N$ is monadic \cite[Proposition 7.6.2]{norms} with the free functor computed by $E \mapsto \bigvee D_n^\mot E$.
The map $m_E$ is obtained by restricting the structure map $\bigvee D_n^\mot E \rightarrow E$ to the $n=2$ summand.
\end{construction}

\begin{example}
By Proposition \ref{prop:HZ-higher}(1), for every $n \ge 0$ we obtain $\H\Z/n \in \NAlg(\Spec(\Z))$.
Hence we obtain $\H\Z/n \in \HNAlg_2(\Spec(\Z))$, and thus for every scheme $S$ we obtain $\H\Z/n \in \HNAlg_2(\SH(S))$, by Example \ref{ex:HNAlg2-base-change}.
\end{example}

\begin{example} \label{ex:HNAlg2-base-change}
Here is a simple procedure to produce new examples out of old ones. Let $f: S' \to S$ be an arbitrary morphism of schemes, and $(E, m_E) \in \HNAlg_2(\SH(S))$. Since motivic extended powers commute with base change by Example \ref{ex:Dn-stable-base-change-specific}, we obtain a map $m_{f^* E}: D_2^\mot(f^* E) \wequi f^* D_2^\mot(E) \xrightarrow{f^* m_E} f^* E$. One checks easily that this defines a functor $f^*: \HNAlg_2(\SH(S)) \to \HNAlg_2(\SH(S'))$.
\end{example}

For a further example, see Construction \ref{cons:map-HNAlg} below. The following result will eventually lead to the Cartan formula for power operations.
\begin{lemma}[Ur-Cartan formula] \label{lemm:ur-cartan-formula}
Let $E \in \HNAlg_2(\SH(S))$ and $A, B \in \SH(S)$. Given maps $a: A \to E$ and $b: B \to E$ in $\SH(S)$, we
denote by $ab: A \wedge B \to E \wedge E \to E$, $D_2(a): D_2^\mot(A) \to D_2^\mot(E) \to E$, and so on, the maps obtained by using multiplication in $E$. Then the following diagram in $\SH(S)$ commutes (up to homotopy)
\begin{equation*}
\begin{tikzcd}
D_2^\mot(A \wedge B) \ar[r, "D_2(ab)"] \ar[d, "\beta_{A,B}" swap] & E \\
D_2^\mot(A) \wedge D_2^\mot(B) \ar[ru, "D_2(a)D_2(b)" swap].
\end{tikzcd}
\end{equation*}
Here $\beta_{A,B}$ denotes the op-lax witness transformation from Definition \ref{def:op-lax-witness}.
\end{lemma}
\begin{proof}
Consider the diagram
%
\[
\begin{tikzcd}[column sep=6em]
D_2^\mot(A \wedge B) \ar[r, "D_2(a \wedge b)"] \ar[d, "\beta_{A,B}"'] & D_2^{\mot}(E \wedge E) \ar[r, "="] \ar[d, "\beta_{E, E}"] &  D_2^\mot(E \wedge E) \ar[d, "D_2(\alpha_E)"]   \\
D_2^\mot(A) \wedge D_2^\mot(B) \ar[r, "D_2(a) \wedge D_2(b)"] & D_2^\mot(E) \wedge D_2^\mot(E) \ar[r, "\alpha_{D_2(E)}"] \ar[d, "m_E \wedge m_E"] & D_2^\mot(D_2^\mot(E)) \ar[d, "D_2(m_E)"] \\
&  E \wedge E                 \ar[r, "\alpha_E"] &    D_2^\mot(E) \ar[d, "m_e"]          \\        
& &  E.
\end{tikzcd}
\]
The composite from top left to bottom right via top right is $D_2(ab)$, whereas the composite along the bottom is $D_2(a)D_2(b) \circ \beta_{A,B}$.
It remains to show that the diagram commutes. The top left square commutes by naturally of the op-lax witness transformation. The top right square commutes by Theorem \ref{thm:ur-cartan-reln}. The bottom right square commutes by naturality of $\alpha$ (see Remark \ref{rmk:alpha-naturality}).
\end{proof}

\section{Diagonal maps} \label{sec:diagonals}
The spectrum parametrizing operations is constructed by taking the limit of a certain diagram. In this section, we construct the transition maps. 

\begin{definition} \label{def:spectra-w-diagonal} A \emph{motivic spectrum with an $n$-fold diagonal} is a motivic spectrum $E \in \SH(S)$ equipped with a map
\begin{equation} \label{eq:diagonal}
\Delta_n: E^\triv \rightarrow E^{\wedge \ul{n}} \in \SH^{\Sigma_n}(S),
\end{equation}
where $E^{\wedge \ul{n}}$ denotes the functor from \S\ref{sec:enhanced-smash-powers}.
\end{definition}
When no confusion can arise, we will write $E \in \SH^{\Sigma_n}(S)$ instead of $E^\triv$.

\begin{example}
The construction from Definition \ref{def:diagonal-transform} provides the suspension spectrum of a motivic space $X \in \Spc(S)_*$ with the structure of a motivic spectrum with $n$-fold diagonals for all $n$.
\end{example}


The transition maps for the spectrum parametrizing operations will be obtained from the following.
\begin{construction} \label{constr:enh-diag} Fix $n \geq 0$, and let $F \in \SH(S)$ be a motivic spectrum with a $n$-fold diagonal $\Delta_n: F \rightarrow F^{\wedge \ul{n}}$. Let $E \in \SH(S)$. We have a map
\[
F^\triv \wedge E^{\wedge \ul{n}} \stackrel{\Delta_n \wedge \id}{\rightarrow} F^{\wedge \ul{n}}\wedge E^{\wedge \ul{n}} 
\]
in $\SH^{\Sigma_n}(S)$. Applying the geometric homotopy orbits functor of Definition~\ref{def:gho} and using Lemma~\ref{lemm:triv-quot-projection-formula} we obtain a map in $\SH(S)$
\[
\Delta'_{F, E}:F \wedge \DD_n(E) \wequi (F^\triv \wedge E^{\wedge \ul n})_{\hh \Sigma_n} \rightarrow \DD_n( F \wedge E).
\]
By the natural equivalence of Corollary~\ref{cor:d-vs-d}, we obtain a transformation
\begin{equation} \label{conorm-general}
\Delta_{F, -}: F \wedge D_n^{\mot}(\ph) \Rightarrow D_n^{\mot}(F \wedge (\ph)): \SH(S) \rightarrow \SH(S),
\end{equation}
and a transformation
\begin{equation} \label{conorm}
\Delta_{(-,-)}: \Sigma^{\infty}(\ph) \wedge D_n^{\mot}(\ph) \Rightarrow D_n^{\mot}( \Sigma^{\infty}(\ph) \wedge (\ph)): \Spc(S)_* \times \SH(S) \rightarrow \SH(S).
\end{equation}
\end{construction}

\begin{definition} \label{def:conorm} Let $F \in \SH(S)$ be a motivic spectrum with a $n$-fold diagonal $\Delta_n: F \rightarrow F^{\wedge \ul{n}}$. The transformation of endofunctors $\Delta_{F,-}$ of \eqref{conorm-general} is called the \emph{$n$-fold conormed transformation of $F$}.
When $n=2$ we simply call it the \emph{conormed transformation of $F$}.
\end{definition}

The conormed transformation lets us construct the structure of homotopy normed spectrum on certain mapping spectra. 

\begin{construction} \label{cons:map-HNAlg}
Let $E \in \HNAlg_2(\SH(S))$ and let $F \in \SH(S)$ be a motivic spectrum with a $2$-fold diagonal $\Delta_2: F \rightarrow F^{\wedge \ul{2}}$ (for example, the suspension spectrum of a motivic space).
Then we promote $\imap(F, E)$ to an object of $\HNAlg_2(\SH(S))$.
More precisely, we factor the internal mapping spectrum functor as
\[
\begin{tikzcd}
 & \HNAlg_2(\SH(S)) \ar{d} \\
\HNAlg_2(\SH(S)) \ar[r, "{\imap(F, U\ph)}" swap]  \ar[dashed]{ur}& \SH(S).
\end{tikzcd}
\]
To do so we define $m: D_2^\mot(\imap(F, E)) \to \imap(F, E)$ as \begin{gather*} D_2^\mot(\imap(F, E)) \xrightarrow{u} \imap(F, F \wedge D_2^\mot(\imap(F, E))) \to \imap(F, D_2^\mot(F \wedge \imap(F, E))) \\ \xrightarrow{c}  \imap(F, D_2^\mot(E)) \xrightarrow{m} \imap(F, E), \end{gather*} where the unlabelled map in the middle is induced by the conormed transformation from Construction~\ref{constr:enh-diag}.
\end{construction}
\begin{remark}\label{rmk:mapping-action-conormed}
If $E \in \HNAlg_2(\SH(S))_{R/}$ for some $R \in \HNAlg_2(\SH(S))$, it need not be the case that $\imap(F, E) \in \HNAlg_2(\SH(S))_{R/}$.
In fact, to specify a map $R \to \imap(F,E)$ we need to give a map $R \wedge F \to E$, which seems difficult to produce without a $0$-fold diagonal $F \to \1$.

However, $\imap(F, E)$ is still a module under the non-unital ring $R$, i.e. there is an action $R \wedge \imap(F, E) \to \imap(F, E)$.
More precisely, $\imap(F,E)$ promotes to an object of \[ \Mod_R^w(\HNAlg_2(\SH(S))) := \HNAlg_2(\SH(S)) \times_{h\SH(S)} \Mod_R(h\SH(S)). \]
\end{remark}

The following coherence result will be used when defining the spectra of operations in \S\ref{sec:def-spectra-of-ops}.
\begin{lemma} \label{lem:cohere-diag} Suppose that $\scr X, \scr Y \in \Spc(S)_\pt$ and $\scr F \in \SH(S)$. Then the following diagram commutes
\begin{equation*}
\begin{tikzcd} 
\scr X \wedge \scr Y \wedge D_2^{\mot}(\scr F) \ar{r}{\id \wedge \Delta_{\scr Y, \scr F}} \ar{d}{\tau \wedge \id} &  \scr X \wedge D_2^{\mot}(\scr Y \wedge \scr F) \ar{d}{\Delta_{\scr X, \scr Y \scr X}}\\\
\scr Y \wedge \scr X \wedge D_2^{\mot}(\scr F) \ar{d}{\id \wedge \Delta_{\scr X, \scr F}} & D_2^{\mot}(\scr X \wedge \scr Y \wedge \scr F) \ar{d}{D_2^{\mot}(\tau \wedge \id)}\\
\scr Y \wedge D_2^{\mot}(\scr X \wedge \scr F)\ar{r}{\Delta_{\scr Y, \scr X \wedge \scr Y}} & D_2^{\mot}(\scr Y \wedge \scr X \wedge \scr F),
\end{tikzcd}
\end{equation*}
where $\tau$ denotes the switch automorphism.
\end{lemma}

\begin{proof} We expand the diagram as
\begin{equation*}
\begin{tikzcd} 
\scr X \wedge \scr Y \wedge D_2^{\mot}(\scr F) \ar{rr }{\id \wedge \Delta_{\scr Y, \scr F}} \ar{d}{\tau \wedge \id} \ar{drr}{\Delta_{\scr X \wedge \scr Y, \scr F}} & &  \scr X \wedge D_2^{\mot}(\scr Y \wedge \scr F) \ar{d}{\Delta_{\scr X, \scr Y \scr X}}\\\
\scr Y \wedge \scr X \wedge D_2^{\mot}(\scr F) \ar{d}{\id \wedge \Delta_{\scr X, \scr F}} \ar{drr}{\Delta_{\scr Y \wedge \scr X, \scr F}} & &  D_2^{\mot}(\scr X \wedge \scr Y \wedge \scr F) \ar{d}{D_2^{\mot}(\tau)}\\
\scr Y \wedge D_2^{\mot}(\scr X \wedge \scr F)\ar{rr}{\Delta_{\scr Y, \scr X \wedge \scr Y}} & &  D_2^{\mot}(\scr Y \wedge \scr X \wedge \scr F)
\end{tikzcd}
\end{equation*}
The functoriality of Construction~\ref{constr:enh-diag} gives us the commutativity of the middle parallelogram. The commutativity of the upper triangle follows from the commutativity of the following diagram in $\SH^{\Sigma_n}(S)$
\[
\begin{tikzcd}
\scr X \wedge \scr Y \wedge \scr F^{\wedge \ul{n}} \ar{r} \ar{dr} & \scr X \wedge \scr Y^{\wedge \ul{n}}\wedge \scr F^{\wedge \ul{n}} \ar{d}\\
& (\scr X \wedge \scr Y)^{\wedge \ul{n}} \wedge \scr F^{\wedge \ul{n}}
\end{tikzcd}
\]
and the fact that the equivalence $\DD_2 \simeq D_2^{\mot}$ is functorial. The commutativity of the bottom triangle follows similarly.
\end{proof}

\section{Spectra of operations} \label{sec:spectra-of-ops}

\subsection{Definition of the spectra of operations} \label{sec:def-spectra-of-ops}


We are now ready to define the spectra of operations. We shall construct a functor \[ C: \Z \times \Z \to \SH(S), \] which takes the following value on objects \[ C(m, n) = S^{-m} \wedge \Gmp{-n} \wedge D_2^\mot(S^m \wedge \Gmp{n}) = \Sigma^{-m-n,-n}D_2^\mot(S^{n+m, n}). \] Using \cite[Tag 00J4]{kerodon} we may replace the source $\infty$-category $\Z \times \Z$ by the simplicial set corresponding to the graph with vertices $\Z \times \Z$ and a unique edge from $(m,n)$ to $(m+1,n)$ and $(m,n+1)$. That is, in order to construct the functor, we need to give objects $C(m,n)$, morphisms $C(m,n) \to C(m+1,n)$ and $C(m,n) \to C(m,n+1)$, and homotopies making the square with corners $(m,n), (m+1,n), (m, n+1), (m+1,n+1)$ commute. In other words, we have to construct a commutative diagram
\[
\begin{tikzcd}
C(m, n) \ar{d} \ar{r}  & C(m+1, n) \ar{d}\\
C(m, n+1) \ar{r}\ar[ur, Leftrightarrow] & C(m+1, n+1).
\end{tikzcd}
\]
We define the squares above to be $S^{-m-1} \wedge \Gmp{-n-1} \wedge D_{m,n}$, where $D_{m,n}$ is the commutative square obtained from Lemma \ref{lem:cohere-diag} by putting $\scr X = S^1, \scr Y = \Gm$ and $\scr F = S^{-m} \wedge \Gmp{-n}$.


\begin{definition} \label{def:spectrum-of-ops}
Let $R \in \SH(S)$. We define the \emph{spectrum of $R$-quadratic operations} as
\begin{align*}
  \Ops_R &= \lim_{\Z \times \Z} (C \wedge R) \\
\end{align*}
For $(a, b) \in \Z \times \Z$ we have the canonical projection map 
\[\pi_{a, b}: \Ops_R \rightarrow C(a, b) \wedge R.  \]
\end{definition}

\begin{remark} \label{rmk:Ops-explicit}
By considering an appropriate final subcategory, the above homotopy limit can be made more concrete. Indeed
\[ \Ops_R \wequi \lim R \wedge (\dots \to C(-2, -2) \to C(-1, -1) \to C(0, 0)) \]
where $C(-n, -n) = T^n \wedge D_2^\mot(T^{-n})$.
\end{remark}

\subsection{The homotopy groups of spectra of operations}
\label{subsec:homotopy-groups-spectra-ops}

The various spectra of operations defined in \S\ref{sec:def-spectra-of-ops} parametrize power operations that we shall define in \S\ref{def-ops}. We will now determine the homotopy groups of the spectrum $\Ops_{H\Z/2},$ at least over appropriate bases.

When $R = H\Z/2$, Proposition~\ref{prop:homology-of-dmot} in conjunction with Proposition~\ref{prop:bmu} gives us that
\begin{gather*}
H\Z/2 \wedge D_2^{\mot}(T^{\wedge n}) \simeq  \bigoplus_{p_{\geq 0}}\Sigma^{4n,2n} H\Z/2\{e_{2p}\} \oplus \bigoplus_{p_{\geq 0}}\Sigma^{4n,2n}H\Z/2\{e_{2p+1}\},
\\\text{ where } |e_{2p}| = (2p, p) \text{ and } |e_{2p+1}| = (2p+1,p+1).
\end{gather*}

Now we have the ingredients to obtain a computation of the homotopy groups of the spectrum of operations. 

\begin{theorem} \label{thm:hpty-of-ops} Suppose that $\frac{1}{2} \in \scr O_S$, where $S$ is essentially smooth over a Dedekind domain. Then there is an isomorphism of $H\Z/2_{**}(S)$-modules
\begin{equation*}
\pi_{**}  \Ops_{H\Z/2} \iso \underset{n \in \Z}{\bigoplus} H\Z/2_{**}\{ e_{2n}\} \oplus \underset{n \in \Z}{\bigoplus} H\Z/2_{**}\{ e_{2n+1}\},
\end{equation*}
where $e_{i}$ are certain canonical classes with $|e_{2n}| = (2n, n)$ and $|e_{2n+1}| = (2n+1, n+1)$
\end{theorem}

\begin{remark} \label{rmk:Ops-HZ2-projection}
It follows from the proof that under the projection \[ \Ops_{H\Z/2} \to T^n \wedge D_2(T^{-n}) \wedge H\Z/2 \wequi T^{-n} \wedge \B_\et \Sigma_2 \wedge H\Z/2 \] the class $e_{2r + \epsilon}$ is sent to $T^{-n} \wedge e_{2(r+n) + \epsilon}$ provided that $r+n \ge 0$, and else it is sent to zero.
\end{remark}

\begin{proof} Using the Milnor exact sequence \cite[Proposition VI.2.15]{goerss2009simplicial} we obtain
\begin{equation*}
0 \rightarrow \mathrm{lim}^1 \pi_{*+1*}C(-n,-n) \rightarrow \pi_{**}\Ops_{H\Z/2} \rightarrow \lim \pi_{**}C(-n,-n) \rightarrow 0.
\end{equation*}
From Proposition~\ref{prop:homology-of-dmot}, for each $n \in \Z$ we have a string of equivalences
\begin{eqnarray*}
H\Z/2 \wedge C(-n,-n) & = &  H\Z/2 \wedge \Sigma^{2n,n} D_2^{\mot}(S^{-2n,-n}) \\
& \simeq & H\Z/2 \wedge \Sigma^{2n,n}\Sigma^{-4n,-2n}BC_{2+}\\
 & \simeq & H\Z/2 \wedge \Sigma^{-2n,-n}BC_{2+}\\
& \simeq & \Sigma^{-2n, -n}((H\Z/2 \wedge \P_+^{\infty}) \oplus (H\Z/2 \wedge \P_+^{\infty}(1)[1]) )\\
& \simeq & \underset{j \geq -n}{\bigoplus} \Sigma^{2j, j} H\Z/2 \oplus \underset{j \geq -n}{\bigoplus} \Sigma^{2j+1, j+1}H\Z/2. 
\end{eqnarray*}
On homotopy groups, we thus have the following isomorphism
\begin{equation*}
H\Z/2_{**}C(-n,-n) \cong \underset{j \geq -n}{\bigoplus} H\Z/2_{*, *}\{ v^{(n)}_{j} \} \oplus \underset{j \geq -n}{\bigoplus} H\Z/2_{*, *}\{ u^{(n)}_{j} \} 
\end{equation*}
where $|v^{(n)}_j| = (2j, j)$ and $|u^{(n)}_{j}| = (2j+1, j+1)$. With this notation, Lemma~\ref{lem:action} below tells us the system is Mittag--Leffler and we have an isomorphism
\[
\pi_{**}\Ops_{H\Z/2} \cong \lim_n \left(\underset{j \geq -n}{\bigoplus} H\Z/2_{*, *}\{ v^{(n)}_{j} \} \oplus \underset{k \geq -n}{\bigoplus} H\Z/2_{*, *}\{ u^{(n)}_{j} \} \right).
\]
It remains to prove that the map
\[
 \underset{n \in \Z}{\bigoplus} \Sigma^{2n,n}H\Z/2_{**}\{ e_{2n}\} \oplus \underset{n \in \Z}{\bigoplus} \Sigma^{2n+1,n+1} H\Z/2_{**}\{ e_{2n+1}\}  \to \lim_n \left(\underset{j \geq -n}{\bigoplus} H\Z/2_{*, *}\{ v^{(n)}_{j} \} \oplus \underset{j \geq -n}{\bigoplus} H\Z/2_{*, *}\{ u^{(n)}_{j} \} \right)
\]
where $e_{2n} \mapsto v^{(j)}_n$, $e_{2n+1} \mapsto u^{(j)}_n$ (defined as soon as $n \geq -j$) is an isomorphism. Note that the right hand side is canonically isomorphic to \[ \prod_n H\Z/2_{*, *}\{ v^{(n)}_{-n} \} \oplus \prod_n H\Z/2_{*, *}\{ u^{(n)}_{-n} \}.  \]
Our assumption on $S$ together with Lemma \ref{lemm:tate-sum-product} implies that we may replace the direct product by a direct sum. The result follows.
\end{proof}

\begin{lemma} \label{lem:diagonal-euler}
The diagonal map \[ \Th(\A^1) \to \Th(\A^\rho) \wequi \Th(\A^1) \wedge \Th(\A^\sigma) \in \Spc^{C_2}(S)_* \] is homotopic to $\id_T \wedge e(\sigma)$.
\end{lemma}
\begin{proof}
Since $1/2 \in \scr O_S$ we have $\A^\rho \wequi \A^1 \times \A^\sigma$, and under this decomposition $\Delta: \A^1 \to \A^\rho$ corresponds to the map $\id \times 0: \A^1 \times 0 \to \A^1 \times \A^\sigma$.
\end{proof}

\begin{lemma} \label{lem:action} In the notation of the proof of Theorem~\ref{thm:hpty-of-ops}, the map \[ \phi_{n}: H\Z/2_{**}C(-n,-n) \to H\Z/2_{**}C(-n+1,-n+1) \] is specified by \[ v^{(n)}_j \mapsto v^{(n-1)}_j \quad\text{and}\quad u^{(n)}_j \mapsto u^{(n-1)}_j. \]
\end{lemma}
\begin{proof} We note that $v^{(n)}_j$ (resp. $u^{(n)}_{j}$) are shifts of $v^{(0)}_{j}$ (resp. of $u^{(0)}_{j}$) under the isomorphims induced by the equivalence of Proposition~\ref{prop:homology-of-dmot}. In other words, we have an identification
\begin{equation} \label{eq:u's}
v^{(n)}_j \leftrightarrow \Sigma^{-4n,-2n}\Sigma^{2n, n}v^{(0)}_j \quad(\text{resp. } u^{(n)}_s \leftrightarrow \Sigma^{-4n,-2n}\Sigma^{2n,n}u^{(0)}_s).
\end{equation}
With this notation, we give the argument for $v_j^{(n)}$, the case of $u_j^{(n)}$ being identical.

After the identification of Corollary~\ref{cor:d-vs-d}, the transition map
\[
C(-n, -n) \rightarrow C(-n+1, -n+1),
\]
can be computed as the map
\[
\phi_n: T^{n} \wedge \DD_n(T^{-n}) \rightarrow T^{n-1} \wedge \DD_n(T^{-(n-1)}),
\]
described as in Construction~\ref{constr:enh-diag}.
We claim that the induced map on motivic cohomology of $BC_{2+}$ is given by
\[
H\Z/2^{**}(BC_{2}) \stackrel{\cup c_1(\scr L)}{\rightarrow}H\Z/2^{*+2,*+1}(BC_{2}).
\]
This follows from Lemma \ref{lem:diagonal-euler} (and Lemma \ref{lem:equiv-thom}).

We now dualize to obtain the effect on homology. The following diagram commutes
\[
\begin{tikzcd}
H\Z/2_{**}(B_{\et}C_2) \ar[swap, "\phi_{n}"]{d} \ar{r} &  \Hom( H\Z/2^{**}( B_{\et}C_2), H\Z/2_{**}) \ar{d}{\cup c_1(\mathcal{L}))^*}\\
H\Z/2_{*-2,*-1}(B_{\et}C_2) \ar{r} & \Hom( H\Z/2^{*-2,*-1}(B_{\et}C_2), H\Z/2_{**}).
\end{tikzcd}
\]
By Example \ref{ex:dualization-dedekind} and Lemma \ref{lemm:mot-hom-cohom-dual}(2), the horizontal maps are isomorphisms.
It follows that $\phi_{n**}$ sends the dual of $c_1(\mathcal{L})^i$ to the dual of $c_1(\mathcal{L})^{i-1}$. Under the identifications of \eqref{eq:u's}, we see that $\phi_{n**}$ corresponds to the claimed action.
\end{proof}
 
\subsection{The Steenrod comodule structure of $\Ops_{H\Z/2}$} Next our goal is to prove Theorem~\ref{thm:homotopy-of-ops} which gives us access to the completed co-action of the Steenrod algebra on $\Ops_{H\Z/2}$ (see Corollary~\ref{cor:psil} for the formulas) and thus the co-Nishida relations, which we discuss in~\S\ref{sect:co-nishida}. The key point is that we want to compute the left co-action on $\Ops_{H\Z/2}$ via the right co-action on the mapping spectrum $\imap( B_{\et}\Sigma_{2+} , H\Z/2)[v^{-1}]$ which is geometrically accessible by an explicit model for the classifying stack.  We require some generalities on comodule actions which can be found in \S\ref{sec:comodule-structures}.

\subsubsection{Coaction of $\scr A_{**}$ on the $\pi_{**} \Ops_{H\Z/2}$}\label{sec:coact} Note that the hypotheses of Lemma~\ref{lem:flat-ef} hold for $E = F = H\Z/2$ as reviewed in Theorem~\ref{thm:HZ-dual-steenrod}, so $H\Z/2 \wedge H\Z/2$ is a flat $H\Z/2$-module in the sense of Definition \ref{defn:t-flat}. Hence we may freely speak of left and right comodule actions.

In the language of \S\ref{sec:completed-homology-coaction}, $\pi_{**}\Ops_{H\Z/2}$ is an instance of a completed homology.
In the course of the proof of Theorem \ref{thm:hpty-of-ops} we established a $\limone$ vanishing which implies that Construction \ref{construct-completed-h-coaction} applies.
In other words we get coactions \[ \psi_R: \pi_{**}\Ops_{H\Z/2} \to \pi_{**}\Ops_{H\Z/2} \widehat\otimes_{H\Z/2_{**}} \scr A_{**} \quad\text{and}\quad \psi_L: \pi_{**}\Ops_{H\Z/2} \to \scr A_{**} \widehat\otimes_{H\Z/2_{**}} \pi_{**}\Ops_{H\Z/2}. \]

The two structures are related by the compatible equalities of Lemma~\ref{lem:antipode}, namely \[ \psi^i_L = \psi^i_R \circ \sw_{H\Z/2,X_i*} \circ \chi, \] giving us an equality
\begin{equation} \label{eq:left-to-right}
  \psi_L = \psi_R \circ \sw_{H\Z/2,X*} \circ \chi.
\end{equation}


\subsubsection{The Steenrod comodule structure on mapping spectra} In order to access the completed coaction of the dual Steenrod algebra on $\Ops_{H\Z/2}$ we will make use of an auxilliary comodule $\imap(B_\et C_{2+}, H\Z/2)$; the homotopy groups of this spectrum are simply the motivic cohomology of $B_{\et}C_2.$ This is an instance of a general construction, recalled in \S\ref{subsub:cohomology-coaction}: whenever $X \in \SH(S)$ is a filtered colimit of dualizable objects, $H\Z/2^{**} X$ acquires a canonical completed comodule structure.\footnote{The required $\limone$ vanishing is easily verified.}

In our situation, the motivic spectrum $\Sigma^{\infty} B_{\et}C_{2+}$ is a colimit of dualizable objects in $\SH(S)$. Indeed, this is true for $\Sigma^{\infty}B_{\et}\mu_{n+}$ for any $n \in \mathbb{N}$: it is a colimit of the (suspension spectra of) punctured vector bundles $\scr O_{\P^k}(-n) \setminus 0 \rightarrow \P^k,$ as $k \rightarrow \infty.$ In turn, the suspension spectrum of each punctured vector bundle sits in a cofiber sequence 
\[
\Sigma^{\infty}_+ \scr O_{\P^k}(-n) \rightarrow \Sigma^{\infty}_+ \P^k \rightarrow \Sigma^{\infty}\Th_{\P^k}(\scr O(-n))
\]
where the last two terms are dualizable (e.g. use \cite[Proposition 2.4.31]{cisinski-deglise} and \cite[Proposition 2.17]{A1-homotopy-theory}), whence $\Sigma^{\infty}_+ \scr O_{\P^k}(-n)$ itself is dualizable.

We shall prove the following.
\begin{theorem} \label{thm:homotopy-of-ops} \todo{Why not avoid the twist by using left actions on both sides?}
Let $S$ be smooth over a Dedekind domain (or field) of residue characteristics $\ne 2$ and denote by $v \in H^{2,1}(B_{\et}C_2; \Z/2)$ the Euler class of the line bundle $\scr L(\sigma)$ (as in Definition~\ref{def:char-classes}). There is a canonical map $t: \imap(\Sigma^\infty_+ B_\et \Sigma_2, H\Z/2) \to \Sigma \Ops_{H\Z/2} \in \SH(S)$, with the following properties:
\begin{enumerate}
\item It induces a canonical equivalence
\[t[v^{-1}]:\imap(\Sigma^\infty_+ B_\et \Sigma_2, H\Z/2)[v^{-1}] \to \Sigma \Ops_{H\Z/2}.
\]
of $\imap(B_\et \Sigma_2, H\Z/2)$-modules.
\item The map is compatible with the Steenrod comodule structures in the following sense: Under the equivalence of (1), we have the following commutative diagram 
\begin{equation*}
\begin{tikzcd}
\pi_{**}\imap( B_{\et}\Sigma_{2+} , H\Z/2)[v^{-1}] \ar{dr}{t[v^{-1}]_*} \ar{dd}{\psi_R}&  \\
 &  \pi_{**}\Sigma\Ops_{H\Z/2} \ar{dd}{\psi_L}\\
\pi_{**}\imap(B_{\et}\Sigma_{2+}, H\Z/2)[v^{-1}]  \widehat{\otimes}_{H\Z/2_{**}} \scr A_{**} \ar{dd}{\chi}  & \\
 & \scr A_{**} \widehat{\otimes}\pi_{**}\Sigma\Ops_{H\Z/2}\\
 \scr A_{**} \widehat{\otimes}_{H\Z/2_{**}} \pi_{**}\imap(B_{\et}\Sigma_{2+}, H\Z/2)[v^{-1}] \ar{ur}{\simeq,t[v^{-1}]_*} \ar{ur} &  \\
\end{tikzcd}
\end{equation*}
\item We have $t(v^i) = \Sigma e_{-2i-1}$ and $t(uv^i) = \Sigma e_{-2i-2} + \rho\Sigma e_{-2i-1}.$\todo{We are slightly unhappy about this formula (specifically the second term), but can't find anything wrong with it.}
%
\end{enumerate}
\end{theorem}

\subsubsection{Some equivariant preparation} To proceed further, we will need to delve into some equivariant motivic homotopy theory. 
\begin{lemma} \label{free-e} For any $E \in \SH^{G}(S)$ the canonical map $E \rightarrow \imap(\EE G_{+}, E)$ induces an equivalence
\begin{equation*}
\EE G_+ \wedge E \xrightarrow{\wequi} \EE G_{+} \wedge \imap(\EE G_{+}, E). 
\end{equation*}
\end{lemma}
\begin{proof}
In the notation of Proposition \ref{prop:coloc-smash}, we must prove that $u^*(E \to \imap(\EE G_{+}, E))$ is an equivalence.
In other words for a smooth $G$-scheme $X$ with a free action and $n \in \Z$, we must show that 
\[
\Map(\Sigma^{-n \rho} \Sigma^{\infty}X_+, E) \rightarrow \Map(\Sigma^{-n \rho} \Sigma^{\infty}X_+, \imap(\EE G_{+}, E)) \simeq \Map(\Sigma^{-n \rho} \Sigma^{\infty}X_+ \wedge\EE G_{+} , E)
\]
is an equivalence.
This holds since $X$ is a free $G$-scheme and so $\Sigma^{\infty}X_+ \wedge\EE G_{+} \simeq \Sigma^{\infty}X_+$ in $\SH^{G}(S).$
\end{proof} 

We now specialize to $G=C_2$.
Recall the  \emph{motivic isotropy separation cofiber sequence}, see e.g. \cite[(2.7)]{Heller:2016aa},
\begin{equation*} \label{isotropy}
\EE C_{2+} \rightarrow S^0 \rightarrow \widetilde{\EE }C_2,
\end{equation*}
which is a cofiber sequence in $\SH^{C_2}(S).$
As explained in \emph{loc. cit}, we furthermore have that
\begin{equation} \label{geom-fix}
\widetilde{\EE }C_2 \simeq \colim_n T^{n \sigma},
\end{equation}
where the colimit diagram is given by tensoring with the Euler class of Definition~\ref{def:euler-class}:
\begin{equation*} \label{eq:euler-colim}
\cdots \to T^{-\sigma} \stackrel{e(\sigma) \otimes \id}{\rightarrow} S^0 \stackrel{e(\sigma)}{\rightarrow} T^{\sigma} \stackrel{e(\sigma) \otimes \id}{\rightarrow} T^{2\sigma} \rightarrow \cdots T^{n\sigma}\stackrel{e(\sigma) \otimes \id}{\rightarrow} T^{(n+1)\sigma} \to \cdots.
\end{equation*}

Note that from Lemma~\ref{free-e} we have for any $X\in \SH^{C_2}(S)$ the cofiber sequence
\begin{equation}\label{eq:norm-cofiber}
	 \EE C_{2+} \wedge X \to \imap(\EE C_{2+}, X) \to \widetilde\EE C_2 \wedge \imap(\EE C_{2+}, X).
\end{equation}

\begin{lemma} \label{lem:tate-of-orientable}
Let $E \in \CAlg(h\SH(S))$ be oriented.
Write $v \in E^{2,1} BC_2$ for the Euler class of the tautological bundle.
Then there is a canonical equivalence of homotopy ring spectra \[ (\widetilde\EE C_{2} \wedge \imap(\EE C_{2+}, E^\triv))^{C_2} \wequi \imap(\BB C_{2+}, E)[v^{-1}]. \]
\end{lemma}
\begin{proof}
Using \eqref{geom-fix} and cocontinuity of $(\ph)^{C_2}$ we get \[ (\widetilde\EE C_{2+} \wedge \imap(\EE C_{2+}, E^\triv))^{C_2} \wequi \colim_n \imap(\EE C_{2+}, \Sigma^{n\sigma} E^\triv). \]
The result now follows from Lemma \ref{lem:equiv-thom}, using that the equivariant Euler class $e(\sigma)$ corresponds to the Euler class $v$ of the tautological bundle.
\end{proof}

We shall now provide a different computation of the left hand side of the above result.
We begin with the following.
\begin{lemma} \label{tilde-versus-1-susp} For each $n \in \Z$, there is a canonical equivalence in $\SH^{C_2}(S)$
\begin{equation*} \label{tilde-1-susp}
\widetilde{\EE }C_2/T^{-n\sigma} \simeq T^{-n\sigma} \wedge \Sigma \EE C_{2+},
\end{equation*}
which is compatible with the Euler class in the sense that the following diagram commutes
\begin{equation*} \label{compat-eul}
\begin{tikzcd}
\widetilde{\EE }C_2/T^{-n\sigma} \ar{d} \ar{r}{\simeq} &  T^{-n\sigma} \wedge \Sigma \EE C_{2+} \ar{d}{e(\sigma) \otimes \id} \\
\widetilde{\EE }C_2/T^{-(n-1)\sigma}   \ar{r}{\simeq} & T^{-(n-1)\sigma} \wedge \Sigma \EE C_{2+}.
\end{tikzcd}
\end{equation*}
\end{lemma}
\begin{proof} For the first claim we note that the following diagram of cofiber sequences commutes
\begin{equation*} \label{eq:mult-e}
\begin{tikzcd}
T^{-n\sigma} \otimes S^0 \ar{r} \ar{d}{\simeq} & T^{-n\sigma} \otimes \widetilde{\EE }C_2 \ar{r} \ar{d}{e(\sigma)^{\otimes n}} & T^{-n\sigma} \otimes \Sigma \EE C_{2+} \ar{d}\\
T^{-n\sigma}  \ar{r}{e(\sigma)^{\otimes n}}& \widetilde{\EE }C_2 \ar{r} &  \widetilde{\EE }C_2/T^{-n\sigma},
\end{tikzcd}
\end{equation*}
where the top horizontal row is obtained by smashing the cofiber sequence $S^0 \rightarrow \widetilde{\EE }C_2 \rightarrow \Sigma \EE C_{2+}$ with $T^{-n\sigma}$. Now the middle vertical arrow is an equivalence (as follows from \eqref{geom-fix}), which gives us the first statement. Comparing the left square of the diagram  with the corresponding left square of the diagram for the $(n-1)$-case, we obtain the desired compatibility with multiplication by the Euler class.
\end{proof}

Now for $E \in \SH^{C_2}(S)$ we obtain maps \begin{gather*} \widetilde\EE C_{2+} \wedge \imap(\EE C_{2+}, E) \to \widetilde\EE C_{2+}/T^{-n\sigma} \wedge \imap(\EE C_{2+}, E) \\ \stackrel{L.\ref{tilde-versus-1-susp}}{\wequi} T^{-n\sigma} \wedge \Sigma \EE C_{2+} \wedge \imap(\EE C_{2+}, E) \stackrel{L.\ref{free-e}}{\wequi} \Sigma T^{-n\sigma} \wedge \EE C_{2+} \wedge E. \end{gather*}


By \eqref{eq:norm-cofiber},  we get a commutative diagram where the rows are exact and the vertical arrows are multiplication with the Euler class

\begin{equation}\label{eq:big-ops-diagram}
	\begin{tikzcd}[column sep=small]
		\EE C_{2+}\wedge E \ar[r] & \imap(\EE C_{2+}, E) \ar[r] & \widetilde\EE C_2\wedge \imap(\EE C_{2+}, E)  
		\\
		T^{-\sigma}\wedge\EE C_{2+}\wedge E \ar[r]\ar[u, "e(\sigma)"] & T^{-\sigma}\wedge\imap(\EE C_{2+}, E) \ar[r]\ar[u, "e(\sigma)"]& T^{-\sigma}\wedge\widetilde\EE C_2\wedge \imap(\EE C_{2+}, E)  \ar[u, "e(\sigma)", "\simeq"'] 
		\\
		\vdots \ar[u, "e(\sigma)"] & \vdots \ar[u, "e(\sigma)"]  & \vdots \ar[u, "e(\sigma)", "\simeq"'] 
		\\
		T^{-n\sigma}\wedge\EE C_{2+}\wedge E \ar[r]\ar[u, "e(\sigma)"] & T^{-n\sigma}\wedge\imap(\EE C_{2+}, E) \ar[r]\ar[u, "e(\sigma)"]& T^{-n\sigma}\wedge\widetilde\EE C_2\wedge \imap(\EE C_{2+}, E)  \ar[u, "e(\sigma)", "\simeq"'] 
		\\
		\vdots \ar[u, "e(\sigma)"] & \vdots \ar[u, "e(\sigma)"]  & \vdots \ar[u, "e(\sigma)", "\simeq"'] 
	\end{tikzcd}
\end{equation}

Taking the limit of the previous diagram yields the cofiber sequence
\begin{equation}\label{eq:lim-big-diagram}
	\lim_n \left( T^{-n\sigma} \wedge \EE C_{2+} \wedge E \right) \to
\lim_n \left( T^{-n\sigma} \wedge\imap(\EE C_{2+}, E)\right) \to	\widetilde\EE C_{2+} \wedge \imap(\EE C_{2+}, E)
\end{equation}

\begin{lemma} \label{lem:ops-halfway}
	For $E \in \SH^{C_2}(S)$, we have
	$\lim_n \left( T^{-n\sigma} \wedge\imap(\EE C_{2+}, E) \right)\simeq 0$. In particular,	
	the  map \[ \widetilde\EE C_{2+} \wedge \imap(\EE C_{2+}, E) \to \lim_n \left( \Sigma T^{-n\sigma} \wedge \EE C_{2+} \wedge E \right) \] is an equivalence.
\end{lemma}
\begin{proof}
	We have
	\begin{eqnarray*}
		\lim_n (T^{-n\sigma} \wedge \imap(\EE C_{2+}, E)) & \simeq & \lim_n \imap(\EE C_{2+} \wedge T^{n\sigma} , E)\\
		& \simeq &  \imap(\colim_n \EE C_{2+} \wedge T^{n\sigma} , E)\\
		& \simeq & \imap(\EE C_{2+} \wedge \widetilde{\EE} C_2, E)\\
		& \simeq & \imap(0, E) \\
		& \simeq & 0.
	\end{eqnarray*}
	(In the middle we have used \eqref{geom-fix}.)
\end{proof}

Now let $E \in \SH(S)$.
We have \[ ( T^{-n\sigma} \wedge \EE C_{2+} \wedge E^\triv)^{C_2} \wequi ( T^{-n\sigma} \wedge E^\triv)_{\hh C_2} \wequi T^{n} \wedge \DD_2(T^{-n}) \wedge E. \]
Here the first equivalence is by the \emph{motivic Adams isomorphism} \cite[Proposition 5.11]{gepner-heller}, and the second is by the projection formula \ref{lemm:triv-quot-projection-formula}.
Note that these are exactly the spectra appearing in the definition of $\Ops_E$.
The following result says that the transition maps also agree.

\begin{lemma} \label{lem:ops-alltheway}
	For $E \in \SH(S)$ we have a canonical equivalence 
	\[	
	\lim_n \left( T^{-n\sigma} \wedge \EE C_{2+} \wedge E \right)^{C_2}\simeq \Ops_E. 
	\]
	In particular, there is a canonical equivalence 
	\[ 
	 (\widetilde\EE C_{2+} \wedge \imap(\EE C_{2+}, E))^{C_2} \xrightarrow{\wequi} \Sigma \Ops_E. 
	 \]
\end{lemma}
\begin{proof}
	The second statement follows from Lemma \ref{lem:ops-halfway}. We need to verify that the induced maps $T^{n} \wedge \DD_2(T^{-n}) \to T^{n+1} \wedge \DD_2(T^{-(n+1)})$ are the ones appearing in 
	Definition \ref{def:spectrum-of-ops}.
	The maps being induced by $e(\sigma)$, this follows from Lemma \ref{lem:diagonal-euler}.
\end{proof}

\subsubsection{}
\begin{proof}[Proof of Theorem \ref{thm:homotopy-of-ops}.]
The map $t$ arises as 
\begin{align*} \imap(\Sigma^\infty_+ B_\et \Sigma_2, H\Z/2) &\to \imap(\Sigma^\infty_+ B_\et \Sigma_2, H\Z/2)[v^{-1}] \\ &\stackrel{L.\ref{lem:tate-of-orientable}}{\wequi} (\widetilde\EE C_{2+} \wedge \imap(\EE C_{2+}, H\Z/2^\triv))^{C_2} \stackrel{L.\ref{lem:ops-alltheway}}{\wequi} \Sigma \Ops_{H\Z/2}. 
\end{align*}
It is easily verified to be a map of modules.\todo{...}
This proves (1).
Moreover by construction we get such an equivalence for any oriented ring spectrum $E$, and the equivalence is natural.
Applying this to the unit map $H\Z \to H\Z \wedge H\Z$ one deduces that $t$ is compatible with left coactions.
Result (2) follows by invoking \eqref{eq:left-to-right}.\NB{Is this too easy?}

(3) Write $\mathcal{L}$ for the line bundle on $BC_2$ corresponding to $\sigma$. 
Consider the cofiber sequence
\[
T^{-n\sigma}\wedge \EE C_{2+}\wedge H\Z/2\to T^{-n\sigma} \wedge \imap(\EE C_{2+}, H\Z/2)\to T^{-n\sigma}\wedge \widetilde\EE C_{2} \wedge \imap(\EE C_{2+}, H\Z/2) .
\]
 Upon taking fixed points, we obtain the cofiber sequence
 \[
 \Th(-\mathcal{L}^{n})\wedge H\Z/2 \to \imap(\Th(\mathcal{L}^{n}),H\Z/2)\to \imap(\Th(\mathcal{L}^n),H\Z/2)[v^{-1}].
 \]
 Note that the $H\Z/2_**$-module map
 \[
 \pi_{**}(\Th(-\mathcal{L}^{n})\wedge H\Z/2) \to \pi_{**}(\imap(\Th(\mathcal{L}^{n}),H\Z/2) )
 \] 
 is $0$. 
We thus obtain short exact sequences of $H\Z/2_{**}$-modules
\[
0\to \pi_{**}\imap(\Th(\mathcal{L}^{n}),H\Z/2)\to  \pi_{**}\imap(\Th(\mathcal{L}^{n}),H\Z/2)[v^{-1}] \to \pi_{**}(\Sigma\Th(-\mathcal{L}^n)\wedge H\Z/2) \to 0.
\]
Applying Thom isomorphisms, we these are expressed as
\[
0\to \Sigma^{2n,n}\pi_{**}\imap(BC_{2+}, H\Z/2) \to \Sigma^{2n,n}\pi_{**}\imap(BC_{2+},H\Z/2)[v^{-1}] \to \pi_{**}(\Sigma\Th(-\mathcal{L}^n)\wedge H\Z/2)\to 0,
\] 
which in turn is identified with
\begin{align*}
0\to \bigoplus_{i\geq n} H\Z/2_{**}&\{v^i\}\oplus   \bigoplus_{i\geq n} H\Z/2_{**}\{uv^i\} 
\\& \to  \bigoplus_{i\in \Z} H\Z/2_{**}\{v^i\}\oplus  \bigoplus_{i\in \Z} H\Z/2_{**}\{uv^i\} 
 \xrightarrow{t_{n}} \pi_{**}(\Sigma\Th(-\mathcal{L}^n)\wedge H\Z/2)\to 0.
\end{align*}
From this we obtain 
\[
\pi_{**}(\Sigma\Th(-\mathcal{L}^n)\wedge H\Z/2) \cong  \bigoplus_{i<n} H\Z/2_{**}\{v^i\}\oplus  \bigoplus_{i<n} H\Z/2_{**}\{uv^i\}.
\]
On the other hand, since $C(-n,-n)\wedge H\Z/2\simeq \Th(-\mathcal{L}^n)\wedge H\Z/2$, (the proof of) Theorem~\ref{thm:hpty-of-ops} gives 
\[
\pi_{**}(\Sigma\Th(-\mathcal{L}^n)\wedge H\Z/2)\cong
\bigoplus_{j\geq -n} H\Z/2_{**}\{u^{(n)}_j\}\oplus  \bigoplus_{j\geq -n} H\Z/2_{**}\{v^{(n)}_j\}.
\]
The map $t$ is identified with the limit of the maps $t_n$. We need to see that 
$t_n(v^i) = \Sigma u_{-i-1}^{(n)}$ and $t_n(uv^i) = \Sigma v_{-i-1}^{(n)}+ \rho\Sigma u_{-i-1}^{(n)}$. 
To check this, we may assume that $S=\Spec(\Z[1/2])$.
We can write $t_n(v^i) = \sum (\lambda_j\cdot \Sigma u_{j-1}^{(n)} + \lambda'_j \cdot \Sigma v_{j-1}^{(n)}) $ where $\lambda_j,\lambda_j'\in H\Z/2_{**}$.
A straightforward analysis using Lemma \ref{lem:ptm} shows that in fact $t_n(v^i) = \lambda \Sigma u_{-i-1}$ for some $\lambda \in \F_2$ (i.e. all other coefficients vanish for degree reasons).
Since $t_n(v_i)$ cannot be zero (it is a basis element), we must have $\lambda_0=1$, i.e., $t_n(v^i) = \Sigma u_{-i-1}^{(n)}$.

Since $t_n$ is a map of $\pi_{**}\imap(BC_{2+}, H\Z/2)$-modules 
we have $t_{n}(uv^i) = u\cdot \Sigma u_{-i-1}^{(n)}= 
\Sigma (u\cdot u_{-i-1}^{(n)})$. 
We have an isomorphism, see Lemma~\ref{lemm:mot-hom-cohom-dual},
\[
H_{**}(\Th(-\mathcal{L}^n))\xrightarrow{\cong} \Hom_{H\Z/2_{**}}(H^{**}(\Th(-\mathcal{L}^n)), H\Z/2_{**}).
\]
The element $u_{-i-1}^{(n)}$ is dual to $uv^{-i-1}$. 
We have $\langle -,\,  u\cdot u_{-i-1}^{(n)} \rangle = \langle -\cdot u,\,   u_{-i-1}^{(n)} \rangle$, so (using Lemma \ref{lemm:coh-of-BSigma2} below)
\[
\langle u^{\epsilon}v^{j},  u\cdot u_{-i-1}^{(n)} \rangle
= \begin{cases}
	1 & \epsilon  = 0\textrm{ and } j=-i-1 \\
	\rho & \epsilon = 1\textrm{ and } j=-i-1\\
	0 & \textrm{else}.
\end{cases}
\]
It follows that $u\cdot u_{-i-1}^{(n)} = v_{-i-1}^{(n)} + \rho u_{-i-1}^{(n)}$, as needed.
\end{proof}

\subsection{Comodule structure} \label{subsec:comodule-structure}

In summary, we have constructed a completed right comodule structure 
\begin{equation} \label{eq:right}
\psi_R: \pi_{**}\imap( B_{\et}\Sigma_{2+} , H\Z/2) \rightarrow   \pi_{**}\imap(B_{\et}\Sigma_{2+}, H\Z/2)  \widehat{\otimes}_{H\Z/2_{**}} \scr A_{**},
\end{equation}
and a completed left comodule structure 
\begin{equation} \label{eq:left}
\psi_L: \pi_{**}\Ops_{H\Z/2} \rightarrow \scr A_{**} \widehat{\otimes}\pi_{**}\Ops_{H\Z/2},
\end{equation}
whose compatibility is expressed in Theorem \ref{thm:homotopy-of-ops}(2).

Our main goal is to compute the $\psi_L$, but we do so via $\psi_R$ which we now proceed to describe. To do so, we begin by choosing a presentation $H^{**}(B_{\et}\Sigma_2; \Z/2)$ by generalizing \cite[Theorem 6.10]{voevodsky2003reduced}. We first name some elements; we assume that $2$ is invertible in $S$.

\begin{enumerate}
\item Using \cite[Theorem 7.10]{spitzweck2012commutative}, we have an isomorphism $H^{0,1}(S; \Z/2) \simeq \{ \pm 1 \} \subset \scr O_S(S)^{\times}$. We call the element corresponding to to $-1$ the \emph{motivic Bott element} which we denote by $\tau$.
\item Using \emph{loc. cit}, we also have an isomorphism $H^{1,1}(S; \Z) \cong H^0(S; \scr O^\times)$. We denote by $\rho \in H^{1,1}(S; \Z)$ the class corresponding to $-1$. We denote its image in $H^{1,1}(S, \Z/2)$ also by $\rho$.
\item We have already discussed the element $v \in H^{2,1}(B_{\et}\Sigma_{2+}; \Z/2)$ which is the Euler class of the line bundle $\scr L(\sigma)$.
\item Lastly, there is an element $u \in H^{1,1}(B_{\et}\Sigma_{2+}; \Z/2)$ which is uniquely characterized by the fact that it Bocksteins to $v$ and restricts to $0$ to the base point of $B_{\et}\Sigma_{2}$ \cite[Lemma 11.6]{spitzweck2012commutative}.
\end{enumerate}

For $E \in \HNAlg_2(\SH(S))_{H\Z/2/}$ we will continue to denote the images of $u, v, \rho, \tau$  under the map $H\Z/2 \rightarrow E$ by the same name.

\begin{lemma} \label{lemm:coh-of-BSigma2}
Let $E \in \HNAlg_2(\SH(S))_{H\Z/2/}$ and $\scr X \in \Spc(S)_\pt$. Then we have an isomorphism of $E_{**}$-algebras \[E^{**}(\scr X \wedge (B_\et \Sigma_2)_+) \wequi E^{**}(\scr X)\fpsr{u, v}/(u^2 = \tau v + \rho u). \]
\end{lemma}
\begin{proof}
The additive structure is \cite[Theorem 11.14]{spitzweck2012commutative} after noting that $E$ is an $H\Z/2$-module. In order to determine the multiplicative structure, we may assume $E = H\Z/2$, $\scr X = S^0$ and $S = Spec(\Z)$. We need only determine $u^2$. Since motivic cohomology vanishes in negative weights, $u^2 = av + bu + c + duv$, where $a \in H^{0,1}(S, \Z/2), b \in H^{1,1}(S, \Z/2), c \in H^{0,1}(S, \Z/2)$ and $d \in H^{-1, 0}(S, \Z/2)$. Since $H^{*,0}$ is ordinary cohomology with $\Z/2$ coefficients, $d=0$. The inclusion $* \hookrightarrow B_\et \Sigma_2$ kills $u$ but detects $c$, so $c=0$. Since $\Z$ is a unique factorization domain, Proposition \ref{prop:HZ-dedekind}(2) implies that $H^{*,1}(Spec(\Z), \Z) = \{\pm 1\}$ and so $H^{*,1}(Spec(\Z), \Z/2) = \Z/2\{\tau, \rho\}$\todo{possibly make this a separate lemma}. In other words, either $a=0$ or $a=\tau$, and either $b=0$ or $b=\rho$. This can be checked after pullback to any field in which $-1$ is not a square (and so in particular $-1 \ne 1$, i.e. $\rho$ and $\tau$ both remain non-zero), whence the result follows from \cite[Lemma 6.8, Lemma 6.9]{voevodsky2003reduced}.
\end{proof}

\subsubsection{Formulas for $\psi_R, \psi_L$} \label{sect:formulas} We will now express the map $\psi_R$ in terms of the bases above. For notation and facts about the dual Steenrod algebra, see Theorem~\ref{thm:HZ-dual-steenrod}. We note that, by construction, the map $\psi_R$ is a map of algebras, so it suffices to determine what the map does on $u$ and $v$. This is given by
\begin{lemma} \label{lem:psir-easy} \cite[Corollary 11.23]{spitzweck2012commutative} The completed right comodule action $\psi_R$ on the additive basis of $\pi_{**}\Hom( B_{\et}\Sigma_{2+} , H\Z/2) \cong H^{**}(B_{\et}\Sigma_{2+}; \Z/2)$ is given by
\begin{enumerate}
\item $\psi_R(v) = v + \SSigma_{i \geq 1} \xi_i \otimes v^{2^i}$, and
\item $\psi_R(u) = u + \SSigma_{i \geq 0} \tau_i \otimes v^{2^i}$.
\end{enumerate}
Consequently, the completed left comodule structure is given by
\begin{enumerate}
\item $\psi_L(v) = v + \SSigma_{i \geq 1} v^{2^i} \otimes \overline{\xi_i}$, and
\item $\psi_L(u) = u + \SSigma_{i \geq 0}  v^{2^i} \otimes \overline{\tau_i}$.
\end{enumerate}
\end{lemma}

In order to present the answer for $v^k, v^ku$ for $k \in \Z$, we will introduce some notation. Consider the following power series in $\scr A_{**}\fpsr{t}$. Define
\[ \xi(t) := \underset{i \geq 0}{\mathlarger{\sum}} \xi_i t^{2^i} \quad\text{and}\quad \tau(t) := \underset{i \geq 0}{\SSigma} \tau_it^{2^i}. \]
We will also consider their conjugates $\overline{\xi}(t) := \overline{\xi(t)} := \underset{i \geq 0}{\SSigma} \overline{\xi_i} t^{2^i}$, $\overline{\tau}(t) := \overline{\tau(t)} := \underset{i \geq 0}{\SSigma} \overline{\tau_i}t^{2^i}$. If $f(t) \in \scr A_{**}\fpsr{t}$, we write $[f(t)]_{t^{k}}$ to denote the coefficient of $t^k$ in the power series $f(t)$.

A key result about these polynomials which we will use later is the following.
\begin{lemma} \label{lem:comp-inv} We have the following identities
\begin{enumerate}
\item $\xi( \overline{\xi} (t)) = \overline{\xi}(\xi(t)) = t$
\item $\tau(\overline{\xi}(t)) = \overline{\tau}(t)$
\item $\overline{\tau}(\xi(t)) = \tau(t)$. 
\end{enumerate}
\end{lemma}
\begin{proof}
The identities above are all immediate consequences of the formulas for antipodes as written in Theorem~\ref{thm:HZ-dual-steenrod}(5).
For example, for the first identity, we can write $\xi(\overline{\xi}(t))$ as a formal power series in $t$; clearly the only non-vanishing coefficients occur at $t^{2^i}$ for $i \ge 0$.
The coefficient at $t = t^{2^0}$ is immediately verified to be $1$; vanshing of the coefficient at $t^{2^r}$ means that \[ \sum_{i=0}^r \xi_{r-i}^{2^i} \overline{\xi}_i = 0, \] which is equivalent to the first formula in Theorem~\ref{thm:HZ-dual-steenrod}(5).
Conjugating both sides of $\xi(\overline{\xi}(t)) = t$ we deduce that $\overline{\xi}(\xi(t)) = t$ as well.
Formulas (2) and (3) are established similarly.
\end{proof}

\begin{lemma} \label{lem:psir} The completed right comodule action $\psi_R$ on $\pi_{**}\Hom( B_{\et}\Sigma_{2+} , H\Z/2)[v^{-1}]$ is given by
\begin{enumerate}
\item $\psi_R(v^k) = \underset{j \in \Z}{\SSigma} v^j \otimes [\xi(t)^k]_{t^j}$, and
\item $\psi_R(v^ku) = \underset{j \in \Z}{\SSigma} v^ju \otimes [\xi^k(t)]_{t^j} + \underset{j \in \Z}{\SSigma} v^j \otimes [\tau(t)\xi^k(t)]_{t^j}$.
\end{enumerate}
\end{lemma}
\begin{proof} 
If $f$ is a formal power series in $t$, we have $f(v) = \sum_i v^i [f(t)]_{t^i}$.
Rewriting like this, the formulas take the form \[ \psi_R(v^k) = \xi(v)^k \quad\text{and}\quad \psi_R(v^k u) = \xi(v)^k(u + \tau(v)). \]
Since $\psi_R$ is a ring homomorphism, we are reduced to showing that $\psi_R(v) = \xi(v)$ and $\psi_R(u) = u + \tau(v)$.
Via Lemma \ref{lem:psir-easy}, this holds essentially by definition.
\end{proof}

\begin{lemma}[\cite{wilson2017power}, Lemma 3.23] \label{lemm:inversion-trick}
We have $[\xi(t)^r \tau(t)]_{t^s} = [\overline\tau(t)\overline\xi(t)^{-s-1}]_{t^{-r-1}}$ and $[\xi(t)^r]_{t^s} = [\overline{\xi}(t)^{-s-1}]_{t^{-r-1}}$.
\end{lemma}
\begin{proof}
Recall the residue formula $(2\pi i)[F(t)]_{t^s} = \oint F(t)t^{-s-1}\mathrm{d}t$ from complex analysis.
Taking $F(t) = \xi(t)^r \tau(t)$ and substituting $t = \overline\xi(u)$ (whence $\mathrm{d}t = \overline{\xi'}(u) \mathrm{d}u$) and using Lemma \ref{lem:comp-inv} we deduce the formal power series identity $[\xi(t)^r \tau(t)]_{t^s} = [\overline\tau(t)\overline\xi(t)^{-s-1}\overline{\xi'}(t)]_{t^{-r-1}}$.
Since we are in characteristic $2$, $\overline{\xi'}(t) = 1$, whence the first formula.
The second one is proved similarly.
\end{proof}

\begin{corollary} \label{cor:psil} The completed left comodule action $\psi_L$ on $\Ops_{H\Z/2}$ is given by
\begin{enumerate}
\item $\psi_L(e_{2k}) = \underset{j \in \Z}{\SSigma} [\xi(t)^j]_{t^k} \otimes e_{2j} + \underset{j \in \Z}{\SSigma} [\tau(t)\xi(t)^j]_{t^k} \otimes e_{2j+1}$, and
\item $\psi_L(e_{2k+ 1}) = \underset{j \in \Z}{\SSigma} [\xi(t)^j]_{t^k} \otimes e_{2j+1}$.
\end{enumerate}
\end{corollary}
\begin{proof}
We prove the first equation, the second is established similarly.
By Theorem \ref{thm:homotopy-of-ops}(3) we have \[ t(v^{-k-1}(u + \rho)) = e_{2k}. \]
Hence using Theorem \ref{thm:homotopy-of-ops}(2) we get \[ \psi_L(e_{2k}) = \chi(t(\psi_R(v^{-k-1}(u + \rho)))). \]
This is evaluated using Lemma \ref{lem:psir} to be \[ t\left(\sum_{j \in \Z} [v^j(u+\rho+\overline{\tau}(t))\overline{\xi}(t)^{-k-1}]_{t^j}\right). \]
The result now follows from Lemma \ref{lemm:inversion-trick} (and Theorem \ref{thm:homotopy-of-ops}(3) again).
\end{proof}

\section{Power operations and the motivic Steinberger theorem} \label{sec:power-ops}
\subsection{Definition of the operations}\label{def-ops}
We can use homotopy classes in the spectra of operations to define power operations, as follows.

\begin{definition} \label{def:power-ops-general}
Let $R \in \HNAlg_2(\SH(S))$ and $E \in \HNAlg_2(\SH(S))_{R/}$ (more generally we are allowed to let $E \in \Mod_R^w(\HNAlg_2(\SH(S))$ by Remark \ref{rmk:mapping-action-conormed}). For $c \in \pi_{m,n}(\Ops_R)$ we define an operation 
\[Q^c: \pi_{**}(E) \to \pi_{*+m,*+n}(E)
\] as follows. Let $x \in \pi_{a,b}(E)$ be represented by a map 
\[x: S^{a,b} \to E
\] and $c$ by 
\[c: S^{m,n} \to \Ops_R.
\] Then $Q^c(x)$ is the following composite in $\SH(S)$
\begin{equation} \label{eq:power-map}
\begin{tikzcd}
S^{a+m, b+n} \ar[r, "\id \wedge c"] \ar[d, dashed, "Q^c(x)"]
  & S^{a,b} \wedge \Ops_R \ar[r, "\pi_{a,b}"] & S^{a,b} \wedge S^{-a, -b} \wedge D_2^{\mot}(S^{a, b}) \wedge R \ar[d, "D_2^{\mot}(x) \wedge \id"]   \\
E & E \wedge R \ar[l]                         & D_2^{\mot}(E) \wedge R \ar[l], \\
 \end{tikzcd}
\end{equation}
where the bottom horizontal map comes from the multiplication in $E$ and the $R$-module structure on $E$ (see Remark \ref{rmk:HNAlg2-modules}).
\end{definition}

\begin{remark} \label{rem:naturality} The power operations are natural in $\HNAlg_2(\SH(S))_{R/}$, i.e. if $E \rightarrow F \in \HNAlg_2(\SH(S))_{R/}$ then for any $x \in \pi_{a, b}(E)$ and $c \in \pi_{m,n}(\Ops_R)$ we have
\[
f_*Q^c(x) = Q^c(f_*(x)).
\]
\end{remark}

\begin{remark} \label{rem:functoriality}For a scheme $S$, let us write $\Ops_R(S) \in \SH(S)$ for the spectrum of operations from Definition \ref{def:spectrum-of-ops} when we need to make the dependence on $S$ clear. If $S$ is any scheme with residue characteristics $\ne 2$, then there is a unique map $f: S \to \Z[1/2]$. There is then an induced map \[ f^* \Ops_{H\Z/2}(\Z[1/2]) \wequi f^*(\lim C \wedge H\Z/2) \to \lim C \wedge H\Z/2 \wequi \Ops_{H\Z/2}(S)\]
and hence $\pi_{**} \Ops_{H\Z/2}(\Z[1/2]) \to \pi_{**} \Ops_{H\Z/2}(S)$.
\end{remark}

Theorem \ref{thm:hpty-of-ops} provides us with canonical classes $e_i \in \pi_{i, \ceil{i/2}}(\Ops_{H\Z/2}(\Z[1/2]))$, the images of which in $\pi_{**}(\Ops_{H\Z/2}(S))$ we still denote by $e_i$.
\begin{definition}
Let $S$ be a scheme on which $2$ is invertible, and $E \in \HNAlg_2(\SH(S))_{H\Z/2/}$ (or more generally $E \in \Mod_R^w(\HNAlg_2(\SH(S))$). For $i \in \Z$ denote by $Q^i: \pi_{**}(E) \to \pi_{*+i, *+\ceil{i/2}}(E)$ the power operation corresponding via Definition \ref{def:power-ops-general} to the class $e_i$.
\end{definition}

\subsection{Basic properties}\label{sub:bp}
In this section, we record the basic properties of the power operations $Q^i$. Throughout, all schemes will be assumed to have residue characteristics $\ne 2$.

\subsubsection{Naturality} The next proposition asserts that the operations are functorial for maps of homotopy normed spectra.
\begin{proposition}
Let $\alpha: E \to F \in \HNAlg_2(\SH(S))_{H\Z/2/}$ (or more generally in $\Mod_{H\Z/2}^w(\HNAlg_2(\SH(S))$). Then for all $m,n,i$, the following diagram commutes
\begin{equation*}
\begin{CD}
\pi_{m,n}(E) @>{\alpha_*}>> \pi_{m,n}(F) \\
@V{Q^i}VV                   @V{Q^i}VV    \\
\pi_{m+i,n+\ceil{i/2}}(E) @>{\alpha_*}>> \pi_{m+i,n+\ceil{i/2}}(F). \\
\end{CD}
\end{equation*}
\end{proposition}
\begin{proof}
This follows from Remark \ref{rem:functoriality}.
\end{proof}

\subsubsection{Base change} We further have the following base change property.
\begin{proposition} \label{prop:powerop-base-change}
Let $f: S' \to S$ be any morphism and $E \in \HNAlg_2(\SH(S))_{H\Z/2/}$ (or more generally in $\Mod_{H\Z/2}^w(\HNAlg_2(\SH(S))$). Then for all $m,n,i$, the following diagram commutes
\begin{equation*}
\begin{CD}
\pi_{m,n}(E) @>{f^*}>> \pi_{m,n}(f^*F) \\
@V{Q^i}VV                   @V{Q^i}VV    \\
\pi_{m+i,n+\ceil{i/2}}(E) @>{f^*}>> \pi_{m+i,n+\ceil{i/2}}(f^*E). \\
\end{CD}
\end{equation*}
\end{proposition}
\begin{proof}
We need to establish that the diagram~\eqref{eq:power-map} defining $Q^c(x)$ is stable under base change. This is clear since all constituent spectra and morphisms are (see Example \ref{ex:Dn-stable-base-change} for $D_2^\mot(x)$).
\end{proof}

\subsubsection{Squaring} In certain degrees, the power operation acts by squaring a class. To prove this, we begin with the following lemma.
\begin{lemma} \label{lemm:squaring-class}
Let $m \ge 0$.
\begin{enumerate}
\item The composite \[ T^{2m} \xrightarrow{p} D_2^\mot(T^m) \to D_2^\mot(T^m) \wedge H\Z/2 \wequi T^{2m} \wedge BC_{2+} \wedge H\Z/2 \] is homotopic to the ``bottom cell'' $T^{2m} \wedge e_0$.
\item The following diagram commutes
\begin{equation*}
\begin{CD}
S^{2\cdot(2m+1,m+1)} @>p>> D_2^\mot(S^{2m+1,m+1}) \wedge H\Z/2 \\
@V{e_1}VV                                                     @A{\Delta}AA        \\
S^{4m,2m} \wedge S^{1,1} \wedge D_2^\mot(S^0) \wedge H\Z/2 @= S^{1,1} \wedge D_2^\mot(S^{2m,m}) \wedge H\Z/2. \\
\end{CD}
\end{equation*}
\end{enumerate}
Here the maps $p$ are the canonical ones.
\end{lemma}
\begin{proof}
Compatibility of $p$ with Thom isomorphisms reduces immediately to $m=0$.
The case (1) is clear.

For (2), we argue as follows.\todo{I would like a more direct proof... also a less sloppy one}
We may assume that $S = \Spec(\Z[1/2])$.
By Corollaries \ref{corr:D-mot-excisive} and \ref{corr:D-mot-reduced}, the functor $D_2^\mot$ is $2$-excisive and reduced.
By Proposition \ref{prop:D2-mot-cr2}, the cross effect is given by $cr_2(\Sigma E, \Sigma E) \wequi E \wedge E$.
Thus Proposition \ref{prop:fibn-sequence} yields a cofiber sequence \[ \Sigma S^{2,2} \to \Sigma D_2^\mot(S^{1,1}) \to D_2^\mot(S^{2,1}) \xrightarrow{\partial} S^{4,2}. \]
Since we know that $D_2^\mot(S^{2,1}) \wedge H\Z/2$ is pure Tate, with lowest weight term $S^{4,2}$, the map $\partial \wedge H\Z/2$ has at most one non-vanishing component.
This is excluded by considering complex realization.
We thus obtain \[ D_2^\mot(S^{1,1}) \wedge H\Z/2 \wequi (S^{2,2} \vee S^{3,2} \vee S^{4,3} \vee \dots) \wedge H\Z/2. \]
Using Lemma \ref{lem:ptm}(1) we see that $\Delta \circ e_1 = a p$, for some $a \in \F_2$.
Again by considering complex realization, we see that $a=1$ as needed.
\end{proof}

\begin{proposition} \label{prop:squaring}
For $E \in \HNAlg_2(\SH(S))_{H\Z/2/}$ (or more generally in $\Mod_{H\Z/2}^w(\HNAlg_2(\SH(S))$), $\epsilon \in \{0,1\}$ and $x \in \pi_{2m-\epsilon,m}(E)$ we have $Q^{2m-\epsilon}(x) = x^2.$
\end{proposition}
\begin{proof}
By definition, $Q^{2m-\epsilon}(x)$ is the composite \[ S^{2m-\epsilon,m} \wedge S^{2m-\epsilon,m} \xrightarrow{\id \wedge e_{2m-\epsilon}} S^{2m-\epsilon, m} \wedge \Ops_{H\Z/2} \to D_2^\mot(S^{2m-\epsilon,m}) \wedge H\Z/2 \xrightarrow{D_2^\mot(x) \wedge \id} D_2^\mot(E) \wedge H\Z/2 \to E. \]
By Remark \ref{rmk:Ops-HZ2-projection} and Lemma \ref{lemm:squaring-class}, the map $S^{4m-2\epsilon,2m} \wequi S^{2m-\epsilon,m} \wedge S^{2m-\epsilon,m} \to D_2^\mot(S^{2m-\epsilon,m}) \wedge H\Z/2$ coincides with the bottom class $S^{2m-\epsilon,m} \wedge S^{2m-\epsilon,m} \to D_2^\mot(S^{2m-\epsilon})$.
It follows that $Q^{2m-\epsilon}(x)$ coincides with the composite via the top right corner in the following commutative diagram
\begin{equation*}
\begin{CD}
S^{2m-\epsilon,m} \wedge S^{2m-\epsilon,m} @>>> D_2^\mot(S^{2m-\epsilon,m}) \\
@V{x \wedge x}VV              @V{D_2^\mot(x)}VV  \\
E \wedge E               @>>> D_2^\mot(E)  @>>> E.
\end{CD}
\end{equation*}
The composite via the bottom left corner is $x^2$ by definition.
\end{proof}

\subsubsection{Stability}
For $\scr X \in \Spc(S)_*$, by Construction \ref{cons:map-HNAlg} and Remark \ref{rmk:mapping-action-conormed} we have 
\[
\imap(\scr X, E) \in \Mod_{H\Z/2}^w(\HNAlg_2(\SH(S)).
\]
Applying this to $\scr X = S^{p+q,q}$ (with $p, q \ge 0$) we obtain $\Omega^{p+q,q} E \in \Mod_{H\Z/2}^w(\HNAlg_2(\SH(S))$, and so this spectrum still has power operations.
\begin{proposition}
For $E \in \Mod_{H\Z/2}^w(\HNAlg_2(\SH(S))$, $p, q \ge 0$, $m, n, i \in \Z$ the following square commutes
\begin{equation*}
\begin{CD}
\pi_{m+p+q, n+q}(E) @= \pi_{m,n} \Omega^{p+q,q} E \\
@V{Q^i}VV              @V{Q^i}VV \\
\pi_{m+p+q+i, n+q+\ceil{i/2}}(E) @= \pi_{m+i,n+\ceil{i/2}} \Omega^{p+q,q} E. \\
\end{CD}
\end{equation*}
\end{proposition}
\begin{proof}
Let $x: S^{m,n} \to \Omega^{p+q,q} E$ and consider the diagram
\begin{equation*}
\begin{tikzcd}
S^{p+q,q} \wedge S^{m,n} \wedge S^{i, \ceil{i/2}} \ar[d, "\id \wedge \id \wedge e_i"] \\
S^{p+q,q} \wedge S^{m,n} \wedge \Ops_{H\Z/2} \ar[d, "\id \wedge p"] \ar[rd, "p'"] \\
S^{p+q,q} \wedge D_2^\mot(S^{m,n}) \wedge H\Z/2 \ar[d, "\id \wedge D_2^\mot(x) \wedge \id"] \ar[r, "\Delta \wedge \id"] & D_2^\mot(S^{p+q,q} \wedge S^{m,n}) \wedge H\Z/2 \ar[d, "D_2^\mot(\id \wedge x) \wedge \id"] \ar[dd, bend left=100, "D_2^\mot(x') \wedge \id"] \\
S^{p+q,q} \wedge D_2^\mot(\Omega^{p+q,q} E) \wedge H\Z/2 \ar[d, "m'"] \ar[r, "\Delta \wedge \id"] & D_2^\mot(S^{p+q,q} \wedge \Omega^{p+q,q}E) \wedge H\Z/2 \ar[d, "D_2^\mot(\eta) \wedge \id"] \\
E & D_2^\mot(E) \wedge H\Z/2 \ar[l, "m"].
\end{tikzcd}
\end{equation*}
The maps $p$ and $p'$ are the canonical projections, $x'$ is adjoint to $x$, $\Delta$ is the diagonal transformation of \eqref{conorm-general}, $m'$ is adjoint to the multiplication in $\Omega^{p+q,q} E$ and $m$ is the multiplication in $E$. The composite from top to bottom along the left column is adjoint to $Q^i(x)$. The composite along the far right is $Q^i(x')$. It thus suffices to show that the diagram commutes. The top triangle commutes by definition of the transition maps in the inverse system defining $\Ops_{H\Z/2}$, the middle square commutes by naturality of $\Delta$, the right hand triangle commutes by definition of $x, x', \eta$ and the bottom square commutes by definition the multiplication in $\Omega^{p+q,q}E$. This concludes the proof.
\end{proof}

\begin{remark} \label{rmk:adjust-mn-stab}
Suppose we need to study a power operation $Q^i$ originating in $\pi_{m,n}$ of some normed spectrum $E$.
By the above result we can replace $E$ by $\Omega^{p+q,q} E$ (for $p,q \ge 0$), whence our operation now originates in $\pi_{m',n'}$ with $m'=m-p-q$ and $n'=n-q$.
In particular by choosing $q$ sufficiently large we can always ensure that $m' - 2n' = m-n-p+q \ge 0$.
Then by adjusting $p$ we may ensure that $(m',n')$ is of the form $(2w,w)$.
\end{remark}

\subsubsection{Vanishing} The next result gives us some guaranteed vanishing of operations.
\begin{proposition} \label{prop:power-ops-vanishing}
For $E \in \HNAlg_2(\SH(S))_{H\Z/2/}$ (or more generally in $\Mod_{H\Z/2}^w(\HNAlg_2(\SH(S))$), $i \in \Z$, $\epsilon \in \{0,1\}$ and $x \in \pi_{p,q}(E)$ with $i < p-q + \epsilon$ and $i \le q$ we have $Q^{2i-\epsilon}(x) = 0$.
\end{proposition}
\begin{proof}
By stability, we may replace $E$ by $\Omega^{q-i,q-i} E$ and hence assume $i=q$ (and so $p > 2q - \epsilon$). Let $E' = \Omega^{p-(2q - \epsilon),0} E$. Then $x \in \pi_{p,q}(E)$ corresponds to $x' \in \pi_{2q-\epsilon,q}(E')$. Hence by squaring and stability, $Q^{2i-\epsilon}(x)$ corresponds to $x'^2$. It thus suffices to show: for $E \in \HNAlg_2(\SH(S))$ the multiplication $\Omega^{1,0}E \wedge \Omega^{1,0}E \to \Omega^{1,0}E$ is zero.\NB{in particular $\Omega^{1,0}E$ is non-unital} This boils down to the diagonal map $S^1 \to S^1 \wedge S^1$ being null-homotopic.\NB{I guess this must be right. Still slightly disturbing that $\Omega^{1,0}E \wedge \Omega^{1,0}E \to \Omega^{1,0}E$ is zero while $\Omega^{1,0}E \wedge \Omega^{1,0}E \to D_2^\mot(\Omega^{1,0}E)$ and $D_2^\mot(\Omega^{1,0}E) \to \Omega^{1,0} E$ are not.}
\end{proof}

\subsubsection{Additivity} The operations are all additive maps.
\begin{proposition}
Let $E \in \HNAlg_2(\SH(S))_{H\Z/2/}$ (or more generally in $\Mod_{H\Z/2}^w(\HNAlg_2(\SH(S))$), $m,n,i \in \Z$ and $x, y \in \pi_{m,n}(E)$. Then \[Q^i(x+y) = Q^i(x) + Q^i(y).\] \NB{is more generally $Q^i$ compatible with transfers?}
\end{proposition}
\begin{proof}
Using stability (as in Remark \ref{rmk:adjust-mn-stab}), we may assume that $(m,n) = (2p, p)$ for $p$ arbitrarily small.
We have a class $e_i \in H\Z/2_{**} D_2^\mot(T^p)$ such that for $z \in \pi_{2p,p}(E)$ the operation $Q^i(z)$ is given by the composite \[ S^{2(p+i), p + \ceil{i/2}} \xrightarrow{e_i} D_2^\mot(T^p) \wedge H\Z/2 \xrightarrow{D_2^\mot(z)} D_2^\mot(E) \wedge H\Z/2 \to E. \]
Consider the following commutative diagram
\begin{equation*}
\begin{tikzcd}
S^{2(p+i), p + \ceil{i/2}} \ar[r, "e_i"] & D_2^\mot(T^p) \wedge H\Z/2 \ar[r, "D_2^\mot(x+y)"] \ar[d, swap, "D_2^\mot(\nabla)"] & D_2^\mot(E) \wedge H\Z/2 \ar[r] & D_2^\mot(E) \wedge H\Z/2 \ar[r] & E \\
                                         & D_2^\mot(T^p \vee T^p) \wedge H\Z/2. \ar[ur, swap, "D_2^\mot(x \vee y)"]
\end{tikzcd}
\end{equation*}
Under the splitting $D_2^\mot(T^p \vee T^p) \wequi D_2^\mot(T^p) \vee D_2^\mot(T^p) \vee T^{2p}$ of Proposition \ref{prop:D2-mot-cr2}, the composite $S^{2(p+i), p + \ceil{i/2}} \to D_2^\mot(T^p \vee T^p) \wedge H\Z/2 \to D_2^\mot(E) \wedge H\Z/2 \to E$, splits as $\alpha_1 + \alpha_2 + \alpha_3$, where $\alpha_1 = Q^i(x)$ and $\alpha_2 = Q^i(y)$. Moreover $\alpha_3$ factors as $S^{2(p+i), p + \ceil{i/2}} \xrightarrow{\alpha_3'} T^{2p} \wedge H\Z/2 \xrightarrow{xy} E$.
Increasing $p$ by $1$ if necessary, we may assume that $xy=0$ (as at the end of the proof of Proposition \ref{prop:power-ops-vanishing}).
The result follows.
\end{proof}

\subsubsection{Cartan formula} While the operations are not multiplicative, the Cartan formula expresses what the operations do on products. 

\begin{lemma} \label{lemm:BSigma2-diagonal}
Under the op-lax witness transformation $\delta_*: H\Z/2_{**}D_2^\mot(T^{n+m}) \wequi H\Z/2_{**}D_2^\mot(T^n \wedge T^m) \to H\Z/2_{**} D_2^\mot(T^n) \wedge D_2^\mot(T^m)$ we have \[ \delta_*(e_{2i}) = \sum_{a+b=i} e_{2a} \otimes e_{2b} + \sum_{a+b=i+1} \tau e_{2a-1} \otimes e_{2b-1} \] and \[ \delta_*(e_{2i-1}) \sum_{a+b=i} (e_{2a-1} \otimes e_{2b} + e_{2a} \otimes e_{2b-1}) + \sum_{a+b=i+1} \rho e_{2a-1} \otimes e_{2b-1}. \]
\end{lemma}
\begin{proof}
The problem is stable under base change, so we may assume that $S = Spec(\Z[1/2])$.
Compatibility of the op-lax witness transformation with Thom isomorphisms (which holds by inspection of the proof of Proposition \ref{prop:Dmot-oplax-monoidal}\NB{...}) implies that we may reduce to $m=n=0$.
We are now looking at the effect in homology of the diagonal map $\delta: B\Sigma_2 \to B\Sigma_2 \times B\Sigma_2$\NB{why?}, which is by definition dual to the multiplication map in cohomology.
By Lemma \ref{lemm:coh-of-BSigma2} we have $H\Z/2^{**}(B\Sigma_2) \wequi H\Z/2^{**}\fpsr{u, v}/(u^2 = \tau v + \rho u)$, where $e_{2i}$ (by definition) is dual to $v^i$ and $e_{2i+1}$ is dual to $uv^i$.
The result is obtained by a straightforward computation.
\end{proof}

Define an auxiliary spectrum of ``pairs of operations'' \[ \Ops^{(2)} := \lim_{(\Z \times \Z) \times (\Z \times \Z)} (C \wedge C \wedge H\Z/2). \]
The op-lax witness transformations $D_2^\mot(S^{m,n} \wedge S^{p,q}) \to D_2^\mot(S^{m,n}) \wedge D_2^\mot(S^{p,q})$ assemble into a codiagonal map\NB{details?} $\delta: \Ops_{H\Z/2} \to \Ops^{(2)}$ such that for every $(m,n),(p,q) \in (\Z \times \Z) \times (\Z \times \Z)$ the following diagram commutes
\begin{equation} \label{eq:ops2-delta}
\begin{CD}
\Ops_{H\Z/2} @>>> S^{-m-p,-n-q} \wedge D_2^\mot(S^{-m-p,-n-q}) \wedge H\Z/2 \\
@V{\delta}VV            @V{\beta}VV \\
\Ops^{(2)} @>>> D_2^\mot(S^{m,n}) \wedge S^{-p,-q} \wedge D_2^\mot(S^{p,q}) \wedge H\Z/2 \\
\end{CD}
\end{equation}
We also have a canonical map $m: \Ops_{H\Z/2} \wedge \Ops_{H\Z/2} \to \Ops^{(2)}$ induced by the multiplication $H\Z/2 \wedge H\Z/2 \to H\Z/2$. In other words, for every $(m,n),(p,q) \in (\Z \times \Z) \times (\Z \times \Z)$ the following diagram commutes
\begin{equation} \label{eq:ops2-m}
\begin{CD}
\Ops_{H\Z/2} \wedge \Ops_{H\Z/2} @>>> S^{-m,-n} \wedge D_2^\mot(S^{m,n}) \wedge H\Z/2 \wedge S^{-p,-q} \wedge D_2^\mot(S^{p,q}) \wedge H\Z/2 \\
@V{m}VV                                   @VVV \\
\Ops^{(2)} @>>> S^{-m,-n} \wedge D_2^\mot(S^{m,n}) \wedge S^{-p,-q} \wedge D_2^\mot(S^{p,q}) \wedge H\Z/2.
\end{CD}
\end{equation}

\begin{lemma} \label{lemm:Ops2-diagonal}
\begin{enumerate}
\item Suppose that $S$ is essentially smooth over a Dedekind domain. Then $\pi_{**} \Ops^{(2)}$ has a $H\Z/2_{**}$-module basis given by the classes $e_i \otimes e_j := m(e_i \wedge e_j)$ for $i, j \in \Z$.
\item We have \[ \delta_*(e_{2i}) = \sum_{a+b=i} e_{2a} \otimes e_{2b} + \sum_{a+b=i+1} \tau e_{2a-1} \otimes e_{2b-1} \] and  \[ \delta_*(e_{2i-1}) \sum_{a+b=i} (e_{2a-1} \otimes e_{2b} + e_{2a} \otimes e_{2b-1}) + \sum_{a+b=i+1} \rho e_{2a-1} \otimes e_{2b-1}. \]
\end{enumerate}
\end{lemma}
\begin{proof}
(1) is a straightforward adaptation of the techniques of \S\ref{subsec:homotopy-groups-spectra-ops}. (2) follows from Lemma \ref{lemm:BSigma2-diagonal}.\NB{details?}
\end{proof}

\begin{proposition} \label{prop:cartan}
Let $E \in \HNAlg_2(\SH(S))_{H\Z/2/}$ and $x, y \in \pi_{**}(E)$. We have \[ Q^{2i}(xy) = \sum_{j+k=i} Q^{2j}(x)Q^{2k}(y) + \tau\sum_{j+k=i+1} Q^{2j-1}(x)Q^{2k-1}(y) \] and \[ Q^{2i-1}(xy) = \sum_{j+k=i} (Q^{2j-1}(x)Q^{2k}(y) + Q^{2j}(x)Q^{2k-1}(y)) + \rho\sum_{j+k=i+1} Q^{2j-1}(x)Q^{2k-1}(y). \]
\end{proposition}
\begin{proof}
Let $x \in \pi_{m,n} E$ and $y \in \pi_{p,q} E$. Consider the diagram
\begin{equation*}
\begin{CD}
S^{2i-\epsilon,i} \wedge S^{m,n} \wedge S^{p,q} @>{e_{2i-\epsilon}}>> S^{m,n} \wedge S^{p,q} \wedge \Ops_{H\Z/2} @>{\delta}>> S^{m,n} \wedge S^{p,q} \wedge \Ops^{(2)} \\
 @.                                                       @VVV                                                    @VVV                               \\
                                @.                  D_2^\mot(S^{m+p,n+q}) \wedge H\Z/2              @>{\beta}>>  D_2^\mot(S^{m,n}) \wedge D_2^\mot(S^{p,q}) \wedge H\Z/2 \\
 @.                                                       @V{D_2^\mot(xy)}VV                                           @V{D_2^\mot(x)D_2^\mot(y)}VV                 \\
                                @.                  E @= E.
\end{CD}
\end{equation*}
The top square commutes by construction (i.e. the commutativity of \eqref{eq:ops2-delta}), and the bottom square commutes by Lemma \ref{lemm:ur-cartan-formula}. The shortest composite from top left to bottom middle is $Q^{2i-\epsilon}(xy)$. The result now follows from Lemma \ref{lemm:Ops2-diagonal} and commutativity of \eqref{eq:ops2-m}, which determine the composite via the top right hand corner.
\end{proof}

\subsubsection{Cohomological power operations}
Let $E \in \HNAlg_2(\SH(S))_{H\Z/2/}$ and $\scr X \in \Spc(S)_*$. By Construction \ref{cons:map-HNAlg} and Remark \ref{rmk:mapping-action-conormed}, we have $\imap(\scr X, E) \in \Mod_{H\Z/2}^w(\HNAlg_2(\SH(S))$.
\begin{definition} \label{def:cohomological-power-ops}
For $E \in \HNAlg_2(\SH(S))_{H\Z/2/}$ and $\scr X \in \Spc(S)_*$ we define the operation 
\[
\Sq^i: E^{**}(\scr X) \to E^{*+i,*+\floor{i/2}}(\scr X)
\] via $Q^{-i}$ applied to $\pi_{-*,-*} \imap(\scr X, E)$.
\end{definition}

\begin{example} \label{ex:voevodsky-ops}
If $S$ is the spectrum of a field and $E = H\Z/2$, then $\Sq^i$ as constructed above coincides with Voevodsky's motivic Steenrod operation from \cite{voevodsky2003reduced}.
This follows from \cite[Example 7.25]{norms} together with the bistability of our construction.
\end{example}

The properties of the $Q^i$ established so far translate immediately into properties of the $\Sq^i$.
\begin{proposition} \label{prop:cohomological-power-ops-properties}
Let $E \in \HNAlg_2(\SH(S))_{H\Z/2/}$ and $\scr X \in \Spc(S)_*$.
\begin{description}
\item[(1) naturality (I)] If $f: \scr X \to \scr Y \in \Spc(S)_*$ is any map, then the following diagram commutes.
\begin{equation*}
\begin{CD}
E^{**}(\scr Y)              @>{\Sq^i}>> E^{*+i,*+\floor{i/2}}(\scr Y)                   \\
@V{\imap(f, E)_*}VV                      @V{\imap(f, E)_*}VV                             \\
E^{**}(\scr X)              @>{\Sq^i}>> E^{*+i,*+\floor{i/2}}(\scr X)                   \\
\end{CD}
\end{equation*}

\item[(2) naturality (II)] If $\alpha: E \to F \in \HNAlg_2(\SH(S))_{H\Z/2/}$, then the following diagram commutes.
\begin{equation*}
\begin{CD}
E^{**}(\scr X)              @>{\Sq^i}>> E^{*+i,*+\floor{i/2}}(\scr X)                   \\
@V{\imap(\scr X, f)_*}VV                 @V{\imap(\scr X, f)_*}VV                        \\
F^{**}(\scr X)              @>{\Sq^i}>> F^{*+i,*+\floor{i/2}}(\scr X)                   \\
\end{CD}
\end{equation*}

\item[(3) base change] Let $f: S' \to S$ be any morphism of schemes. Then $f^*E \in \HNAlg_2(\SH(S'))_{H\Z/2/}$ (Example \ref{ex:HNAlg2-base-change}) and so $(f^*E)$-cohomology also has operations. Moreover there is a commutative diagram as follows.
\begin{equation*}
\begin{CD}
E^{**}(\scr X)          @>{\Sq^i}>> E^{*+i,*+\floor{i/2}}(\scr X)          \\
@V{f^*}VV                            @V{f^*}VV                              \\
(f^*E)^{**}(f^* \scr X) @>{\Sq^i}>> (f^*E)^{*+i,*+\floor{i/2}}(f^*\scr X)  \\
\end{CD}
\end{equation*}

\item[(4) additivity] For $x, y \in E^{**}(\scr X)$ we have $\Sq^i(x+y) = \Sq^i(x) + \Sq^i(y)$.
\item[(5) squaring] For $x \in E^{2n+\epsilon,n}(\scr X)$ we have $\Sq^{2i+\epsilon}(x) = x^2$.
\item[(6) vanishing] For $x \in E^{p,q}(\scr X)$ with $i > p-q-\epsilon$ and $i \ge q$ we have $\Sq^{2i+\epsilon}(x) = 0$.
\item[(7) unstability] Suppose $E=H\Z/2$ and $S$ is essentially smooth over a Dedekind domain\NB{what is the correct generality?}. Then $\Sq^0=\id$ and $\Sq^i = 0$ for $i<0$.

\item[(8) bistability] The following diagram commutes
\begin{equation} \label{eq:cohomology-operations-bistable}
\begin{CD}
E^{*+1,*}(S^1 \wedge \scr X) @>{\Sq^i}>> E^{*+1+i,*+\floor{i/2}}(S^1 \wedge \scr X)      \\
@|                                            @|                                         \\
E^{**}(\scr X)              @>{\Sq^i}>> E^{*+i,*+\floor{i/2}}(\scr X)                   \\
@|                                            @|                                         \\
E^{*+1,*+1}(\Gm \wedge \scr X) @>{\Sq^i}>> E^{*+1+i,*+1+\floor{i/2}}(\Gm \wedge \scr X),  \\
\end{CD}
\end{equation}
  where the vertical isomorphisms are the canonical ones.

\item[(9) Cartan formula] For $x, y \in E^{**}(\scr X)$, we have \[ \Sq^{2i}(xy) = \sum_{j+k=i} \Sq^{2j}(x)\Sq^{2k}(y) + \tau\sum_{j+k=i-1} \Sq^{2j+1}(x)\Sq^{2k+1}(y) \] and \[ \Sq^{2i+1}(xy) = \sum_{j+k=i} (\Sq^{2j+1}(x)\Sq^{2k}(y) + \Sq^{2j}(x)\Sq^{2k+1}(y)) + \rho\sum_{j+k=i-1} \Sq^{2j+1}(x)\Sq^{2k+1}(y). \]

\item[(10) Bockstein] Suppose $E=H\Z/2$ and $S$ is essentially smooth over a Dedekind domain\NB{what is the correct generality?}. Then $\Sq^1$ is the Bockstein operation, which is a derivation.\NB{$\Sq^{2n+1} = \beta\Sq^{2n}$ then follows from Adem...}
\end{description}
\end{proposition}
\begin{proof}
Properties (1), (2), (4), (5), (6), (8), and (9) follow immediately from the corresponding statements in the previous subsections.

For (3), we first need to explain the map $f^*: E^{**}(\scr X) \to (f^*E)^{**}(f^* \scr X)$. It is the composite \[ E^{**}(\scr X) = \pi_{-*-*}\imap(\scr X, E) \xrightarrow{f^*} \pi_{-*-*} f^*\imap(\scr X, E) \to \pi_{-*-*} \imap(f^* \scr X, f^*E),\] where the last map is obtained by applying $\pi_{-*-*}$ to the map $f^*\imap(\scr X, E) \to \imap(f^* \scr X, f^*E)$ obtained by passing to adjoints and using that $f^*$ is symmetric monoidal. Now (3) is an immediate consequence of Proposition \ref{prop:powerop-base-change}.

(7, 10) We may assume that $S$ is connected. By Theorem \ref{thm:HZ-steenrod-dedekind}, we know that the algebra of bistable cohomology operations is free on the admissible monomials in Steenrod operations. We deduce that $\Sq^0 = \sum_I a_I \Sq^I$, where the sum is over admissible sequences and $a_I \in H\Z/2^{**}$. Note that unless $a_I \in \{0,1\}$, it must have positive weight. Since all $\Sq^I$ (with $I$ admissible) have non-negative weight, and $\Sq^0$ has weight zero, we find that all $a_I \in \{0,1\}$. The only $\Sq^I$ of bidegree $(0,0)$ is $\Sq^\emptyset = \id$ and $\Sq^0 \ne 0$ (e.g. because we know this result over fields), so the first claim follows. The other claims are proved similarly.
\end{proof}

\begin{remark}
The proof of the above result references Theorem \ref{thm:HZ-steenrod-dedekind}, the proof of which in turn references the above result.
Nevertheless it is easy to verify that the argument is not circular.
\end{remark}

\subsection{Co-Nishida relations} \label{sect:co-nishida}
Let $A, X \in \HNAlg_2(\SH(S))_{\1/}$.
Then $A \wedge X \in \HNAlg_2(\SH(S))_{A/}$ and hence carries power operations parametrized essentially by $e \in A_{**} D_2^\mot(S^{p,q})$; we denote them by $\Theta^e: A_{**} X \to A_{**} X$.
We can make $A \wedge A \wedge X$ into an $A$-module in two ways; we call them the right and left $A$-module structures.
We have the unit maps \[ u_A: A \wedge X \wequi \1 \wedge A \wedge X \to A \wedge A \wedge X \] (which is a right $A$-module map) and \[ u_X: A \wedge A \to A \wedge A \wedge X \] (which is a bimodule map).
These induce \[ A_{**} A \otimes_{A_**} A_{**} X \to \pi_{**} (A \wedge A \wedge X) \] which is an isomorphism under suitable flatness assumptions.
If we want $u_A$ to commute with power operations, then we have to use the \emph{right} $A$-module structure on $A \wedge A \wedge X$, and if we want $u_X$ to preserve power operations, then we have to use the right $A$-module structure on $A \wedge A$.
\emph{For this reason, whenever forming power operations on $A \wedge A$ or $A \wedge A \wedge X$ we will always use the right $A$-module structures.}
We further have the coaction map \[ \psi: A \wedge X \wequi A \wedge \1 \wedge X \to A \wedge A \wedge X; \] this is a left $A$-module map.
In particular, it does not commute with the power operations.
The Nishida relations account for this fact.
We proceed to prove them trough a series of lemmas.

\begin{lemma} \label{lemm:nishida-1}
Let $A \in \HNAlg_2(\SH(S))_{\1/}$ and $E \in \HNAlg_2(\SH(S))_{A/}$.
For $e, e' \in A_{**}D_2^\mot(S^{p,q})$, $a \in A_{**}$ and $x \in \pi_{p,q}E$, we have \[ \Theta^{ae}(x) = a \Theta^e(x) \] and \[ \Theta^{e + e'}(x) = \Theta^e(x) + \Theta^{e'}(x). \]
\end{lemma}
\begin{proof}
Immediate from the definition (i.e. Diagram \eqref{eq:power-map}).
\end{proof}

\begin{lemma} \label{lemm:nishida-2}
Let $\psi: A \to B \in \HNAlg_2(\SH(S))_{\1/}$, $X \in \HNAlg_2(\SH(S))_{\1/}$, $e \in A_{**} D_2^\mot(S^{p,q})$ and $x \in A_{p,q} X$.
Then \[ \psi \Theta^e(x) = \Theta^{\psi(e)}(x). \]
\end{lemma}
\begin{proof}
Consider the following diagram
\begin{equation*}
\begin{CD}
A \wedge D_2^\mot(S^{p,q}) @>x>> A \wedge D_2^\mot(A \wedge E) @>m>> A \wedge A \wedge E @>m>> A \wedge E \\
@V{\psi}VV                    @V{\psi \wedge \id}VV          @V{\psi \wedge \id}VV   @V{\psi}VV \\
B \wedge D_2^\mot(S^{p,q}) @>x>> B \wedge D_2^\mot(A \wedge E) @>m>> B \wedge A \wedge E   @.   B \wedge E \\
    @.                        @V{\id \wedge \psi}VV          @V{\id \wedge \psi}VV   @| \\
                        @.  B \wedge D_2^\mot(B \wedge E) @>m>> B \wedge B \wedge E @>m>> B \wedge E.
\end{CD}
\end{equation*}
The left hand square commutes trivially, and all other squares commute because $\psi$ is a map in $\HNAlg_2(\SH(S))$.
The result follows by chasing a homotopy class from the top left hand corner to bottom right, either via top right corner or along the bottom-most edges.
\end{proof}

\begin{corollary}[Co-Nishida relations, hands-on form] \label{cor:conishida}
Let $A, X \in \HNAlg_2(\SH(S))_{1/}$, $e \in A_{**}D_2^\mot(S^{p,q})$.
Assume that $A$ satisfies suitable flatness hypotheses and write \[ \psi(e) = \sum_i \theta_i \otimes e_i \in A_{**}A \otimes_{A_{**}} A_{**} D_2^\mot(S^{p,q}). \]
Let $x \in A_{p,q} X$.
Then \[ \psi \Theta^e(x) = \sum_i (\theta_i \otimes 1) \Theta^{e_i}(\psi(x)) \in A_{**}A \otimes_{A_**} A_{**} X. \]
Here $\psi$ denotes the coaction map(s) and $\Theta^{e_i}$ refers to the power operation on $A \wedge A \wedge X$ (coming from the right $A$-module structure).

A similar result holds for $x \in A^{p,q} X$ where $X$ is a space, and using the cohomology coaction (see \S\ref{subsub:cohomology-coaction}).
\end{corollary}
\begin{proof}
We prove the result for homology; essentially the same argument works for cohomology.\todo{does it?}
By Lemma \ref{lemm:nishida-2}, we have $\psi \Theta^e(x) = \Theta^{\psi(e)}(\psi(x))$, where $\Theta^{\psi(e)}$ refers to the power operation arising from the structure as an $A \wedge A$-algebra.
We claim that for any $\alpha \in \pi_{p,q}(A \wedge A \wedge X)$ we have $\Theta^{1 \otimes e_i}(\alpha) = \Theta^{e_i}(\alpha)$, where the right hand $\Theta$ refers to the power operations arising from the right $A$-module structure.
This will establish the desired result, by Lemma \ref{lemm:nishida-1}.

To establish the claim, consider the diagram
\begin{equation*}
\begin{CD}
A \wedge D_2^\mot(S^{p,q}) @>{\alpha}>> A \wedge D_2^\mot(A \wedge A \wedge X) @>m>> A \wedge A \wedge A \wedge X @>{m_{13}}>> A \wedge A \wedge E \\
@|                                     @|                                     @|                                          @| \\
\1 \wedge A \wedge D_2^\mot(S^{p,q}) @>{\alpha}>> \1 \wedge A \wedge D_2^\mot(A \wedge A \wedge X) @>m>> \1 \wedge A \wedge A \wedge A \wedge X @>{m_{13}}>> A \wedge A \wedge E \\
@VuVV                                         @VuVV                                            @VuVV                                     @| \\
A \wedge A \wedge D_2^\mot(S^{p,q}) @>{\alpha}>> A \wedge A \wedge D_2^\mot(A \wedge A \wedge X) @>m>> A \wedge A \wedge A \wedge A \wedge X @>{m_{13,24}}>> A \wedge A \wedge E \\
\end{CD}
\end{equation*}
Here $m_{13}$ multiplies the first $A$ from the left into the third, whereas $m_{13,24}$ multiplies the first into the third and the second into the fourth.
The bottom right hand square commutes because $u$ is the unit in $A$; all other squares commute for trivial reasons.
The top row implements the power operation in $A \wedge A \wedge X$ considered as a right $A$-module and the bottom row implements the power operation in $A \wedge A \wedge X$ considered as a $A \wedge A$-module.
The left hand map induces in homotopy groups the map $A_{**} D_2^\mot(S^{p,q}) \to A_{**}A \otimes_{A_{**}} A_{**} D_2^\mot(S^{p,q})$ given by $e \mapsto (1 \otimes e)$.
The result follows.
\end{proof}

\begin{corollary}[Co-Nishida relations] \label{cor:cohere} Let $1/2 \in S$.
Consider the right comodule structure \[\psi_R: \pi_{**}(\Ops_{\H\Z/2}) \rightarrow \pi_{**}(\Ops_{\H\Z/2}) \widehat{\otimes}_{\H\Z/2_{**}} \H\Z/2_{**}\H\Z/2 \] of \eqref{eq:right}, which we write as $\psi_R(e) = \SSigma e_i \otimes r_i$.

Let $E \in \HNAlg_2(\SH(S))_{\1/}$. Then we have the following equality for any $e \in \pi_{**}(\Ops_{\H\Z/2})$ and $x \in \H\Z/2_{**} E$\NB{Should there be antipodes somewhere?}
\begin{equation} \label{eq:psi-v-q}
\psi_R(Q^e(x)) = \underset{i}{\SSigma} (1 \otimes r_i)(Q^{e_i}(\psi_R(x)) \in \H\Z/2_{**} E \otimes_{\H\Z/2_{**}} \H\Z/2_{**} \H\Z/2. 
\end{equation}
Similarly for cohomology.
\end{corollary}
\begin{proof}
Immediate from Corollary \ref{cor:conishida} via Remark \ref{rmk:ops-comodule-compatibility}.
\end{proof}

Recall the generating functions $\xi(t), \tau(t)$ from \S\ref{sect:formulas}.
\begin{theorem} \label{thm:conishida} Let $2$ be invertible in $S$. Then we have the following equality
\begin{equation*}
\SSigma_{n \in \Z, \epsilon \in \{0,1\}} \psi_R(Q^{2n+\epsilon}x)t^ns^\epsilon = \SSigma_{n \in \Z} \overline{\xi}(t)^n \left[ Q^{2n}(\psi_R(x)) + (\overline{\tau}(t) + s)Q^{2n+1}(\psi_R(x)) \right] 
\end{equation*}
\end{theorem}
\begin{proof}
Immediate by combining Corollary~\ref{cor:cohere} and Corollary~\ref{cor:psil} (recalling that switching left and right coactions involves an antipode).
\end{proof}

\subsection{Action on the dual Steenrod algebra}
In this section we describe the action of the power operations on the dual Steenrod algebra. We make use of the generating functions $\xi(t), \tau(t)$ from \S\ref{sect:formulas}.

We begin with some pedestrian observations.
\begin{lemma} \label{lemm:action-on-Bsigma2}
In the notation of Lemma \ref{lemm:coh-of-BSigma2}, the action of power operations on $H^{**}(B\Sigma_2, \Z/2)$ is given by
\begin{equation*}
  Q^{2n+\epsilon} u = \begin{cases} u & n = \epsilon = 0 \\ v & n = -1, \epsilon = 1 \\ 0 & \text{else} \end{cases}
\end{equation*}
and
\begin{equation*}
  Q^{2n+\epsilon} v^{2^i} = \begin{cases} v^{2^i} & n = \epsilon = 0 \\ v^{2^{i+1}} & n = -2^i, \epsilon = 0 \\ 0 & \text{else.} \end{cases}
\end{equation*}
\end{lemma}
\begin{proof}
By definition $v = \beta u$, which is the same as $Q^{-1} u$ by Proposition \ref{prop:cohomological-power-ops-properties}(10). The other power operations on $u$ vanish by Proposition \ref{prop:cohomological-power-ops-properties}(6, 7).
We have $0 = \beta v = \beta^2 u$, so $Q^{-1} v = 0$. Also $Q^{-2} v = v^2$ by Proposition \ref{prop:squaring}. The remaining operations on $v$ are determined by Proposition \ref{prop:cohomological-power-ops-properties}(6, 7), and the operations on $v^{2^i}$ follow by the Cartan formulas (Proposition \ref{prop:cohomological-power-ops-properties}(9)).
\end{proof}

\begin{theorem} \label{thm:bockstein-generates}
The smallest $H\Z/2_{**}$-subalgebra of $H\Z/2_{**}H\Z/2$ containing $\tau_0$ and stable under the genuine motivic power operations $Q^i, i \in \Z$ is $H\Z/2_{**}H\Z/2$ itself.
\end{theorem}
\begin{proof}
Apply Theorem \ref{thm:conishida} to the right coaction on $x := u \in H\Z/2^{**}(B\Sigma_2)$, substitute $t = \xi(t)$ and use Lemmas \ref{lem:comp-inv} and \ref{lemm:action-on-Bsigma2} to get

\[ 
\psi_R(u) + \psi_R(v)\xi(t)^{-1}s = \SSigma_{n \in \Z} t^n \left[ Q^{2n}(\psi_R(u)) + (\tau(t) + s)Q^{2n+1}(\psi_R(u)) \right]. 
\]

This is an equation in a free module with basis $\{t^nv^ms^\epsilon u^\delta\}$.
We shall compare the coefficients of $t^{2^i-1}v$ for $i \ge 1$; on the left hand side we get zero.
On the right hand side we get 
\begin{equation} \label{eq:generation} 
	\left[Q^{2(2^i - 1)}(\psi_R(u))\right]_{v} + \left[\sum_{n \in \Z} \sum_{k \ge 0} \tau_k t^{2^k+n} Q^{2n+1}(\psi_R(u))\right]_{t^{2^i-1}v}. 
\end{equation}
Note that in order to obtain the coefficient $t^{2^i-1}$, $n$ is determined by $k$ as $n=2^i-2^k-1$.
We now plug in $\psi_R(u)$ (from Lemma \ref{lem:psir-easy}) and simplify the power operations.
For this we use two observations.
Firstly, if $r$ is odd, then $Q^r(u) = 0$ except in the special case $r=-1$, when $Q^{-1}(u) = v$.
Secondly, since for odd $r$ we have $Q^r(v^{2^m}) = 0$, and for even $r$ there are only two non-zero power operations (namely $Q^0(v^{2^m}) = v^{2^m}$ and $Q^{-2^{m+1}}(v^{2^m}) = v^{2^{m+1}}$, the Cartan formula (Proposition \ref{prop:cartan}) shows that
\[ 
Q^r(y v^{2^m}) = Q^r(y) v^{2^m} + Q^{r+2^{m+1}}(y) v^{2^{m+1}}. 
\]
Combining these shows that
\[ 
[Q^{2(2^i-1)}(\psi_R(u))]_v = Q^{2(2^i-1)}(\tau_0)
\]
and 
\[ \left[\sum_{n \in \Z} \sum_{k \ge 0} \tau_k t^{2^k+n} Q^{2n+1}(\psi_R(u))\right]_{t^{2^i-1}v} = \sum_{k \ge 0} \tau_k [Q^{2(2^i-2^k-1)+1}(\psi_R(u))]_v = \tau_i + \sum_{k \ge 0}\tau_k Q^{2(2^i-2^k-1)+1}(\tau_0). \]
By the vanishing of Proposition \ref{prop:power-ops-vanishing}, the second sum is zero for $k \ge i$.
Putting everything together we thus find that 
\[ 
Q^{2(2^i-1)}(\tau_0) + \tau_i + \sum_{k=0}^{i-1} \tau_k Q^{2(2^i-2^k-1)+1}(\tau_0) = 0. 
\]
This allows us to express $\tau_i$ in terms of power operations on $\tau_0$ and the $\tau_k$ for $k<i$.
Consequently the subalgebra of $H\Z/2_{**}H\Z/2$ generated by $\tau_0$ and power operations contains all the $\tau_i$.

Now we again consider equation \eqref{eq:generation}, and this time we compare the coefficient of $t^{-1}s$.
Since the leading coefficient of $\xi(t)$ is $t$, the coefficient of $t^{-1}s$ on the left hand side is just $\psi_R(v)$.
On the other hand on the right hand side we just obtain $Q^{-1}\psi_R(u)$.
Thus we find \[ \sum_{i \ge 0} \xi_i v^{2^i} = Q^{-1}\left(u + \sum_{i \ge 0} \tau_i v^{2^i}\right) = v + \sum_{i \ge 0}\left(Q^{-1}(\tau_i) v^{2^i} + Q^{2^{i+1}-1}(\tau_i) v^{2^{i+1}}\right). \]
This shows that the $\xi_i$ can be obtained via power operations from the $\tau_i$, concluding the proof.
\end{proof}

Using similar arguments, it is in fact possible to determine the entire action of power operations on the dual Steenrod algebra.

\begin{theorem} \label{thm:power-ops-on-dual-steenrod}
The power operations act on $H\Z/2_{**}H\Z/2$ as follows:
\begin{align*}
	t^{2^n} \sum_r Q^{2r}(\tau_n) t^r &= \sum_{i \leq n} \tau_i t^{2^i} + \tau(t) \xi(t)^{-1} \sum_{i\leq n} \xi_i t^{2^i} \\
	t^{2^n}\sum_r Q^{2r+1}(\tau_n) t^{r} &= \xi(t)^{-1} \sum_{i\leq n} \xi_i t^{2^i}  + 1\\
	t^{2^n} \sum_r Q^{2r}(\xi_n) t^r &= \sum_{i \leq n} \xi_i t^{2^i} + \xi(t)^{-1} \sum_{i < n}\xi_i^2 t^{2^{i+1}} \\
	\sum_r Q^{2r+1}(\xi_n)t^{r} &= 0
\end{align*}
\end{theorem}
\begin{proof}
Applying Theorem \ref{thm:conishida} to the right coaction on $H^{**}(B\Sigma_2, \Z/2)$ we obtain the formula 
\[ 
\sum_{n, \epsilon} \psi_R(Q^{2n+\epsilon}(u))t^ns^\epsilon = \sum_n \overline{\xi}(t)^n \left[ Q^{2n}(\psi_R(u)) + (\overline{\tau}(t) + s)Q^{2n+1}(\psi_R(u)) \right]. 
\]
This equation takes place in a free module with basis $\{u^{\delta}v^mt^ns^\epsilon \}$ and we will compare coefficients on various basis elements.
Using Lemma \ref{lemm:action-on-Bsigma2} and  Lemma \ref{lem:psir-easy} the left hand side becomes 
\[ 
\psi_R(u) + \psi_R(v) t^{-1}s = (u +\sum_{i\geq 0}\tau_iv^{2^i}) + (\sum_{i\geq 0}\xi_iv^{2^i})t^{-1}s. 
\]
To expand the right hand side we use Lemma \ref{lem:psir-easy} and
 the Cartan formula (Proposition \ref{prop:cartan}) which shows that
\[ 
Q^r(y v^{2^m}) = Q^r(y) v^{2^m} + Q^{r+2^{m+1}}(y) v^{2^{m+1}}. 
\]
Plugging in $t = \xi(t)$ and applying Lemmas \ref{lem:comp-inv} and \ref{lemm:action-on-Bsigma2} we expand the right hand side as
\begin{align*}
\sum_{r\in \Z}t^r\Bigg[ Q^{2r}(u) + \sum_{i\geq 0} \Big( Q^{2r}(\tau_i)v^{2^i} &+ Q^{2r+2^{i+1}}(\tau_i)v^{2^{i+1}} \Big)
	\\
	&+(\tau(t) + s) \Big( Q^{2r+1}(u) + \sum_{i\geq 0} \big( Q^{2r+1}(\tau_i)v^{2^i} + Q^{2r+2^{i+1}+1}(\tau_i)v^{2^{i+1}}  \big)  \Big)\Bigg].
\end{align*}
Note that the only nonzero value of $Q^{2r}(u)$ is $u$ when $r=0$ and the only nonzero value of $Q^{2r+1}(u)$ is $v$ when $r=-1$. 
Now comparing the coefficients on $v^{2^i}$, on the left and right hand sides, we obtain for  the equality for $i=0$
\[
\tau_0 + \xi(t)^{-1}s = \sum_{r\in \Z}t^r\left[ Q^{2r}(\tau_0) + (\tau(t)+s) \left( Q^{2r+1}(\tau_i)\right)\right] + (\tau(t) + s)t^{-1}
\]
and for $i\geq 1$ we obtain
\[
\tau_i + \xi_i\xi(t)^{-1}s = \sum_{r\in \Z}t^r\left[ Q^{2r}(\tau_i) +t^{-2^{i-1}}Q^{2r}(\tau_{i-1}) + (\tau(t)+s) \left(Q^{2r+1}(\tau_i) + t^{-2^{i-1}} Q^{2r+1}(\tau_{i-1})\right)\right].
\]
Finally, comparing the terms on $s^{\epsilon}$ for $\epsilon\in \{0,1\}$ we find the 
 the following recurrence relations determine the power operations on the $\tau_i$:\tombubble{I simplified using $\xi_0=1$}
\begin{align*}
	\sum_{r\in\Z}Q^{2r+1}(\tau_n)t^r 
	&= t^{-2^{n-1}}\sum_{r\in \Z}Q^{2r+1}(\tau_{n-1})t^{r} +\xi_{n}\xi(t)^{-1}, \hspace{3mm} n\geq 1
	\\
	\sum_{r\in\Z} Q^{2r}(\tau_n)t^r
	& = \tau_n + t^{-2^{n-1}}\sum_{r\in \Z}Q^{2r}(\tau_{n-1}) t^r \\
	&+ \tau(t)(\sum_{r\in\Z}Q^{2r+1}(\tau_n)t^{r} + t^{-2^{n-1}}\sum_{r\in\Z}Q^{2r+1}(\tau_{n-1})t^r ), \hspace{3mm} n\geq 1
	\\
	\sum_{r\in\Z}Q^{2r+1}(\tau_0)t^{r} &= \xi(t)^{-1} +t^{-1}, 
	\\
	\sum_{r\in\Z}Q^{2r}(\tau_0)t^{r} & = \tau_0+\tau(t)\xi(t)^{-1}.
\end{align*}
Repeating the same procedure with $v$ in place of $u$ in the co-Nishida relations yields the following recurrence relations determining the power operations on the $\xi_i$: 
\begin{align*}
	\sum_{r\in\Z}Q^{2r+1}(\xi_n)t^r 
	&= t^{-2^{n-1}}\sum_{r\in\Z}Q^{2r+1}(\xi_{n-1})t^{r}, \hspace{3mm} n\geq 1
	\\
	\sum_{r\in\Z}Q^{2r}(\xi_n)t^r 
	&= \xi_n + [\xi(t)^{2}]_{t^{2^n}}\xi(t)^{-1} + t^{-2^{n-1}}\sum_{r\in\Z}Q^{2r}(\xi_{n-1})t^r \\ 
	+&\tau(t)( \sum_{r\in\Z}Q^{2r+1}(\xi_n)t^r  +t^{-2^{n-1}}\sum_{r\in\Z}Q^{2r+1}(\xi_{n-1})t^r), \hspace{3mm} n\geq 1 \\
	\sum_{r\in\Z}Q^{2r+1}(\xi_0)t^r &= 0,  \\
	\sum_{r\in\Z}Q^{2r}(\xi_0)t^r &= \xi_0.
\end{align*}
It is straightforward to verify by induction that these recurrence relations are solved (uniquely) by our stated expressions.
\end{proof}

\begin{corollary}\label{cor:power-ops-on-dual-steenrod}
For $k>0$ we have $Q^{2^{k+1} - 2} \tau_0 = \overline\tau_k$ and $Q^{2^{k+1} - 3} \tau_0 = \overline \xi_k$.
\end{corollary}
\begin{proof}
Applying the first formula of Theorem \ref{thm:power-ops-on-dual-steenrod} with $n=0$ yields \[ \sum_r Q^{2r}(\tau_0) t^r = \tau_0 + \tau(t) \xi(t)^{-1}, \] whence \[ Q^{2^{k+1} - 2} \tau_0 = [\tau(t) \xi(t)^{-1}]_{t^{2^k - 1}}. \]
By Lemma \ref{lemm:inversion-trick}, this is the same as $[\overline\tau(t)\overline\xi(t)^{-2^k}]_{t^0}$, which implies the first desired formula. The second one follows by applying the Bockstein.
\end{proof}

\begin{proof}[Alternative proof of Theorem \ref{thm:bockstein-generates}.]
Since $H\Z/2_{**}H\Z/2$ is generated by the $\tau_i$ and $\xi_i$, and the antipode is an idempotent ring homomorphism, $H\Z/2_{**}H\Z/2$ is also generated by the $\overline{\tau_i}$ and $\overline{\xi_i}$. By Corollary \ref{cor:power-ops-on-dual-steenrod}, $H\Z/2_{**}H\Z/2$ is thus generated as a ring with power operations by $\tau_0$. This was to be shown.
\end{proof}

\begin{subappendices}
\section{Dualization of split Tate motives}
\label{app:dualization}

\begin{assumption} \label{ass:A}
Throughout we fix $A \in \CAlg(\SH(S))$ such that, for every $X \in \Sm_S$, we
have $[X_+, A
\wedge \Gmp{n} \wedge S^m] = 0$ for $m$ sufficiently large or $n$ sufficiently
small (depending on $X$).
\end{assumption}
\begin{example} \label{ex:dualization-dedekind}
This applies in
particular if $A = \H\Z$ or $A = \H\Z/p$ and $S$ is essentially smooth over a Dedekind domain or field (in particular, finite dimensional); see Proposition \ref{prop:HZ-dedekind}. We will work throughout in the category $\Mod_A$.
\end{example}

\begin{lemma} \label{lemm:tate-sum-product}
Let $a_\bullet, b_\bullet: \Z \to \Z \cup \{\pm \infty\}$ be two sequences such that
$\lim_{n \to\infty} a_n = \infty$ and $\lim_{n \to -\infty} b_n = -\infty$.
Then the canonical map
\[ \bigoplus_{n \in \Z} A \wedge S^{a_n} \wedge \Gmp{b_n} \to
      \prod_{n \in \Z} A \wedge S^{a_n} \wedge \Gmp{b_n} \]
is an equivalence. Here $S^\infty := 0$ and $\Gmp{-\infty} := 0$.
\end{lemma}
\begin{proof}
Since $\Mod_A$ is generated by $A \wedge X_+ \wedge S^{m,n}$ for $X \in
\Sm_S$, it suffices to show that
\[ [S^{m,n} \wedge X_+, \bigoplus_{n} A \wedge S^{a_n} \wedge \Gmp{b_n}] \wequi
   [S^{m,n} \wedge X_+, \prod_{n} A \wedge S^{a_n} \wedge \Gmp{b_n}]. \]
Since adding a constant to a sequence does not change the infinite limit, and tensoring
with the invertible object $S^{-m,-n}$ preserves arbitrary sums and products
(being an equivalence), we may assume that $m=n=0$. Since $X_+$ is compact, we
need to show that
\[ \bigoplus_n [X_+, A \wedge S^{a_n} \wedge \Gmp{b_n}] \wequi
   \prod_n [X_+, A \wedge S^{a_n} \wedge \Gmp{b_n}]. \]
In other words, we need to show that only finitely many of the groups $[X_+,
A \wedge S^{a_n} \wedge \Gmp{b_n}]$ are non-zero. But this group is zero as
soon as $a_n >> 0$ or $b_n << 0$, which by assumption
holds for $n$ sufficiently large or small, respectively.
\end{proof}

\begin{remark}
For a related result, see \cite[Lemma 8.7]{totaro2016motive}.
\end{remark}

\begin{definition}
We call an object $M \in \Mod_A$ \emph{degreewise pseudofinite split
Tate\NB{possibly not the best terminology...}} if $M \wequi \bigoplus_{n \in
\Z} A \wedge S^{a_n} \wedge \Gmp{b_n}$, where $a_\bullet, b_\bullet: \Z \to
\Z \cup \{\pm \infty\}$ are two sequences with $\lim_{n \to \pm \infty} a_n =
\pm \infty = \lim_{n \to \pm \infty} b_n$. Here we put $S^{\pm \infty} = 0 =
\Gmp{\pm \infty}$.
\end{definition}

\begin{example}
Let $X_n \in \Mod_A$ be a sequence of split Tate motives, where each $X_n$
is a \emph{finite} sum of $S^{a,b} \wedge A$-s. Then $\bigoplus_{n=0}^\infty X_n$ is
is degreewise pseudofinite if both the connectivity and the
slice-connectivity (weight) of $X_n$ tend to infinity as $n \to \infty$.
\end{example}

\begin{corollary} \label{corr:tate-reflexive}
If $M \in \Mod_A$ is degreewise pseudofinite split Tate, then $\imap_{A}(M,
A)$ is also degreewise pseudofinite split Tate, and $M$ is
reflexive: the canonical map $M \to \imap_{A}(\imap_{A}(M,
A), A)$ is an equivalence.
\end{corollary}
\begin{proof}
Let $M \wequi \bigoplus_{n} A \wedge S^{a_n} \wedge \Gmp{b_n}$. Then
\[ \imap_{A}(M, A) \wequi \prod_{n} A \wedge S^{-a_n} \wedge \Gmp{-b_n}
     \wequi \prod_{n} A \wedge S^{-a_{-n}} \wedge \Gmp{-b_{-n}} . \]
By Lemma \ref{lemm:tate-sum-product}, this infinite product is equivalent to the
corresponding infinite sum; so $\imap_{A}(M, A)$ is split Tate, and
degreewise pseudofinite. The refexivity follows by applying the same argument
once more.
\end{proof}

\begin{remark}
In a closed symmetric monoidal category, it need \emph{not} be the case that every
reflexive object is strongly dualizable. Indeed if the category is
compact-rigidly generated, then dualizable objects are compact. But split Tate
motives are compact only if the corresponding sum is finite, which is not
assumed in Corollary \ref{corr:tate-reflexive}.
\end{remark}

If $M_{**}, N_{**}$ are bigraded $A_{**}$-modules, then $\Hom_{**}(M_{**}, N_{**})$
is the subset of those $A_{**}$-linear homomorphisms which are finite sums
of homogeneous homomorphisms (see also \cite[first paragraph of Section
12]{voevodsky2003reduced}). This is itself a bigraded $A_{**}$-module.
Write $\Mod_{A_{**}}$ for the category of bigraded $A_{**}$-modules. This
has an obvious tensor product, with right adjoint given by $\Hom_{**}$.

If $M \in \Mod_A$, then we obtain $M_{**}, M^{**} \in \Mod_{A_{**}}$, where
$M_{ab} = [A \wedge S^{a,b}, M]$ and $M^{ab} = [M, A \wedge S^{-a,-b}]$.
Composition defines a pairing $M_{**} \times M^{**} \to A_{**}$ which is
bihomogenous and bilinear, so induces $M_{**} \otimes_{A_{**}} M^{**} \to
A_{**}$.

\begin{lemma} \label{lemm:mot-hom-cohom-dual}
Let $M \in \Mod_A$.
\begin{enumerate}
\item If $M$ is split Tate, then the canonical map $M^{**} \to \Hom_{**}(M_{**},
  A_{**})$ is an equivalence.
\item If $M$ is degreewise pseudofinite split Tate, then the canonical map
  $M_{**} \to \Hom_{**}(M^{**}, A_{**})$ is an equivalence.
\end{enumerate}
\end{lemma}
\begin{proof}
We write $D$ for the weak dual functor.

Consider the functor $\mathcal{M}: \Mod_A \to \Mod_{A_{**}}, M \mapsto
M_{**}.$ This functor preserves sums and products. By direct computation, it
preserves tensor products of split Tate motives, and duals of $A \wedge
S^{m,n}$. It follows formally that $\mathcal{M}$ preserves weak duals of split Tate
motives. Noting that $M^{**} = \mathcal{M}(DM)$, we get (1). For (2), we note
that
\[ M_{**} = \mathcal{M}(M) \wequi \mathcal{M}(DDM) \wequi D(\mathcal{M}(DM)) =
      D(M^{**}), \]
where the two equivalences hold because $DM$ is split Tate and $DDM \wequi M$,
both by Corollary \ref{corr:tate-reflexive}. This concludes the proof.
\end{proof}

\begin{example}
Taking $A = \H\Z/2$, $M = \H\Z/2 \wedge \H\Z/2$ we obtain $M^{**} = \scr A^{**}$, the motivic
Steenrod algebra, and $M_{**} = \scr A_{**}$, the dual motivic Steenrod algebra. By
Lemma \ref{lemm:mot-hom-cohom-dual} we get $\scr A_{**} \wequi \Hom_{**}(\scr A^{**},
\H\Z/2_{**})$ and $\scr A^{**} \wequi \Hom_{**}(\scr A_{**}, \H\Z/2_{**})$, as was already
proved by Voevodsky in the case of a field \cite[paragraphs before Proposition 12.1]{voevodsky2003reduced}.
\end{example}

\section{Review of the eightfold way} \label{sec:8fold}

\subsection{Generalities on comodule structures} \label{sec:comodule-structures}
We follow the discussion in \cite[\S 2]{baker2013power} in a suitable generality. Suppose that we have an additive, symmetric monoidal 1-category $\scr C^{\otimes}$ with unit object $\1$. We obtain a \emph{switch transformation}
\[
\sw_{A,B}:  A \otimes B \Rightarrow B \otimes A,
\]
natural in $A$ and $B$.

Let $E, F \in \CAlg(\scr C)$ and let $X \in \scr C$. Then we have the following commutative diagram
\begin{equation} \label{eq:coact-c}
\begin{tikzcd}
&   & E \otimes X \ar[swap, "\id \otimes \eta_F \otimes \id"]{dl} \ar{dr}{\sw_{E,X}} & \\
& E \otimes F \otimes X \ar["\sw_{E, F \otimes X}"]{dr} \ar[swap, "\sw_{E,F}"]{dl}&  & X \otimes E \ar{dl}{\eta_F \otimes \id \otimes \id}\\
F\otimes E \otimes X \ar{rr}{\sw_{E, X}} & & F \otimes X \otimes E.& \\
\end{tikzcd}
\end{equation}

Now, suppose that $\scr T \subset \scr C^\wequi$ is the full subcategory on a set of $\otimes$-invertible objects in $\scr C$ which is closed under $\otimes$ and taking $\otimes$-inverses. Then $\scr T$ is a grouplike symmetric monoidal groupoid, also known as a \emph{Picard groupoid}. For any $X \in \scr C$, we have the functor \[ \pi_\scr{T}X: \scr T \to Ab, T \mapsto [T, X]. \] This is an example of a Picard-graded object \cite[\S 2.1]{GepnerLawson}.
If $F \in \CAlg(\scr C)$, then $F_{\scr T}:=\pi_{\scr T}(F)$ acquires the structure of $\scr T$-graded ring and if $X$ is an $F$-module then $\pi_{\scr T}X$ is naturally a $\scr T$-graded $F_{\scr T}$-module. For any $X \in \scr C$, we put $F_{\scr T}X:= \pi_{\scr T}(F \otimes X)$, which is an $F_\scr{T}$-module by what we just said (since $F \otimes X$ is an $F$-module).

\begin{definition} \label{defn:t-flat}
We say that an $F$-module $X$ is \emph{$\scr T$-flat} if for every $Y \in \scr C$, the canonical map $\pi_\scr{T}(X) \otimes_{F_\scr{T}} F_\scr{T}(Y) \to \pi_\scr{T}(X \otimes Y)$ (coming from the action of $F$ on $X$) is an isomorphism.
\end{definition}

In the situation of~\eqref{eq:coact-c}, assume that $E \otimes F$ is $\scr T$-flat as an $F$-module. Then we get isomorphisms of $F_{\scr T}$-modules
\begin{enumerate}
\item $\pi_{\scr T}(F \otimes X \otimes E) \cong \pi_{\scr T}(F \otimes X) \otimes_{\pi_{\scr T}(F)} \pi_{\scr T}(F \otimes E) = F_{\scr T}X \otimes_{F_{\scr T}} F_{\scr T}E$, and
\item $\pi_{\scr T}(E \otimes F \otimes X) \cong \pi_{\scr T}(E \otimes F) \otimes_{\pi_{\scr T}(F)} \pi_{\scr T}(F \otimes X) = E_{\scr T}F \otimes_{F_{\scr T}} F_{\scr T}X$.
\end{enumerate}
Furthermore, the map $\sw: E \otimes F \otimes X \rightarrow F \otimes X \otimes E$ induces a further isomorphism
\[
E_{\scr T}(F \otimes X) \stackrel{\cong}{\rightarrow} F_{\scr T}X \otimes_{F_{\scr T}} F_{\scr T}E.
\]
Whence the diagram~\eqref{eq:coact-c} gives rise to a diagram of $\scr T$-graded rings
\begin{equation} \label{eq:coact-t}
\begin{tikzcd}
  &  \ar{dl} E_{\scr T}X \ar{dr} & \\
E_{\scr T}F \otimes_{F_{\scr T}} F_{\scr T}X  &  &  F_{\scr T}X \otimes_{F_{\scr T}} F_{\scr T}E \\
& \ar[swap, "\cong"]{ul} E_{\scr T}(F \otimes X) \ar{ur}{\cong} & .
\end{tikzcd}
\end{equation}

In the case that $E = F$, the switch map $\sw: E \rightarrow E$ induces an involution
\[
\chi_E: E_{\scr T}E \rightarrow E_{\scr T}E,
\]
which is called the \emph{antipode}. We have the diagram
\begin{equation} \label{eq:coact-t-e-f}
\begin{tikzcd}
  &  \ar[swap, "\psi_L"]{dl}{\psi_L} E_{\scr T}X \ar{dr}{\psi_R} & \\
E_{\scr T}E \otimes_{E_{\scr T}} E_{\scr T}X  &  &  E_{\scr T}X \otimes_{E_{\scr T}} E_{\scr T}E \\
& \ar[swap, "\cong"]{ul} E_{\scr T}(E \otimes X) \ar{ur}{\cong} & .
\end{tikzcd}
\end{equation}
We call $\psi_L$ (resp. $\psi_R$) the \emph{left (resp. right) $E_{\scr T}E$-comodule structure} structure on $E_{\scr T}X$.

\begin{lemma} \label{lem:antipode} In the notation of~\eqref{eq:coact-t-e-f}, $\psi_L = \psi_R \circ \sw_{E,X*} \circ \chi$. Consequently, if $\psi_R(x) = \SSigma x_i \otimes e_i$, then 
\[
\psi_L(x) = \SSigma e_i \otimes \chi(x_i).
\]
Dually, if $\psi_L(x) = \SSigma e_i \otimes x_i$ then
\[
\psi_R(x) = \SSigma \chi(x_i) \otimes e_i. 
\]
\end{lemma}
\begin{proof}
Follows from commutativity of the bottom left triangle in~\eqref{eq:coact-c}.
\end{proof}

The above constructions furnish a Hopf algebroid structure on $(E_\scr{T}, E_\scr{T} E)$. If $A \in \CAlg(\scr C)$, then $E_\scr{T} A$ is a comodule algebra, functorially in $A$.

Lastly, the following guarantees flatness.
\begin{lemma} \label{lem:flat-ef} Let $F \in \CAlg(\scr C)$ and let $\scr T \subset Pic(\scr C)$ be a Picard groupoid consisting of compact objects. Let $X \in \scr C$ be an $F$-module which is isomorphic as an $F$-module to a small sum of $\scr T$-twists of $E$, i.e. \[ X \simeq \underset{T_{\beta} \in T}{\mathlarger{\bigoplus}} T_{\beta} \otimes F.
\]
Then $X$ is a $\scr T$-flat $F$-module.
\end{lemma}

\begin{proof}
The claim is clear when $X = T \otimes F$ for some $T \in \scr C$. Finite sums of flat $F$-modules are flat. Since $\scr T$ consists of compact objects, the same is true for infinite sums. The result follows.
\end{proof}

%

\subsection{Specialisation to motivic stable homotopy theory}
We apply the discussion of the previous section with $\scr C = h\SH(S)$ and $\scr T = \{\Sigma^{m,n} \1 \mid m, n \in \Z \}$. In this case a $\scr T$-graded ring is essentially the same as an $\epsilon$-commutative ring\NB{what does this really mean?}. We identify $\pi_{**} X = \pi_\scr{T} X$, and so on.

From now on we will assume fixed $E \in \CAlg(h\SH(S))$ such that $E \wedge E$ is a $\scr T$-flat $E$-module.

\subsection{Completed homology coaction} \label{sec:completed-homology-coaction}
Suppose that \[ I \rightarrow \SH(S), i \mapsto X_i \] is a cofiltered diagram.
We put $X = \lim_i X_i$ and $X \widehat\wedge E = \lim_i (X_i \wedge E)$.
We call $\pi_{**}(X \widehat\wedge E)$ the \emph{completed $E$-homology of $X$}.
\begin{construction} \label{construct-completed-h-coaction}
Now assume that
\begin{itemize}
\item $0 = \limone_i \pi_{**}(X_i \wedge E)$.
\end{itemize}
The compatible family of maps \[ X_i \wedge E \to X_i \wedge E \wedge E \] induces \[ X \widehat\wedge E \to X \widehat (E \wedge E). \]
Taking homotopy groups we obtain (using flatness of $E$) \begin{gather*} \pi_{**}(X \widehat\wedge E) \to \pi_{**}(X \widehat\wedge (E \wedge E)) \to \lim_i \pi_{**}(X_i \wedge E \wedge E) \\ \wequi \lim_i \left[ \pi_{**}(X_i \wedge E) \otimes_{E_{**}} E_{**}E \right] =: \pi_{**}(X \widehat\wedge E) \widehat\otimes_{E_{**}} E_{**}E. \end{gather*}

We denote this composite by $\psi^R$.
\end{construction}

\begin{remark}
We did not actually use the assumption that $\limone = 0$ in the construction.
However, the definition of $\pi_{**}(X \widehat\wedge E) \widehat\otimes_{E_{**}} E_{**}E$ seems only reasonable if $\pi_{**}(X \widehat\wedge E) = \lim_i \pi_{**}(X_i \wedge E)$, which is precisely what is ensured by the vanishing.
\end{remark}

\begin{remark} \label{rmk:ops-comodule-compatibility}
By construction, the map $\pi_{**}(X \widehat\wedge E) \to E_{**} X_i$ is a morphism of comodules for each $i \in I$.
\end{remark}

Analogously, we also have the left comodule structure $\psi_L$, related by the equality~\eqref{eq:left-to-right}. Both $\psi_R$ and $\psi_L$ are ring maps, by construction.

\subsection{Cohomology coaction}\label{subsub:cohomology-coaction}
\begin{construction} \label{construct-completed-b}  Suppose that $i \mapsto X_i: I \rightarrow \SH(S)$ is a filtered diagram of dualizable objects\NB{Does this construction depend on the presentation of $X$ as a filtered colimit? One might guess that the category of dualizable objects with a map to $X$ is itself filtered, and so the presentation should be essentially canonical?} and put $X := \colim_I X_i$.
Passing to dual objects we obtain a cofiltered diagram $I^\op \to \SH(S), i \mapsto X_i^\vee$ whose limit we denote by $\widehat{X^\vee}$.
Note that \[ \imap(X, E) \wequi \lim_i (X_i^\vee \wedge E) \wequi \widehat{X^\vee} \widehat\wedge E \] and hence $E^{**}(X)$ is the completed $E$-homology of $\widehat{X^\vee}$.
Assuming a suitable $\limone$ vanishing, Construction \ref{construct-completed-h-coaction} thus supplies us with a coaction \[ \psi^R: E^{**}(X) \to E^{**}(X) \widehat\otimes E_{**}E \] and similarly for $\psi^L$.
\end{construction}

\subsection{Composition algebra}\label{subsub:composition-algebra}
$E^{**}E$ is an associative $E^{**}$-algebra under composition. If \S\ref{app:dualization} applies to $E$ and $E \wedge E$ is degreewise pseudofinite split Tate, so that Lemma \ref{lemm:mot-hom-cohom-dual}(2) applies (e.g. $E = H\Z/2$), then $E_{**}E \wequi (E^{**}E)^\vee$. Moreover, this is an isomorphism of bimodules.\NB{ref?}

\section{Oriented motivic spectra} \label{sec:oriented-spectra}
\subsection{Non-equivariant situation}
We will adopt the definitions and conventions \cite{deglise-orientation} regarding Thom classes and Chern classes which we quickly review. Let $R \in \SH(S)$ be an oriented ring spectrum with an orientation $c \in [\Sigma^{\infty}\P^{\infty}_S, \Sigma^{2,1}R]_{\SH(S)}$. We have the \emph{first Chern class} as in \cite[Definition 2.1.8]{deglise-orientation}
\[
c_1: \Pic(S) \rightarrow R^{2,1}(S).
\] 
Suppose that $\scr E \rightarrow X$ is a vector bundle of rank $d$ and $X$ a smooth $S$-scheme. The projective bundle formula (which we review later in Corollary\ref{prop:cor-global}) gives rise to higher Chern classes $\{ c_i(\scr E) \in E^{2i,i}(X)\}_{0 \leq i \leq d}$; this is a standard maneuver due to Grothendieck and is discussed in \cite[Definition 2.1.16]{deglise-orientation}.

From this we have three more characteristic classes (see \cite[Definition 2.2.2]{deglise-orientation} and \cite[Page 12]{voevodsky2003reduced}):
\begin{definition}\label{def:char-classes} Suppose that $\scr E \rightarrow X$ is a vector bundle of rank $d$ and $X$ a smooth $S$-scheme. We have the following characteristic classes:
\begin{itemize}
\item $th(\scr E)' $ is the \emph{Thom class}, it lives in $R^{2d,d}(\P(\scr E \oplus \scr O))$ and is defined by the formula
\[
th(\scr E)' =  {\SSigma} p^*(c_i(\scr E)) (-c_1(\scr O(-1)))^{d-i}.
\]
\item $th(\scr E)$ is the \emph{refined Thom class}, it lives in $R^{2d, d}(Th_X(\scr E))$ and is characterized as the unique class in $R^{2d,d}(\Th_X(\scr E))$ which maps to the Thom class under pullback along $q: \P(\scr E \oplus \scr O) \rightarrow \Th_X(\scr E)$:
\[
R^{2d,d}(Th_X(\scr E)) \rightarrow R^{2d,d}(\P(\scr E \oplus \scr O)).
\]
\item $e(\scr E)$ is the \emph{Euler class}, it lives in $R^{2d,d}(X)$ and is defined to be the pullback of the refined Thom class along the zero section map $z:X_+ \rightarrow \scr E_+ \rightarrow \Th_X(\scr E))$, i.e., $e(\scr E):=z^*(th(\scr E))$.
\end{itemize}
\end{definition} 

\begin{proposition} \label{prop:pbf} Let $R \in \SH(S)$ be an oriented ring spectrum. Then, for any $n \geq 0$ there is a canonical equivalence
\begin{equation} \label{eq:pbf}
\bigoplus_{0 \leq p \leq n} R \wedge S^{2p,p} \simeq R \wedge \P^n.
\end{equation}
Furthermore, the inclusion $i_n:\P_S^{n} \rightarrow \P_S^{n+1}; [x_0: \cdots, x_n] \mapsto [x_0: \cdots: x_n:0]$ is compatible with the equivalence of~\eqref{eq:pbf} in the sense that the following diagram commutes
\[
\begin{tikzcd}
\bigoplus_{0 \leq p \leq n} R \wedge S^{2p,p} \ar{r} \ar{d} &  R \wedge \P^n \ar{d}{i_n}\\
\bigoplus_{0 \leq p \leq n+1} R \wedge S^{2p,p} \ar{r} &  R \wedge \P^{n+1},
\end{tikzcd}
\]
where the left vertical map is inclusion into the first $n+1$-summands.
\end{proposition}

\begin{proof}
See \cite[\S2.1.12]{deglise-orientation}.
\end{proof}

\begin{corollary} \label{prop:cor-global} Let $R \in \SH(S)$ be an oriented ring spectrum and $X \in \Sm_S$ and $\scr E \rightarrow X$ a vector bundle, then there is a canonical equivalence
\begin{equation} \label{eq:pbf-global}
\bigoplus_{0 \leq p \leq n} R \otimes \Sigma^{\infty}_+X \otimes S^{2p,p} \simeq R \otimes \P(\scr E).
\end{equation}
\end{corollary}

\begin{proof}
We may assume that $X=S$. We define maps \[ \sigma_p: R \otimes \P(\scr E) \xrightarrow{\Delta} R \otimes \P(\scr E)^{\otimes p} \xrightarrow{c_1} R \otimes (S^{2,1})^{\otimes p} \wequi R \otimes S^{2p,p}. \] We claim that $\bigoplus_{0 \le p \le n} \sigma_p$ is an equivalence. Our construction is natural in $S$, so by Zariski excision we may assume that $\scr E$ is trivial. The result now follows from Proposition \ref{prop:pbf}.
\end{proof}

The proof also gives the following formula on the level of $R$-cohomology \cite[Theorem 2.1.13]{deglise-orientation}.

\begin{proposition} \label{prop:prop-pbf-e} Let $R \in \SH(S)$ be an oriented ring spectrum with Chern class $c_1: \Pic(S) \rightarrow  R^{2,1}(S)$. For any $n \geq 0$, let $\scr E \rightarrow S$ be a vector bundle and $p_n: \P(\scr E) \rightarrow S$ be the projection map from the projectivization of $\scr E$. Then the map
\[
\bigoplus_{0 \leq p \leq n} R^{*, *}(S) \rightarrow R^{**}(\P(\scr E)),
\]
\[
(x_0, \cdots, x_n) \mapsto \sum_{p} p_n^*(x_p) c_1(\scr O(-1))^p
\]
is an isomorphism.
\end{proposition}

\begin{proof}
In this form, this is exactly \cite[Theorem 2.1.13]{deglise-orientation} 
\end{proof}

\subsection{Interaction with equivariant spectra}

\begin{definition} \label{def:euler-class} Let $G \rightarrow S$ be a group scheme and $\phi: G \rightarrow GL(V)$ a representation. The \emph{Euler class of $\phi$} is the map in $\SH^{G}(S)$
\[
e(\phi): S^0 \rightarrow \A(\phi)_+ \rightarrow T^{\phi},
\]
where the first map is the equivalence induced by the zero section $z: S \rightarrow \A(\phi)$ and the second map is the canonical projection.
\end{definition}

\begin{lemma} \label{lem:equiv-thom}
Let $E \in \CAlg(h\SH(S))$ be an oriented ring spectrum and $V$ a $G$-equivariant vector bundle on $S$ of rank $n$.
Then there is a
\emph{Thom isomorphism} \[ \imap(\EE G_+, \Sigma^V E^\triv)^G \wequi \Sigma^{2n,n} \imap(\BB G_+, E). \]
The map \[  \imap(\BB G_+, E) \wequi \imap(\EE G_+, E^\triv)^G \to \imap(\EE G_+, \Sigma^V E^\triv)^G \wequi \Sigma^{2n,n} \imap(\BB G_+, E) \] induced by the zero section of $V$ (i.e. the Euler class of $V$) is given by multiplication by $c_n(V)$.
\end{lemma}
\begin{proof}
We have $\imap(\EE G_+, \Sigma^V E^\triv)^G \wequi \imap((\Sigma^{-V} \EE G_+)/G, E)$.
Modelling $\EE G/G = \BB G$ by an ind-smooth $S$-scheme $B_{gm} G = \colim_i B_i$, $V$ corresponds to a vector bundle $V'$ on $B_{gm}G$ and then $(\Sigma^{-V} \EE G_+)/G \wequi \colim_i Th(-V_i')$, where $V_i'$ is the restriction of $V'$ to $B_i$.
The Thom classes for the $V_i$ yield the desired
equivalence.
The relationship between Euler and Chern classes is well-known.
\end{proof}

\begin{proposition} \label{prop:homology-of-dmot} Let $R \in \SH(S)$ be an oriented spectrum. For any $n \in \Z$ there is a
Thom isomorphism
\[
R \wedge D_2^{\mot}(T^{\wedge n}) \simeq \Sigma^{4n,2n}R \wedge BC_{2+}.
\]
\end{proposition}
\begin{proof} If $R$ has a normed orientation, this follows from Proposition~\ref{prop:Dn-thom-iso}. In general we can argue via equivariant motivic homotopy theory.
Denote by $\rho$ the regular representation of $C_2$.
Then $n\rho$ defines a (virtual) $C_2$-vector bundle on $S$.
We have \[ (T^{\wedge n})^{\wedge \underline{2}}  \wequi Th(n\rho) \in \SH^{C_2}(S) \] and hence \[ D_2(T^{\wedge n}) \wequi  (\EE C_2 \wedge Th(n\rho) \in \SH^{C_2}(S))/C_2. \]
As in the proof of Lemma \ref{lem:equiv-thom} we can write this as a filtered colimit of Thom spaces of vector bundles on smooth schemes, the orientation of $R$ supplies (incoherent) Thom isomorphisms, which fit together to yield the desired equivalence.
\end{proof}

\begin{remark}
The equivalences we construct do not seem canonical, but they are: this follows from \cite[Theorem 3.2]{khan2022equivariant} which asserts that certain pro-objects are pro-constant, and in particular all relevant $\limone$ obstructions vanish.
An alternative way to obtain canonical equivalences is to assume that the orientation $\MGL \to R$ upgrades to an $\scr E_1$-map, and this will always be the case in our applications.
\end{remark}

\section{Motivic cohomology} \label{sec:mot-coh}
\label{app:motivic-cohomology}

Throughout, we make a lot of use of Spitzweck's motivic cohomology spectrum \cite{spitzweck2012commutative}. We denote it by $\H\Z$, or possibly $\H\Z_S \in \SH(S)$ to emphasize the base. There are also the finite coefficient variants $\H\Z/n = cofib(\H\Z \xrightarrow{n} \H\Z)$.
To compress notation, sometimes we write $H\Z/0 = H(\Z/0) := H\Z$; note that this is \emph{not} the cofiber of multiplication by $0$ on $H\Z$.

\subsection{Main properties}
\begin{proposition}[General properties of $\H\Z$] \label{prop:HZ-general}
Let $m \ge 0$ and $S$ any scheme.
\begin{enumerate}
\item If $f: S' \to S$ is any morphism of schemes, then there is a canonical equivalence $f^* \H\Z_S/m \wequi \H\Z_{S'}/m$.
\item There is a canonical orientation $\MGL \to \H\Z/m$.
\item We have $\H\Z/m \in \SH(S)^\veff$.
\end{enumerate}
\end{proposition}
\begin{proof}
For (1), see \cite[Theorem 9.19]{spitzweck2012commutative}. For (2), see \cite[Proposition 11.1]{spitzweck2012commutative}. For (3), see \cite[Lemma 13.6]{norms}.
\end{proof}

\begin{proposition}[Properties of $\H\Z$ over Dedekind domains] \label{prop:HZ-dedekind}
Let $m \ge 0$ and suppose that $S$ is essentially smooth over a Dedekind domain.
\begin{enumerate}
\item We have $f_1 \H\Z/m \wequi 0$. In fact $\H\Z \wequi s_0(\1)$.
\item $\H\Z$ represents motivic cohomology in the sense of Levine: we have \[ \imap(\Sigma^\infty X_+, \Sigma^{0,q} \H\Z/m) \wequi \scr M^X(q)/m. \]
\item We have $\ul{\pi}_{p,q} \H\Z/m = 0$ if $q > 0$ or $p < q$.
\item We have $\pi_{p,q} \H\Z/m = 0$ for $q > 0$. We also have $\imap(\Sigma^\infty X_+, \H\Z/m) = \Z_\Zar^X/m$. In particular $\pi_{p,0}(\H\Z/m) = 0$ for $p \ne 0$ and $\pi_{0,0}(\H\Z/m) = \Z^X/m$ (the set of continuous functions from $X$ to $\Z/m$).
\end{enumerate}
\end{proposition}
\begin{proof}
For (2) see \cite[Theorem 7.18]{spitzweck2012commutative}. This implies the first half of (1) by definition of $\scr M^X(q)$ for $q<0$. For the second half of (1), see \cite[Theorem B.4]{norms}. (3) follows from (1) and (2), using \cite[Theorem 3.3]{spitzweck2012commutative} (note that $\ul{\pi}_{p,q} \H\Z/m = \scr H^{-p} \scr M(-q)$). (4) also follows from (2).
\end{proof}

\begin{remark}
If $S$ is essentially smooth over a field, then Proposition \ref{prop:HZ-general}(3) implies Proposition \ref{prop:HZ-dedekind}(3). Over more general bases this argument does not work, because of the failure of stable $\A^1$-connectivity.
\end{remark}

Thanks to the description of motivic cohomology a over Dedekind domain in terms of Levine's higher Chow groups, we have the next lemma which we will be useful. We define $\PTM(S, \Z/m) \subset \Mod_{H\Z/m}(S)$ to be the full subcategory of $H\Z/m$-modules over $S$ spanned by the pure Tate motives, i.e. the smallest subcategory closed under arbitrary sums and containing $\{ \Sigma^{2n,n}H\Z/m \}_{n \in \Z}$.

\begin{lemma} \label{lem:ptm}
Let $D$ be a Dedekind domain with $Pic(D) = 0$.
\begin{enumerate}
\item We have $H^p(\Spec(D), \Z/m(q)) = 0$ for $q \le 1$ and $p > q$, or $q$ arbitrary and $p>q+1$, or $q$ negative.
\item If $D \subset \C$ then the Betti realization functor $\PTM(\Spec(D), R) \rightarrow \Mod_{H\Z/m} = D(\Z/m)$ is faithful. In particular, it is conservative.
\end{enumerate}
\end{lemma}
\begin{proof}
(1) The claim for general $q$ follows from Proposition \ref{prop:HZ-dedekind} and the descent spectra sequence.
For $q=0$ one has $H^p(\Spec(D), \Z/m(0)) = H^p_\Zar(\Spec(D), \Z/m)$ which vanishes for $p>0$ because constant sheaves are flasque.
For $q=1$ we have $H^{p+1}(\Spec(D), \Z/m(1)) = H^p_\Zar(\Spec(D), \Gm \otimes_\Z^{L} \Z/m)$\NB{slightly sloppy notation}, which implies the desired vanishing.

(2)
It follows from (1) that $[\Sigma^{2t,t} H\Z/m, \Sigma^{2n,n}H\Z/m]_{H\Z/m} = \Z/m$ if $t = n$ and $=0$ else. We see by direct calculation that Betti realization is faithful on arbitrary sums of motives of the form $\Sigma^{2t,t}HR$, as was to be shown.
\end{proof}

\begin{proposition}[Higher structures on $\H\Z$] \label{prop:HZ-higher}
Let $m \ge 0$.
\begin{enumerate}
\item For any scheme $S$, there exists a canonical lifting $\H\Z/m \in \CAlg(\SH(S))$.
  If $S$ is smooth ever a Dedekind scheme, there is a lifting $\H\Z/m \in \NAlg(\SH(S))$.
\item The comparison maps of Proposition \ref{prop:HZ-general}(1) are $E_\infty$ (and normed if $f$ is essentially smooth).
\end{enumerate}
\end{proposition}
\begin{proof}
The case where $S$ is a Dedekind scheme follows from \cite[Theorem 13.9]{norms}.
Since commutative monoids are stable under base change, and normed spectra are stable under smooth base change \cite[Proposition 7.6(7)]{norms}, the case of general $S$ follows.
\end{proof}

\begin{remark}
We will show in the next paper in this series that normed spectra are stable under arbitrary base change.
Consequently in Proposition \ref{prop:HZ-higher}, all assumptions on $S$ and $f$ can be dropped.
\end{remark}

\subsection{Cohomology of $B\mu_\ell$}
\begin{lemma} \label{lem:bmu} For any prime $\ell$ and scheme $S$, we have an equivalence in $\SH(S)$
\begin{equation*}
H\Z/\ell \wedge \mathcal{O}_{\mathbb{P}^\infty}(-\ell) \setminus z(\mathbb{P}^\infty) \simeq H\Z/\ell \wedge \mathbb{P}^{\infty} \oplus \Sigma^{1,1} H\Z/\ell \wedge \mathbb{P}^{\infty}.
\end{equation*}
\end{lemma}
\begin{proof}
We may assume that $S=\Spec(\Z)$. For each $n  \geq 0$, the purity isomorphism furnishes a cofiber sequence in $\Spc(\Z)$
\[
\mathcal{O}_{\mathbb{P}^n}(-\ell) \setminus z(\mathbb{P}^n) \rightarrow \mathcal{O}_{\mathbb{P}^n}(-\ell) \rightarrow \Th\mathcal{O}_{\mathbb{P}^n}(-\ell).
\]
These cofiber sequences are compatible for different choices of $n$ (with the canonical inclusion $\P^n \hookrightarrow \P^{n+k}$) and in the colimit we obtain a cofiber sequence
\[ \mathcal{O}_{\mathbb{P}^\infty}(-\ell) \setminus z(\mathbb{P}^\infty) \rightarrow \mathcal{O}_{\mathbb{P}^\infty}(-\ell) \wequi \P^\infty \rightarrow \Th(\mathcal{O}_{\mathbb{P}^\infty}(-\ell)). \]
Since the motivic spectrum $H\Z/\ell$ is oriented, the cofiber sequence above induces a cofiber sequence in $\SH(S)$:
\[ H\Z/\ell\wedge (\mathcal{O}_{\mathbb{P}^\infty}(-\ell) \setminus z(\mathbb{P}^\infty))_+ \rightarrow H\Z/\ell\wedge \P^\infty \rightarrow \Sigma^{2,1} H\Z/\ell \wedge \P^\infty. \]
We claim that the right-most map is zero. By Proposition~\ref{prop:pbf}, the map in question is between pure Tate motives, i.e. sums of terms of the form $\Sigma^{2n,n} H\Z/\ell$. Thus the claim follows from Lemma~\ref{lem:ptm} and the corresponding classical fact.
It follows that we have a non-canonical equivalence as stated.
We obtain a canonical basis by appeal to \cite[Theorem 10.16]{spitzweck2012commutative}.
\end{proof}

\begin{proposition} \label{prop:bmu}We have a canonical equivalence
\[
H\Z/\ell \wedge (B_{\et}\mu_{\ell})_+  \simeq  H\Z/\ell \wedge \mathbb{P}^{\infty} \oplus \Sigma^{1,1} H\Z/\ell \wedge \mathbb{P}^{\infty} \in \SH(S).
\] 
\end{proposition}

\begin{proof}
Since $B\mu_\ell \wequi \mathcal{O}_{\mathbb{P}^n}(-\ell) \setminus 0$\NB{ref?}, the claim follows from Lemma~\ref{lem:bmu}.

\end{proof}

\subsection{The motivic Steenrod algebra and its dual}\label{app:steenrod}

\begin{theorem}[The motivic dual Steenrod algebra] \label{thm:HZ-dual-steenrod}
Let $\ell$ be a prime invertible on the scheme $S$.
\begin{enumerate}
\item There exist canonical elements $\tau_i \in \pi_{2\ell^i - 1, \ell^i - 1}(\H\Z/\ell \wedge \H\Z/\ell), i = 0,1,2 \dots$ and $\xi_i \in \pi_{2(\ell^i - 1), \ell^i - 1}(\H\Z/\ell \wedge \H\Z/\ell), i = 1, 2, \dots$\sjeremiah{$i=0$ needed? (for $\xi_0$?)}\tombubble{$\xi_0=1$?} such that $\H\Z/\ell_{**} \H\Z/\ell$ is the free $\H\Z/\ell_{**}$-module on generators of the form $\prod'_{i \ge 0} \tau_i^{\epsilon_i} \times \prod'_{i > 0} \xi_i^{r_i}$. Here $\prod'$ denotes a product where all but finitely many exponents are 0, $\epsilon_i \in \{0,1\}$ and $r_i \ge 0$. Moreover $\H\Z/\ell \wedge \H\Z/\ell$ splits into a sum of shifts and twists of $\H\Z/\ell$, corresponding to these generators, i.e., we have an equivalence of the form
\[
\underset{B}{\mathlarger{\bigoplus}} \Sigma^{p(I), q(I)}H\Z/\ell \stackrel{\simeq}{\rightarrow} H\Z/\ell \wedge H\Z/\ell,
\]
where the indexing set $B$ are generators of $\H\Z/\ell_{**} \H\Z/\ell.$ The elements $\tau_i, \xi_i$ are stable under base change.
\end{enumerate}
From now on we put $\ell=2$.
\begin{enumerate}
\setcounter{enumi}{1}
\item We have $\tau_i^2 = \tau \xi_{i+1} + \rho \tau_{i+1} + \rho \tau_0 \xi_{i+1}$.
\item The comultiplication $\psi: H\Z/2_{**}H\Z/2 \to H\Z/2_{**}H\Z/2 \otimes_{H\Z/2_{**}} H\Z/2_{**}H\Z/2$ satisfies
\begin{align*}
\psi(\rho) &= \rho \\
\psi(\tau) &= \tau \\
\psi(\tau_r) &= \tau_r \otimes 1 + 1 \otimes \tau_r + \sum_{i=0}^{r-1} \xi_{r-i}^{2^i} \otimes \tau_i \\
\psi(\xi_r) &= \xi_r \otimes 1 + 1 \otimes \xi_r + \sum_{i=0}^{r-1} \xi_{r-i}^{2^i} \otimes \xi_i. \\
\end{align*}
\item The right unit $H\Z/2_{**} \to H\Z/2_{**}H\Z/2$ sends $\rho$ to $\rho$ and $\tau$ to $\tau + \rho \tau_0$.
\item The antipode $\chi: H\Z/2_{**}H\Z/2 \to H\Z/2_{**}H\Z/2$ satisfies
\begin{align*}
\chi(\xi_r) &= \xi_r + \stackrel{r-1}{\underset{i=1}{\SSigma}}\xi^{\ell^i}_{r-i}\chi(\xi_i) \\
\chi(\tau_r) &= \tau_r + \stackrel{r-1}{\underset{i=0}{\SSigma}}\xi^{\ell^i}_{r-i}\chi(\tau_i).
\end{align*}
\end{enumerate}
\end{theorem}
\begin{remark}
This theorem implies that the Hopf algebroid $H\Z/2_{**}H\Z/2$ can be described exactly as in \cite[Theorem 5.6]{hoyois2013motivic}.
\end{remark}
\begin{proof}[Beginning of proof of Theorem \ref{thm:HZ-dual-steenrod}.]
Since $H\Z/\ell$ is stable under base change, and all statements of the theorem also are, we see that we may immediately reduce to the case $S = Spec(\Z[1/\ell])$. Moreover part (1) is \cite[Theorem 11.24]{spitzweck2012commutative} and part (5) follows via the identity \cite[Definition A1.1.1]{ravenel1986complex} \[ m \circ (\id \otimes \chi ) \circ \psi = \eta_L \circ \epsilon \] from (3).
\end{proof}

In order to finish the proof, we need some preparations. Denote by $\scr B_S$ the set of \emph{bistable cohomology operations}. In other words $\scr B_S^{**}$ is the bigraded set with $\scr B_S^{p,q}$ the set of natural transformations $\H^{**}(\bullet, \Z/\ell) \Rightarrow \H^{*+p,*+q}(\bullet, \Z/\ell)$ of functors from $\Spc(S)_\pt$ to bigraded abelian groups, which are \emph{bistable} (i.e. the diagram analogous to \eqref{eq:cohomology-operations-bistable} commutes). Then $\scr B_S^{**}$ is an associative algebra over $\H\Z/\ell^{**}$ (with the operation given by composition). Note that there is an evident algebra homomorphism $\alpha: \H\Z/\ell^{**} \H\Z/\ell \to \scr B$. Now let $\ell=2$. The construction of the cohomological power operations in Definition \ref{def:cohomological-power-ops} together with their stability and naturality, as established in Proposition \ref{prop:cohomological-power-ops-properties}(1,3), furnishes us with elements $\Sq^i \in \scr B_S$.

For a sequence $I = (i_1, i_2, \dots, i_n)$, we denote by $\Sq^I = \Sq^{i_1} \circ \dots \circ \Sq^{I_n}$ the composite operation. Recall that a sequence is called \emph{admissible} if $i_j \ge 2i_{j+1} > 0$ for all $j$. The empty sequence $\emptyset$ is considered admissible, and $\Sq^\emptyset := \id$.

\begin{theorem}[The motivic Steenrod algebra] \label{thm:HZ-steenrod-dedekind}
Let $\ell$ be a prime invertible on the scheme $S$.
\begin{enumerate}
\item The morphism $\alpha: \H\Z/\ell^{**} \H\Z/\ell \to \scr B$ is surjective.
\end{enumerate}
From now on let $\ell=2$, and $S$ essentially smooth over a Dedekind domain. Then moreover we have the following.
\begin{enumerate}
\setcounter{enumi}{1}
\item $\alpha$ is an isomorphism.
\item $\scr B^{**}$ is free (as a left $\H\Z/2^{**}$-module) on the admissible monomials $\Sq^I$.
\item $\alpha$ is natural in the following sense: if $f: S' \to S$ is a morphism of schemes essentially smooth over Dedekind domains, then $f^* \Sq^I = \Sq^I$, where we view $\Sq^I \in \H\Z/2^{**}\H\Z/2$ (over $S$ and $S'$, respectively).
\item The elements $\tau_0 \in \H\Z/2_{**}\H\Z/2$ and $\Sq^1 \in \H\Z/2^{**}\H\Z/2$ are dual.
\item $\Sq^1 \tau = \rho$, $\Sq^i \tau = 0$ for $i>1$ and $\Sq^i \rho = 0$ for $i>0$.
\end{enumerate}
\end{theorem}

\begin{proof}[End of proof of Theorem \ref{thm:HZ-dual-steenrod}, assuming Theorem \ref{thm:HZ-steenrod-dedekind}.]
Recall that we have reduced to $S = Spec(\Z[1/2])$, which is a Dedekind domain, so Theorem \ref{thm:HZ-steenrod-dedekind} applies.

(2) Consider the right cohomology coaction map \[ \psi^R: H^{**}(\R\P^\infty, \Z/2) \to H^{**}(\R\P^\infty, \Z/2) \otimes_{H^{**}} H\Z/2_{-*-*}H\Z/2. \] See \S\ref{subsub:cohomology-coaction} for a reminder of its construction, and the fact that this is a ring map. We have $H^{**}(\R\P^\infty, \Z/2) \wequi H^{**}[u, v]/u^2 + \tau v + \rho u$ by Lemma \ref{lemm:coh-of-BSigma2}. The coaction sends \begin{gather*} v \mapsto v \otimes 1 + \SSigma_{i \geq 0} v^{2^i} \otimes \xi_i \\ u \mapsto u \otimes 1 + \SSigma_{i \geq 0} v^{2^i} \otimes \tau_i, \end{gather*} by Lemma \ref{lem:psir-easy}. Computing $\psi^R(u^2) = \psi^R(u)^2 = \psi^R(\tau v + \rho u)$ and comparing the coefficient of $v^{2^{i+1}}$ gives the desired result. What is important to observe here is that $\psi^R$ is a morphism of \emph{left} $\H\Z/2^{**}$-algebras, and so for example \[ \psi^R(\tau v) = \tau \psi^R(v) = \tau \left[ v \otimes 1 + \SSigma_{i \geq 0} v^{2^i} \otimes \xi_i \right]. \] Here on the right hand side, $\tau(v \otimes 1) = (\tau v) \otimes 1$, and so on. However we want to isolate the powers of $v$, which we do by using $\tau(m \otimes n) = (\tau m) \otimes n = m \otimes (\chi(\tau) n)$. \todo{This is very subtle. Can we find a good reference?} This is how the term $\chi(\tau) = \tau + \rho \tau_0$ appears in the formula for $\tau_i^2$.

(3) Voevodsky's proof \cite[\S12]{voevodsky2003reduced} essentially works as before. We have $\psi(\xi_r) = \sum_i a_i \otimes b_i$ if and only if, for all $\alpha, \beta \in H\Z/2^{**}H\Z/2$ we have \[ \lra{\xi_r, \alpha\beta} = \sum_i \lra{a_i, \alpha \lra{b_i, \beta}}. \] It suffices to check this on a $H\Z/2_{**}$-basis of $H\Z/2^{**}H\Z/2$. Taking the dual basis of the monomial basis, we can ensure that $\lra{\xi_i, \beta} \in \Z/2 \subset H\Z/2_{**}$, which is in particular central in $H\Z/2^{**}H\Z/2$. It is thus enough to show that \[ \lra{\xi_k, \alpha \beta} = \sum_i \lra{\xi_{k-i}^{2^i}, \alpha} \lra{\xi_i, \beta} \] for $\alpha, \beta \in H\Z/2^{**}H\Z/2$ such that $\lra{\xi_i, \beta} \in \Z/2$. To begin with, note that \[ \alpha(v^{2^i}) = \sum_j \lra{\xi_j^{2^i}, \alpha} v^{2^{i+j}}. \] Indeed this claim is equivalent to \[ \psi(v^{2^i}) = 1 \otimes v^{2^i} + \SSigma_{j \geq 0} \xi_j^{2^i} \otimes v^{2^{i+j}}, \] which holds for $i=0$ by definition and follows for $i>0$ since $\psi$ is a ring map and we are working in characteristic $2$. In particular we have \[ (\alpha\beta)(v) = \sum_i \lra{\xi_i, \alpha \beta} v^{2^i} \] and similarly \[(\alpha\beta)(v) = \alpha(\beta(v)) = \alpha \left( \sum_i \lra{\xi_i, \beta} v^{2^i} \right) = \sum_{i,j} \lra{\xi_j^{2^i}, \alpha} \lra{\xi_i, \beta} v^{2^{i+j}}, \] using the assumption that $\lra{\xi_i, \beta} \in \Z/2$. Comparing the coefficients of $v^{2^k}$ gives the desired result. The claim about $\tau_i$ is handled similarly. Since $\psi$ is a morphism of $H\Z/2_{**}$-algebras\NB{really?}, the claims about $\rho$ and $\tau$ are clear.

(4) By construction, $\eta_R(x) = \sum_I a_I S_I$, where $S^I$ is dual to $\Sq^I$ in the Serre-Cartan basis and $\Sq^I(x) = a_I$. This determines the value of $\eta_R$ on $\tau$ and $\rho$ from the action of the $\Sq^i$ on them. In other words the result follows from Theorem \ref{thm:HZ-steenrod-dedekind}(5).
\end{proof}

\begin{proof}[Proof of Theorem \ref{thm:HZ-steenrod-dedekind}.]
(1) Present $\H\Z/\ell$ as the $\Omega_T$-spectrum $(K_0, K_1, K_2, \dots)$. By standard arguments (e.g. \cite[Proposition 2.7]{voevodsky2003reduced}), we have $\scr B^{**} \wequi \lim_n \H^{*+2n,*+n}(K_n, \Z/\ell)$. There is the Milnor exact sequence \[ 0 \to \limone_n \H^{*+2n-1,*+n}(K_n, \Z/\ell) \to \H\Z/\ell^{**} \H\Z/\ell \to \lim_n \H^{*+2n,*+n}(K_n, \Z/\ell) \to 0; \] consequently $\alpha$ is surjective.

From now on let $\ell=2$.

(2, 3) Let $\scr D^{**}$ denote the free (bigraded) $\H\Z/2^{**}$-module on the admissible monomials in Steenrod operations. Choosing a lift of each such operation along $\alpha$, we may form $\beta: \scr D^{**} \to \H\Z/2^{**} \H\Z/2$. We shall show first that $\beta$ is an isomorphism. We have $\H\Z/2^{**}\H\Z/2 = (\H\Z/2_{**}\H\Z/2)^\vee$, which is itself a free $\H\Z/2^{**}$-module by Theorem \ref{thm:HZ-dual-steenrod} (1) and Lemma \ref{lemm:mot-hom-cohom-dual}(2)\NB{This is where we use the assumption on $S$. In general we do not know that $\H\Z/2^{**} \H\Z/2$ is free, even though $\H\Z/2_{**} \H\Z/2$ is.}. It follows from \cite[Corollary 14]{bachmann-real-etale} that a morphism of free $\H\Z/2^{**}$-modules is an isomorphism if and only if it is so after ``pulling back to the residue fields'' (i.e. after tensoring with $\H\Z/2^{**}(S) \to \H\Z/2^{**}(k)$ for the various residue fields $k$ of $S$). Since $\H\Z/2^{**}\H\Z/2$ is free on a canonical basis, we have $\H\Z/2^{**}\H\Z/2 \otimes_{\H\Z/2^{**}} \H\Z/2^{**}(k) \wequi \H\Z/2^{**}\H\Z/2_k$ and the map $[\H\Z/2, \H\Z/2]^{**}_S \wequi \H\Z/2^{**}\H\Z/2 \to \H\Z/2^{**}\H\Z/2_k \wequi [\H\Z/2, \H\Z/2]^{**}_k$ is the usual pullback. Let $i: Spec(k) \to S$ denote a residue field inclusion. We claim that $i^*(\beta(\Sq^I)) = \Sq^I$. Compatibility of our power operations with base change (Proposition \ref{prop:cohomological-power-ops-properties}) together with Example \ref{ex:voevodsky-ops} implies that if $\scr X \in \Spc(S)_*$ and $a \in \H^{**}(\scr X, \Z/2)$ then $i^*(\beta(\Sq^I))(i^*(a)) = i^*(\Sq^I(a)) = \Sq^I(i^*(a))$. Consider $\scr X = \bigvee_{n \ge 0} (\R\P^\infty)^n_+$ and $a = ((uv)^{\otimes n})_n$\NB{explain better}. Then the map $\H\Z/2^{**}\H\Z/2_k \to \H\Z/2^{**} i^*(\scr X), P \mapsto P(i^*(a))$ is injective \cite[Lemme 5.1.5]{riou2012ops}. Since $i^*(\beta(\Sq^I))(i^*(a)) = \Sq^I(i^*(a))$ we conclude that $i^*(\beta(\Sq^I)) = \Sq^I$, as claimed. It follows that the map $i^*(\beta)$ is an isomorphism \cite[Theorem 1.1]{hoyois2013motivic}, and hence $\beta$ is an isomorphism.

It remains to show that $\beta \alpha$ is an isomorphism, or in other words that the operations $\{\Sq^I\}_I \in \scr B^{**}$ (indexed on admissible monomials) are linearly independent over $\H\Z/2^{**}$. Let $I = \H\Z/2^{**>0}$. This is an ideal and $\H\Z/2^{**}/I = \Z/2$ by Proposition \ref{prop:HZ-dedekind}(4). Let $\scr D^{**\le w}$ denote the free submodule of $\scr D^{**}$ with basis the monomials of weight $\le w$. Let $\scr X_n = (\A^n \setminus 0/\Sigma_2)^n_+$. We have as before the class $a \in \H\Z/2^{**} \scr X_n$. We claim that for $n$ sufficiently large, the canonical map $\gamma: \scr D^{**\le w} \to \H\Z/2^{**} \scr X_n, P \mapsto P(a)$ is a split injection. This will conclude the proof. To see that $\gamma$ is a split injection, we note that it is a map between free finite rank modules over $\H\Z/2^{**}$. It follow from \cite[Lemme 5.1.6]{riou2012ops} that $\gamma$ is a split injection if and only if $\gamma/I$ is injective. But note that $\gamma/I$ is independent of $S$ (since $\H\Z/2^{**}/I = \Z/2$ is, and our Steenrod operations are natural). Hence in order to show hat $\gamma/I$ is injective it suffices to show that $\gamma$ is split injective if $S$ is the spectrum of a field, which is proved in \cite[Lemme 5.1.5]{riou2012ops}.

(4) We know that $f^* \Sq^I$ can be written as a linear combination of finitely many $\Sq^J$, and that these operations can be distinguished by their effect on $a \in \H\Z/2^{**} \scr X_n$ (for $n$ sufficiently large). By naturality of the power operations, we compute that \[ (f^* \Sq^I)(f^*a) = f^*(\Sq^I a) = \Sq^I(f^*a), \] and the result follows.

(5) For sequences $E = (\epsilon_0, \epsilon_1, \dots), R = (r_1, r_2, \dots)$ of integers, almost all of which are zero, $\epsilon_i \in \{0,1\}$, $r_i \ge 0$, we put \[ \tau^E \xi^r = \prod_i \tau_i^{\epsilon_i} \prod_i \xi_i^{r_i}; \] this is a basis of $\H\Z/2_{**} \H\Z/2$. Write $\rho(E,R)$ for the corresponding dual basis elements. We need to check that $\lra{\Sq^1, \tau^E \xi^R} = 0$ for $\tau^E \xi^R \ne \tau_0$ and $=1$ else. By (4), this property is stable under base change, and so we may assume that $S = Spec(\Z[1/2])$. Since the $\rho(E,R)$ form a basis, we may uniquely write $\Sq^1 = \sum_{E,R} a_{E,R} \rho(E,R)$. Since $\Sq^1$ has weight zero and all non-zero $a_{E,R}, \rho(E,R)$ have weight $\ge 0$, we find that only the terms where both $a_{E,R}$ and $\rho(E,R)$ have weight zero can contribute, i.e. $\Sq^1 = a_0 \id + a_1 \tau_0^\vee$. But we must have $a_0 \in \H\Z/2^{1,0}(\Z[1/2]) = 0$ and $a_1 \in \H\Z/2^{0,0}(\Z[1/2]) = \Z/2$, whence either $\Sq^1 = \tau_0^\vee$ or $\Sq^1 = 0$. Since the latter is clearly not the case, we are done.

(6) $\Sq^1 \tau = \rho$ essentially by construction (and since $\Sq^1$ is the Bockstein). The other vanishing results follow from Proposition \ref{prop:cohomological-power-ops-properties}(6,7).
\end{proof}

\begin{example} \label{ex:composition-cartan}
Let $i \ge 0$ and $x \in H\Z/2^{**}$. It follows from Theorem \ref{thm:HZ-dual-steenrod} that $\Sq^i x$ (the endomorphism of $H\Z/2$ obtained as $\Sq^i$ following multiplication by $x$) can be written as a linear combination $\Sigma_j a_j \Sq^j$ with some coefficients $a_j \in H\Z/2^{**}$. To do so, for simplicity let $i = 2n$ be even. We note that for any $t \in H\Z/2^{**} \scr X$ for any $\scr X \in \Spc(S)_*$ we have (using the Cartan formula from Proposition \ref{prop:cohomological-power-ops-properties}) \[ (\Sq^{2n} x)(t) = \Sq^{2n}(x t) = \Sigma_{j+k=n} \Sq^{2j}(x) \Sq^{2k}(t) + \tau \Sigma_{j+k=n-1} \Sq^{2j+1}(x) \Sq^{2k+1}(t) \] and hence \begin{equation}\label{eq:Sq-x} \Sq^{2n} x = \Sigma_{j+k=n} \Sq^{2j}(x) \Sq^{2k} + \Sigma_{j+k=n-1} \tau \Sq^{2j+1}(x)\Sq^{2k+1} \in H\Z/2^{**}H\Z/2. \end{equation}

In a similar fashion, since $\Sq^1 = \beta$ is a derivation, one finds that \begin{equation}\label{eq:beta-x} \Sq^1 x = x\Sq^1 + \Sq^1(x). \end{equation}
\end{example}

\section{2-excisive functors and suspension} \label{app:excisive-functors}

Recall that a functor $F: \scr C \to \scr D$ is called reduced if $F(*) \wequi *$. If $X \in \scr C$ then we denote by \[ C_n(X): P(\{1, 2, \dots, n\}) \to \scr C \] the strongly coCartesian cube generated by $X$. Here as usual $P(A)$ denotes the power set of $A$. In other words $C_n(X)(\emptyset) = X, C_n(X)(\{i\}) = *$ for all $i$, and the rest is determined by the strongly coCartesian condition.  Similarly we denote by \[ D_n(X_1, \dots, X_n): P(\{1, 2, \dots, n\}) \to \scr C \] the strongly cocartesian cube with $D_n(X_1, \dots, X_n)(\emptyset) = *, D_n(X_1, \dots, X_n)(\{i\}) = X_i$. We put \[ cr_n F: \scr C^n \to \scr D, (X_1, \dots, X_n) \mapsto \tfib(F \circ D_n). \] Here $\tfib$ denotes the total homotopy fiber of a cube.  

The main result of this appendix is the following.
\begin{proposition} \label{prop:fibn-sequence}
Let $\scr C$ be a pointed $\infty$-category with finite coproducts,
and $\scr D$ a pointed $\infty$-category with finite limits.
Let $F: \scr C \to \scr D$ be 2-excisive and reduced. Then there is a canonical
fibration sequence
\[ cr_2 F(\Sigma X, \Sigma X) \to F(X) \to \Omega F(\Sigma X). \]
\end{proposition}

\begin{example} \label{ex:2-exc-cr}
Suppose that $\scr C$ is additive and has finite colimits, $\scr D$ has colimits and is stable, and that $F$ preserves geometric realizations. We have for $X, Y \in \scr C$ canonical maps $X \to X \oplus Y \to X$ inducing $F(X \oplus Y) \wequi F(X) \oplus F(Y) \oplus c(X, Y)$.  One checks that $cr_2 F(X, Y) = \Omega^2 c(X, Y)$. Suppose that $c$ is exact in each variable. Then $cr_2 F(\Sigma X, \Sigma Y) = c(X, Y)$. Moreover, $F$ is weakly polynomial of degree $\le 2$ in the sense of \cite[Definition 5.21]{norms} and hence $2$-excisive \cite[Remark 5.22]{norms}. We deduce that there is an exact triangle \[ \Sigma c(X, X) \to \Sigma F X \to F \Sigma X. \] This applies in particular if $\scr C = Spc^{tr}(k)$, $\scr D = \DM^\eff(k, \Z/2)$ and $F = S^2_{tr}$ in the sense of \cite{voevodsky2010motivic}. Then $c(X, Y) = X \otimes Y$ \cite[Proposition 2.37]{voevodsky2010motivic} and hence the above cofiber sequence recovers parts of \cite[Corollary 2.46]{voevodsky2010motivic}.
\end{example}

In the proof of Proposition \ref{prop:fibn-sequence}, we make use of the following standard properties of total fibers, for which we unfortunately could not find a reference.\NB{seriously?}

\begin{lemma} \label{lemm:tfib-technical}
Let $S$ be a finite set, $\scr C$ a pointed $\infty$-category with finite limits and $\phi_S: \scr C \to \Fun(P(S), \scr C)$ the functor sending $X$ to the diagram $D$ with $D(\emptyset) = X$ and $D(U) = 0$ for $U \ne \emptyset$.

Then $\phi_S$ has a right adjoint given by the total fiber.
\end{lemma}
\begin{proof}
Denote by $j: P(S) \setminus \emptyset \to P(S)$ the canonical map, and by $i: * \to P(S)$ the inclusion at the empty set.
We note that $\phi_S$ is right Kan extension along $i$.
For $X \in \scr C$ we have $j^*i_* X \wequi 0$, and hence $\Map(i_* X, j_* j^* D) = 0$ for all $D \in \Fun(P(S), \scr C)$.
Let $t: \Fun(P(S), \scr C) \to \Fun(P(S), \scr C)$ be the functor $D \mapsto fib(D \to j_*j^* D)$.
We deduce that $\Map(i_* X, tD) \wequi \Map(i_* X, D)$.
But $tD \wequi i_* \tfib(D)$ and $i_*$ is fully faithful, so $\Map(i_*X, D) \wequi \Map(X, \tfib(D))$.
The result follows.\NB{This is slightly sloppy.}
\end{proof}

\begin{corollary} \label{corr:tfib-top-bottom}
Let $S$, $\scr C$ as above. Pick $s_0 \in S$ and write $S' = S \setminus \{s_0\}$. We have a map $i_1: P(S') \to P(S), U \mapsto U \coprod \{s_0\}$, and also the canonical inclusion $i_0: P(S') \to P(S)$. There is a (unique) natural transformation $\epsilon: i_0 \Rightarrow i_1$.

For $D: P(S) \to \scr C$, there is a canonical equivalence \[ \tfib(D) \wequi \fib(\epsilon: \tfib(D \circ i_1) \to \tfib(D \circ i_0)). \]
\end{corollary}
In other words, the total fiber of a cube can be computed as fiber of the total fibers of two opposite faces.
\begin{proof}
The existence and uniqueness of $\epsilon$ is clear, since the category is a poset.

We have an isomorphism $P(S) \wequi P(S') \times \{\emptyset, \{s_0\}\}, (U, X) \mapsto U \cup X.$ The functor \[ \phi_S: \scr C \to \Fun(P(S), \scr C) \wequi \Fun(\Delta^1, \Fun(P(S'), \scr C)) \] factors as \[ \scr C \xrightarrow{\phi'} \Fun(\Delta^1, \scr C) \xrightarrow{\phi''} \Fun(\Delta^1, \Fun(P(S'), \scr C)), \] where $\phi' = \phi_{\{\emptyset, \{s_0\}\}}$ and $\phi'' = \Fun(\Delta^1, \phi_{S'})$. It follows from Lemma \ref{lemm:tfib-technical} that the right adjoint factors as\NB{We are using something slightly non-trivial here about adjoints and functor categories.} $\tfib \circ \Fun(\Delta^1, \tfib)$, whence the result.
\end{proof}

\begin{proof}[Proof of Proposition \ref{prop:fibn-sequence}.]
Since $F$ is 2-excisive, $F \circ C_3(X)$ is cartesian, whence $\tfib(F \circ C_3(X)) \wequi 0$. Applying Corollary \ref{corr:tfib-top-bottom}, we deduce that the total fibers of the following two diagrams (the top and bottom faces of $F \circ C_3(X)$) are equivalent:
\begin{equation*}
\begin{CD}
* @>>> F \Sigma X \quad@.\quad  F \Sigma X @>>> F(\Sigma X \coprod \Sigma X) \\
@AAA      @AAA          @AAA                 @AAA                  \\
FX @>>> *         @.   *         @>>>        F \Sigma X
\end{CD}
\end{equation*}
By definition, the total fiber on the right is $cr_2 F(\Sigma X, \Sigma X)$ and the total fiber on the left is $\fib(FX \to \Omega F \Sigma X)$. This concludes the proof.
\end{proof}

\begin{remark}
The proof shows that more generally for any $n$-excisive functor $F$, we have \[ cr_n(\Sigma X, \dots, \Sigma X) \wequi \tfib(F \circ C_n(X)). \]
\end{remark}

\end{subappendices}

%% file: operations-sa.bbl
\providecommand{\bysame}{\leavevmode\hbox to3em{\hrulefill}\thinspace}
\providecommand{\MR}{\relax\ifhmode\unskip\space\fi MR }
\providecommand{\MRhref}[2]{%
  \href{http://www.ams.org/mathscinet-getitem?mr=#1}{#2}
}
\providecommand{\href}[2]{#2}
\begin{thebibliography}{BMMS86}

\bibitem[ACB19]{AB:Thom}
Omar Antol\'{\i}n-Camarena and Tobias Barthel, \emph{A simple universal
  property of {T}hom ring spectra}, J. Topol. \textbf{12} (2019), no.~1,
  56--78. \MR{3875978}

\bibitem[Bac18]{bachmann-real-etale}
Tom Bachmann, \emph{Motivic and real \'{e}tale stable homotopy theory}, Compos.
  Math. \textbf{154} (2018), no.~5, 883--917. \MR{3781990}

\bibitem[Bak13]{baker2013power}
Andrew Baker, \emph{Power operations and coactions in highly commutative
  homology theories}, arXiv preprint arXiv:1309.2323 (2013).

\bibitem[BEH21]{colimits}
Tom Bachmann, Elden Elmanto, and Jeremiah Heller, \emph{Motivic colimits and
  extended powers}, 2021.

\bibitem[BH20]{norms}
Tom Bachmann and Marc Hoyois, \emph{Norms in motivic homotopy theory},
  \href{https://arxiv.org/abs/1711.03061}{arXiv:1711.03061}, to appear in
  Ast\'erisque.

\bibitem[BMMS86]{hinfty}
R.~R. Bruner, J.~P. May, J.~E. McClure, and M.~Steinberger, \emph{{$H\sb \infty
  $} ring spectra and their applications}, viii+388. \MR{836132}

\bibitem[BMS19]{bhatt2019topological}
Bhargav Bhatt, Matthew Morrow, and Peter Scholze, \emph{Topological hochschild
  homology and integral p $ p $-adic hodge theory}, Publications
  math{\'e}matiques de l'IH{\'E}S \textbf{129} (2019), no.~1, 199--310.

\bibitem[CD19]{cisinski-deglise}
Denis-Charles Cisinski and Fr\'{e}d\'{e}ric D\'{e}glise, \emph{Triangulated
  categories of mixed motives}, xlii+406. \MR{3971240}

\bibitem[D{\'e}g11]{deglise-orientation}
Fr{\'e}d{\'e}ric D{\'e}glise, \emph{Orientation theory in arithmetic geometry}.

\bibitem[GH]{gepner-heller}
David Gepner and Jeremiah Heller, \emph{The tom dieck splitting theorem in
  equivariant motivic homotopy theory},
  \href{https://arxiv.org/abs/1910.11485}{arXiv:1910.11485}.

\bibitem[GJ09]{goerss2009simplicial}
Paul~G Goerss and John~F Jardine, \emph{Simplicial homotopy theory}, Springer
  Science \& Business Media, 2009.

\bibitem[GL16]{GepnerLawson}
David Gepner and Tyler Lawson, \emph{Brauer groups and galois cohomology of
  commutative ring spectra}, 2016.

\bibitem[HK{\O}]{hoyois2013motivic}
Marc Hoyois, Shane Kelly, and Paul~Arne {\O}stv{\ae}r, \emph{The motivic
  steenrod algebra in positive characteristic}, to appear in J.Eur.Math.Soc.

\bibitem[HVO16]{Heller:2016aa}
Jeremiah Heller, Mircea Voineagu, and Paul~Arne Ostvaer, \emph{Topological
  comparison theorems for bredon motivic cohomology}.

\bibitem[HW19]{bk-book}
Christian Haesemeyer and Charles~A. Weibel, \emph{The norm residue theorem in
  motivic cohomology}, Annals of Mathematics Studies, vol. 200, Princeton
  University Press, Princeton, NJ, 2019. \MR{3931681}

\bibitem[KN]{kn-lectures}
Achim Krause and Thomas Nikolaus, \emph{Lectures on topological hochschild
  homology and cyclotomic spectra}.

\bibitem[KN19]{kn}
\bysame, \emph{B\"{o}kstedt periodicity and quotients of dvrs}, 2019.

\bibitem[KR22]{khan2022equivariant}
Adeel~A Khan and Charanya Ravi, \emph{Equivariant generalized cohomology via
  stacks}, arXiv preprint arXiv:2209.07801 (2022).

\bibitem[Lur18]{kerodon}
Jacob Lurie, \emph{Kerodon}, \url{https://kerodon.net}, 2018.

\bibitem[Mah79]{mahowald}
Mark Mahowald, \emph{Ring spectra which are {T}hom complexes}, Duke Math. J.
  \textbf{46} (1979), no.~3, 549--559. \MR{544245}

\bibitem[Mil58]{milnor-dual}
John Milnor, \emph{The {S}teenrod algebra and its dual}, Ann. of Math. (2)
  \textbf{67} (1958), 150--171. \MR{99653}

\bibitem[MV99]{A1-homotopy-theory}
Fabien Morel and Vladimir Voevodsky, \emph{$\mathbb{A}^1$-homotopy theory of
  schemes}, Publications Math{\'e}matiques de l'Institut des Hautes {\'E}tudes
  Scientifiques \textbf{90} (1999), no.~1, 45--143 (English).

\bibitem[NS18]{nikolaus-scholze}
Thomas Nikolaus and Peter Scholze, \emph{On topological cyclic homology}, Acta
  Math. \textbf{221} (2018), no.~2, 203--409. \MR{3904731}

\bibitem[Rav86]{ravenel1986complex}
Douglas~C Ravenel, \emph{Complex cobordism and stable homotopy groups of
  spheres}, vol. 121, Academic press New York, 1986.

\bibitem[Rio12]{riou2012ops}
Jo{\"e}l Riou, \emph{Op$\backslash$'erations de steenrod motiviques}, arXiv
  preprint arXiv:1207.3121 (2012).

\bibitem[Spi12]{spitzweck2012commutative}
Markus Spitzweck, \emph{A commutative $\mathbb{P}^1$-spectrum representing
  motivic cohomology over dedekind domains}, arXiv preprint arXiv:1207.4078
  (2012).

\bibitem[Tot16]{totaro2016motive}
Burt Totaro, \emph{The motive of a classifying space}, Geometry \& Topology
  \textbf{20} (2016), no.~4, 2079--2133.

\bibitem[Voe03a]{voevodsky2003motivic}
Vladimir Voevodsky, \emph{Motivic cohomology with $z/2$-coefficients}, Publ.
  Math. Inst. Hautes {\'E}tudes Sci. (2003), 59{\textendash}104.

\bibitem[Voe03b]{voevodsky2003reduced}
\bysame, \emph{Reduced power operations in motivic cohomology}, Publications
  Math{\'e}matiques de l'Institut des Hautes {\'E}tudes Scientifiques
  \textbf{98} (2003), 1--57.

\bibitem[Voe10]{voevodsky2010motivic}
\bysame, \emph{Motivic eilenberg-maclane spaces}, Publications
  math{\'e}matiques de l'IH{\'E}S \textbf{112} (2010), no.~1, 1--99.

\bibitem[Voe11]{voevodsky2011motivic}
\bysame, \emph{On motivic cohomology with z/l-coefficients}, Ann. of Math. (2)
  \textbf{174} (2011), 401{\textendash}438.

\bibitem[Wil17]{wilson2017power}
Dylan Wilson, \emph{Power operations for hf2 and a cellular construction of
  bpr}.

\end{thebibliography}
